\documentclass[reqno]{amsart}

\usepackage{tikz-cd}

\usepackage{fullpage,dsfont}

\usepackage{hyperref} % this should come before amsrefs

\usepackage{parskip}
\makeatletter % need this to avoid the conflict between amsthm and parskip
\def\thm@space@setup{%
 \thm@preskip=\parskip \thm@postskip=0pt
}
\def\th@remark{%
  \thm@headfont{\itshape}%
  \normalfont % body font
  \thm@preskip\parskip \thm@postskip=0pt
}
\makeatother

\usepackage[nobysame,alphabetic,initials,msc-links]{amsrefs}

\usepackage[final]{pdfpages}

\usepackage{multirow}

\usepackage{array}

\DefineSimpleKey{bib}{how}
\DefineSimpleKey{bib}{mrclass}
\DefineSimpleKey{bib}{mrnumber}
\DefineSimpleKey{bib}{fjournal}
\DefineSimpleKey{bib}{mrreviewer}

\renewcommand{\PrintDOI}[1]{%
  \href{http://dx.doi.org/#1}{{\tt DOI:#1}}%
}
\renewcommand{\eprint}[1]{#1}
\BibSpec{book}{%
    +{}  {\PrintPrimary}                {transition}
    +{.} { \PrintDate}                  {date}
    +{.} { \textit}                     {title}
    +{.} { }                            {part}
    +{:} { \textit}                     {subtitle}
    +{,} { \PrintEdition}               {edition}
    +{}  { \PrintEditorsB}              {editor}
    +{,} { \PrintTranslatorsC}          {translator}
    +{,} { \PrintContributions}         {contribution}
    +{,} { }                            {series}
    +{,} { \voltext}                    {volume}
    +{,} { }                            {publisher}
    +{,} { }                            {organization}
    +{,} { }                            {address}
    +{,} { }                            {status}
    +{,} { \PrintDOI}                   {doi}
    +{,} { \PrintISBNs}                 {isbn}
    +{}  { \parenthesize}               {language}
    +{}  { \PrintTranslation}           {translation}
    +{;} { \PrintReprint}               {reprint}
    +{.} { }                            {note}
    +{.} {}                             {transition}
    +{}  {\SentenceSpace \PrintReviews} {review}
}
\BibSpec{article}{%
    +{}  {\PrintAuthors}                {author}
    +{,} { \textit}                     {title}
    +{.} { }                            {part}
    +{:} { \textit}                     {subtitle}
    +{,} { \PrintContributions}         {contribution}
    +{.} { \PrintPartials}              {partial}
    +{,} { }                            {journal}
    +{}  { \textbf}                     {volume}
    +{}  { \PrintDatePV}                {date}
    +{,} { \issuetext}                  {number}
    +{,} { \eprintpages}                {pages}
    +{,} { }                            {status}
    +{,} { \PrintDOI}                   {doi}
    +{,} { \eprint}        {eprint}
    +{}  { \parenthesize}               {language}
    +{}  { \PrintTranslation}           {translation}
    +{;} { \PrintReprint}               {reprint}
    +{.} { }                            {note}
    +{.} {}                             {transition}
    +{}  {\SentenceSpace \PrintReviews} {review}
}
\BibSpec{collection.article}{%
    +{}  {\PrintAuthors}                {author}
    +{,} { \textit}                     {title}
    +{.} { }                            {part}
    +{:} { \textit}                     {subtitle}
    +{,} { \PrintContributions}         {contribution}
    +{,} { \PrintConference}            {conference}
    +{}  {\PrintBook}                   {book}
    +{,} { }                            {booktitle}
    +{,} { \PrintDateB}                 {date}
    +{,} { pp.~}                        {pages}
    +{,} { }                            {publisher}
    +{,} { }                            {organization}
    +{,} { }                            {address}
    +{,} { }                            {status}
    +{,} { \PrintDOI}                   {doi}
    +{,} { \eprint}        {eprint}
    +{}  { \parenthesize}               {language}
    +{}  { \PrintTranslation}           {translation}
    +{;} { \PrintReprint}               {reprint}
    +{.} { }                            {note}
    +{.} {}                             {transition}
    +{}  {\SentenceSpace \PrintReviews} {review}
}
\BibSpec{misc}{%
  +{}{\PrintAuthors}  {author}
  +{,}{ \textit}      {title}
  +{.}{ }             {how}
  +{}{ \parenthesize} {date}
  +{,} { available at \eprint}        {eprint}
  +{,}{ available at \url}{url}
  +{,}{ }             {note}
  +{.}{}              {transition}
}

\usepackage{amssymb, amsfonts, amsxtra, amsmath}
\usepackage{mathrsfs}
\usepackage{mathdots}
\usepackage{wasysym}
\usepackage[all]{xy}
\usepackage{bbm}
\usepackage{calc}
\usepackage{accents}

\numberwithin{equation}{section}

\newtheorem{Theorem}{Theorem}[section]
\newtheorem*{Theorem*}{Theorem}
\newtheorem{Def}[Theorem]{Definition}
\newtheorem{Lem}[Theorem]{Lemma}
\newtheorem{Prop}[Theorem]{Proposition}
\newtheorem{Cor}[Theorem]{Corollary}
\newtheorem{Not}[Theorem]{Notation}
\newtheorem{Rem}[Theorem]{Remark}

\newcommand\bp{\begin{proof}}
\newcommand\ep{\end{proof}}

\mathchardef\mhyph="2D

\DeclareMathOperator{\diag}{\mathrm{diag}}

\DeclareMathOperator{\Ad}{\mathrm{Ad}}

\DeclareMathOperator{\End}{\mathrm{End}}

\DeclareMathOperator{\loc}{\mathrm{loc}}

\DeclareMathOperator{\Hom}{\mathrm{Hom}}
\DeclareMathOperator{\id}{\mathrm{id}}

\DeclareMathOperator{\Inv}{\mathrm{Inv}}

\DeclareMathOperator{\rd}{\mathrm{d}\!}

\DeclareMathOperator{\Tr}{\mathrm{Tr}}

\DeclareMathOperator{\Irr}{\mathrm{Irr}}
\DeclareMathOperator{\Char}{\mathrm{Char}}

\DeclareMathOperator{\Ker}{\mathrm{Ker}}
\DeclareMathOperator{\Spec}{\mathrm{mSpec}}
\DeclareMathOperator{\Stab}{\mathrm{Stab}}
\DeclareMathOperator{\supp}{\mathrm{Supp}}

\newcommand{\cop}{\mathrm{cop}}

\newcommand{\wt}{\mathrm{wt}}
\newcommand{\full}{\mathrm{full}}
\newcommand{\reg}{\mathrm{reg}}
\newcommand{\red}{\mathrm{r}}
\newcommand{\sred}{\mathrm{rr}}
\newcommand{\sym}{\mathrm{s}}
\newcommand{\ungauge}{\mathrm{ug}}
\newcommand{\gauge}{\mathrm{g}}

\newcommand{\msA}{\mathscr{A}}
\newcommand{\msB}{\mathscr{B}}

\newcommand{\msE}{\mathscr{E}}

\newcommand{\msK}{\mathscr{K}}

\newcommand{\msN}{\mathscr{N}}
\newcommand{\msP}{\mathscr{P}}
\newcommand{\msQ}{\mathscr{Q}}
\newcommand{\msR}{\mathscr{R}}

\newcommand{\wmsR}{\widetilde{\msR}}

\newcommand{\msZ}{\mathscr{Z}}

\newcommand{\mfa}{\mathfrak{a}}

\newcommand{\mfb}{\mathfrak{b}}

\newcommand{\mfg}{\mathfrak{g}}
\newcommand{\wmfg}{\widetilde{\mfg}}
\newcommand{\mfh}{\mathfrak{h}}

\newcommand{\mfn}{\mathfrak{n}}

\newcommand{\mfs}{\mathfrak{s}}
\newcommand{\mfsl}{\mathfrak{sl}}
\newcommand{\mfso}{\mathfrak{so}}
\newcommand{\mfsp}{\mathfrak{sp}}
\newcommand{\mfsu}{\mathfrak{su}}
\newcommand{\mft}{\mathfrak{t}}
\newcommand{\mfu}{\mathfrak{u}}
\newcommand{\wmfu}{\widetilde{\mfu}}

\newcommand{\mcG}{\mathcal{G}}
\newcommand{\mcH}{\mathcal{H}}
\newcommand{\Hsp}{\mathcal{H}}

\newcommand{\mcO}{\mathcal{O}}

\newcommand{\mcS}{\mathcal{S}}

\newcommand{\mbr}{\mathbf{r}}

\newcommand{\mbs}{\mathbf{s}}

\newcommand{\mbH}{\mathbf{H}}

\newcommand{\C}{\mathbb{C}}

\newcommand{\N}{\mathbb{N}}

\newcommand{\R}{\mathbb{R}}
\newcommand{\T}{\mathbb{T}}
\newcommand{\U}{\mathbb{U}}
\newcommand{\X}{\mathbb{X}}

\newcommand{\Z}{\mathbb{Z}}

\newcommand{\opp}{\mathrm{op}}

\newcommand{\inv}{\mathrm{inv}}
\newcommand{\sgn}{\mathrm{sgn}}
\newcommand{\hsgn}{\widehat{\sgn}}

\newcommand\rhdb{\blacktriangleright}

\newcommand{\HC}{\mathrm{HC}}

\newcommand{\br}{\mathrm{br}}

\newcommand{\dbbackslash}{\backslash \! \backslash}
\newcommand{\dbslash}{/\!/}

\title{Invariant quantum measure on $q$-deformed twisted adjoint orbits}

\author{Kenny De Commer}
\address{Vrije Universiteit Brussel\\ Vakgroep Wiskunde en Data Science}
\email{kenny.de.commer@vub.be}

\thanks{The work of K.~De Commer was supported by the FWO grants G025115N and G032919N, and the grant H2020-MSCA-RISE-2015-691246-QUANTUM DYNAMICS. The author thanks the Fields Institute for the nice working environment.}

\begin{document}
\maketitle

\begin{abstract}
Let $U$ be a compact semisimple Lie group with complexification $G$ and associated Cartan involution $\Theta$. Let $\nu$ be an involutive complex Lie group automorphism of $G$ commuting with $\Theta$, and consider the associated semisimple real Lie group $G_{\nu} = \{g\in G\mid \nu(g) = \Theta(g)\}$. We consider $q$-deformed analogues of the $U$-orbits of the quotient space $G_{\nu}\backslash G$, and determine for these the associated von Neumann algebra and invariant state.   
\end{abstract}

\section*{Introduction}

Let $U$ be a connected, simply connected compact Lie group with complexification $G=U_{\C}$. Let $\Theta$ be the Cartan involution of $G$ with fixed points $U$, and put $g^* = \Theta(g)^{-1}$. Let $\nu$ be an involutive complex Lie group automorphism of $G$ commuting with $\Theta$. It gives rise to the real form 
\[
G_{\nu} = \{g\in G \mid \nu(g)^* = g^{-1}\} \subseteq G.
\]
Consider the quotient map
\[
\pi: G_{\nu}\backslash G \rightarrow G_{\nu}\backslash G/U
\]
partitioning $G_{\nu}\backslash G$ into $U$-orbits
\[
O_x^{\nu} = \pi^{-1}(x),\qquad x\in G_{\nu}\backslash G/U
\]
with associated $U$-invariant probability measures $\mu_x$. Note that these can also be seen as twisted adjoint $U$-orbits on the space 
\[
Z_{\nu}=\{g\in G\mid \nu(g)^* = g\},\qquad \Ad_u^{\nu}(g) = \nu(u)^{*}gu
\]
through the equivariant embedding
\begin{equation}\label{EqEquivEmb}
G_{\nu}\backslash G \hookrightarrow Z_{\nu},\qquad G_{\nu}g \mapsto \nu(g)^*g.
\end{equation}

The main goal of this paper is to construct quantum analogues of the $(O_x^{\nu},\mu_x)$, as well as the degenerate limits of these orbits obtained by contraction, within the setting of von Neumann algebraic quantum homogeneous spaces. Note that any simply connected symmetric space of compact type can be realized as $O_{[e]}^{\nu}$ for some $(U,\nu)$, where $[e]=G_{\nu}U$ is the class of the unit $e\in G$ within $G_{\nu}\backslash G /U$. Similarly, any flag manifold for $U$ can be realized as a $U$-orbit at a point at infinity of $U\backslash G$. We mention that the degenerate limits can also be interpreted as orbits of $U$ within a wonderful compactification of $G_{\nu}\backslash G$, cf. \cite{EL01}.

Algebraic aspects of the quantum analogues of the above spaces were studied in \cite{DCM18}, to which we also refer for further intuition and motivation. In this paper, we will take into account some finer \emph{spectral conditions} to handle also analytic aspects. For the sake of the introduction, we explain some of the relevant considerations in an easy special case. 

Let $G = SL(N,\C) \subseteq M_N(\C)$ and let $X \mapsto X^*$ be the usual hermitian adjoint with associated Cartan involution $\Theta(g) = (g^*)^{-1}$. For $1\leq m,n$ with $m+n = N$ define 
\[
\msE=\begin{pmatrix} I_m & 0 \\ 0 & -I_n\end{pmatrix}.
\] 
We obtain on $SL(N,\C)$ the involution
\[
\nu(g) = \Ad_{\msE}(g) = \msE g \msE,
\]
and associated to $\nu$ we have the closed real Lie subgroup
\[
G_{\nu} = \{g\in G\mid \nu(g)^* = g^{-1}\} = \{g\in SL(N,\C)\mid g^*\msE g = \msE \} = SU(m,n) \subseteq SL(N,\C).
\] 

Consider the $*$-algebra $\mcO(G_{\R})$ of regular functions on $G$ as a real algebraic group, so $\mcO(G_{\R})$ is generated by the coordinate functions $x_{ij}(g) = g_{ij}$ as well as their $*$-conjugates $\overline{x}_{ij}(g) = \overline{g_{ij}}$. Inside we have the $*$-subalgebra
\[
\mcO(G_{\R})^{G_{\nu}} = \{f\in \mcO(G_{\R}) \mid \forall g\in G_{\nu},h\in G: f(gh) = f(h)\} \subseteq \mcO(G_{\R}),
\]
with associated maximal real spectrum
\[
G_{\nu}\dbbackslash G = \Spec_*\left(\mcO(G_{\R})^{G_{\nu}}\right) =  \{*\textrm{-homomorphisms }\mcO(G_{\R})^{G_{\nu}}\rightarrow \C\},
\]
where we borrow notation from geometric invariant theory. We have a natural embedding 
\begin{equation}\label{EqEmbEquiDB}
G_{\nu}\backslash G \hookrightarrow G_{\nu}\dbbackslash G,
\end{equation}
but it is not bijective. One has however a concrete model for $G_{\nu}\dbbackslash G$: consider the real affine variety $Z_{\nu}$ introduced above,
\[
Z_{\nu} = \{h\in G\mid \nu(h)^* = h\} = \{h\in SL(N,\C)\mid (\msE h)^* =\msE  h\},\quad \Ad_{g}^{\nu}(h)  = h \cdot g = \nu(g)^{*}hg,\qquad g\in G,h\in Z_{\nu}.
\]
It comes equipped with an associated unital $*$-algebra of regular functions $\mcO(Z_{\nu})$ generated by the coordinate functions
\[
z_{ij}(h) = (\msE h)_{ij},\qquad z_{ij}^* = z_{ji}.
\]
We can $G$-equivariantly identify $G_{\nu}\dbbackslash G \cong Z_{\nu}$ by 
\[
\mcO(Z_{\nu}) \cong \mcO(G_{\R})^{G_{\nu}},\quad z_{ij} \mapsto \sum_{k}  \overline{x}_{ki} \msE_{kk}x_{kj},
\]
and under this isomorphism \eqref{EqEmbEquiDB} becomes the injective $G$-equivariant map \eqref{EqEquivEmb}. 

In the following we simply view $G_{\nu}\backslash G \subseteq Z_{\nu}$ by this map. Then $G_{\nu}\backslash G$ is easily seen to be the connected component of the identity $I_N\in Z_{\nu}$, consisting of those $h$ with $\msE h$ selfadjoint and of the same signature as $\msE$. However, this is not sufficient to have an algebraic description of $G_{\nu}\backslash G$ inside $Z_{\nu}$ as a semialgebraic set determined by polynomial equalities and polynomial inequalities. Instead, we break up $G = UAN = ANU$ where $U = G_{\id} = SU(N) = \{g\in G\mid g^* =g^{-1}\}$ is the associated maximal compact subgroup of $SL(N,\C)$ and $AN \subseteq SL(N,\C)$ is the subgroup of upper triangular matrices with positive diagonal. Consider the $AN$-orbit 
\[
Z_{\nu}^{+,\reg} = I_N \cdot AN \subseteq G_{\nu}\backslash G
\] 
with Euclidian closure $Z_{\nu}^+$. With 
\[
a_k(h) = \det((h_{ij})_{i,j:1\rightarrow k})
\] 
the leading minors of $h$, we can realize $Z_{\nu}^{+,\reg}$ as the semialgebraic set
\[
Z_{\nu}^{+,\reg} = \{h \in Z_{\nu} \mid \forall k: a_k(h) = a_k(\msE h)/a_k(\msE)>0\},
\]
and $G_{\nu}\backslash G$ becomes the union of the $U$-orbits through points of $Z_{\nu}^{+,\reg}$,
\begin{equation}\label{EqGnuR}
G_{\nu}\backslash G = Z_{\nu}^{+,\reg} \cdot U. 
\end{equation}
There is also a natural cell decomposition of $G_{\nu}\backslash G$: consider the natural implementation $W = S_N \subseteq U(N)$ of the Weyl group of $SU(N)$, and let $W_{\nu}^+ = \Stab_{W}(\msE) = S_m\times S_n$. Using the natural extension of the right $G$-action to $GL(N,\C)$, put
\[
Z_{\nu}^{[w],\reg} = I_N \cdot w^{-1}AN,\qquad [w] = wW_{\nu}^+ \in W/W_{\nu}^+,
\]
and let $Z_{\nu}^{[w]}$ be its Euclidian closure. Consider
\[
Z_{\nu}^{\reg} =  \{h \in Z_{\nu} \mid \forall k: a_k(h) \neq 0\},\qquad (G_{\nu}\backslash G)^{\reg} = Z_{\nu}^{\reg}\cap G_{\nu}\backslash G. 
\]
We have that $(G_{\nu}\backslash G)^{\reg}$ is Euclidian dense in $G_{\nu}\backslash G$, with 
\[
G_{\nu}\backslash G = \bigcup_{[w]\in W/W_{\nu}^+} Z_{\nu}^{[w]},\qquad (G_{\nu}\backslash G)^{\reg} = \bigsqcup_{[w]\in W/W_{\nu}^+} Z_{\nu}^{[w],\reg}.
\]

Consider now the $U$-orbits 
\[
O_x^{\nu}\subseteq G_{\nu}\backslash G,\qquad x \in G_{\nu}\backslash G/U.
\] 
For example, $O_{[e]}^{\nu}$ is the symmetric space $K_{\nu}\backslash U$ with $K_{\nu} = U\cap G_{\nu} = S(U(m)\times U(n))$. The cell decomposition of $G_{\nu}\backslash G$ passes to each of these orbits, 
\begin{equation}\label{EqCellDecomp}
O_x^{\nu} = \bigcup_{[w]\in W/W_{\nu}^+} Z_{\nu}^{[w]}\cap O_x^{\nu},\qquad O_x^{\nu,\reg} :=  (G_{\nu}\backslash G)^{\reg} \cap O_x^{\nu}= \bigsqcup_{[w]\in W/W_{\nu}^+} Z_{\nu}^{[w],\reg}\cap O_x^{\nu},
\end{equation}
with $O_x^{\nu,\reg}$ dense in $O_x^{\nu}$. 

The above constructions can be performed for any compact semisimple Lie group $U$, and can then be quantized when $\nu$ is chosen in such a way that it preserves the Borel subgroup $B$ of $G = U_{\C}$ along which the quantization of $U$ is performed \cite{DCM18}. That is, for any $0<q<1$ a $*$-algebra $\mcO_q(Z_{\nu})$ can be constructed, along with a coaction by $\mcO_q(U)$, with $\mcO_q(U)$ the standard deformation of the Hopf $*$-algebra of regular functions on $U$, in such a way that the classical setting is obtained by a suitable limiting procedure $q \rightarrow 1$. A quantization of $\mcO_q(G_{\nu}\backslash G)$ can be obtained by endowing $\mcO_q(Z_{\nu})$ with suitable \emph{spectral conditions}.

It turns out that the $*$-algebra $\mcO_q(G_{\nu}\backslash G/U) = \mcO_q(G_{\nu}\backslash G)^{\mcO_q(U)}$ of $\mcO_q(U)$-coinvariants in $\mcO_q(G_{\nu}\backslash G)$ is central, so also quantized algebras $\mcO_q(O_x^{\nu})$ of associated orbit spaces can be obtained, parametrized by suitable characters $x$ on $\mcO_q(G_{\nu}\backslash G/U)$. Such a $*$-algebra comes equipped with a canonical von Neumann algebraic completion $L^{\infty}_q(O_x^{\nu})$, together with an invariant state $\int_{O_{x,q}^{\nu}}$. Our main result can now be stated qualitatively as follows. We denote $T$ for the chosen maximal torus of $U$, and $T^{\nu}$ for the subgroup of $\nu$-fixed points. We endow $T,T^{\nu}$ and $T/T^{\nu}$ with their normalized Lebesgue measure, denoted generically by $\lambda$. 

\begin{Theorem*}(Theorem \ref{TheoDecompvN} and Theorem \ref{TheoFormInvFunct})  Let $W$ be the Weyl group of $U$. Then there exist subgroups $W_{\nu}^+ \subseteq W_{\nu} \subseteq W$ and Hilbert spaces $\Hsp_{[w]}$ for $[w]\in W_{\nu}/W_{\nu}^+$ such that 
\[
L^{\infty}_q(O_x^{\nu}) \cong \left(\oplus_{[w]\in W_{\nu}/W_{\nu}^+} B(\Hsp_{[w]})\right) \overline{\otimes} L^{\infty}(T/T^{\nu}).
\] 
Moreover, the invariant state $\int_{O_{x,q}^{\nu}}$ is given by 
\[
\int_{O_{x,q}^{\nu}} = \frac{1}{\sum_{[w]\in W_{\nu}/W_{\nu}^+} \Tr(A_{[w]})}\left(\oplus_{[w]\in W_{\nu}/W_{\nu}^+} \Tr(A_{[w]} -)\right)\otimes \int_{T/T^{\nu}} -  \rd \lambda(\theta)
\]
where the $A_{[w]}$ are explicitly given as absolute values of a particular element $a_{\rho} \in \mcO_q(G_{\nu}\backslash G)$  in the associated representation.
\end{Theorem*} 

\begin{Rem}
\begin{enumerate}
\item Our actual theorems allow more general input than an involution $\nu$, so that also degenerate cases associated to general flag varieties can be treated. 
\item In case of the compact Lie group $U \times U$ with involution $\nu(u,v) = (v,u)$, we have an isomorphism $\mcO_q(O_{[e]}^{\nu})\cong \mcO_q(U)$. For this case, the above formula for the invariant state was obtained in \cite{RY01}. We will deal with this specific example in Section \ref{SecSpecExDiag}.
\item In the case of quantized Hermitian symmetric spaces, closely related formulas were obtained in a purely algebraic setting in \cite{Vak90}. However, the situation is slightly different there: only one of the cells appear, only the infinitesimal action of the quantized enveloping algebra is considered, and moreover the relevant $*$-structure of the acting quantized enveloping algebra corresponds to a non-compact real form. This is consistent with the fact that classically, one can indeed realize a non-compact hermitian Riemannian symmetric space as a cell within its compact Cartan dual, and thus as a bounded symmetric domain.
\item The above cell structure on symmetric spaces is closely related to the decomposition into maximal leafs for a natural Poisson structure on symmetric spaces \cite{FL04}. Although the Poisson structure there is constructed by choosing the Borel in Iwasawa-Borel position, so that the base point at the unit is itself a one-point symplectic leaf, one can also construct the Poisson structure directly from a Borel preserved by the involution. Note that in turn, the symplectic leafs for this Poisson structure are closely related to projections of $B$-orbits in $G_{\nu} \backslash G$. The latter are of course in one-to-one correspondence with the $G_{\nu}$-orbits on the flag space $G/B$ \cite{RS93}, whose structure can be examined by considering restrictions of the action to rank one subgroups. The same idea is applied here in the quantum setting, similar to what was done for the quantization of $U$ itself in \cite{LS91}. 
\end{enumerate}
\end{Rem}

The precise contents of this paper are as follows. In the \emph{first section}, we prepare some general theory on spectral conditions for a $*$-algebra in the presence of a compact quantum group action. In the \emph{second section}, we present some preliminaries on twisting data for semisimple Lie algebras and their associated Weyl groups. In the \emph{third section}, we recall some of the theory of \cite{DCM18}, mainly to introduce notation. In the \emph{fourth section}, we obtain a version of the Harish-Chandra isomorphism for quantized enveloping algebras twisted by a twisting datum. In the \emph{fifth section} we develop general results on the representation theory of $\mcO_q(Z_{\nu})$, which is then used in the \emph{sixth section} to obtain our main theorem. In the \emph{seventh section} some generalities on the representation theory of $\mcO_q(Z_{\nu})$ are stated, and some concrete examples are considered such as the diagonal action of $U\times U$ on $U$ as well as the quantization of the example mentioned in the introduction. In an \emph{appendix}, some calculations in low rank are gathered.

\section{Coactions on spectral $*$-algebras}

\subsection{Spectral $^*$-algebras}

In the following, all algebras are assumed unital. 

\begin{Def}
A \emph{spectral $*$-algebra} consists of a $*$-algebra $\msA$ together with a family $\msP$ of bounded $*$-representations of $\msA$ on Hilbert spaces, called the \emph{spectral structure}.
\end{Def}

In the following, we write a spectral $*$-algebra as $\mcO(\X)$. We write $\mcO(\X_{\C})$ for $\mcO(\X)$ as an algebra, forgetting the $*$-structure and the spectral structure. We write 
\[
\mcO^*(\X) = \mcO(\X_*) =  \mcO(\X)/I_{*}\qquad I_{*} = \cap_{\pi \in \msP} \Ker(\pi).
\]
We call $I_*$ the \emph{ideal of non-admissibility}. Note that we can have $I_* = \mcO(\X)$, in which case we write $\X_* = \emptyset$. 

If $\msA$ is a $*$-algebra, the associated \emph{full} spectral $*$-algebra $\mcO(\X_{\msA})$ is $\msA$ together with the family of \emph{all} its bounded $*$-representations on Hilbert spaces (possibly empty). For $\mcO(\X)$ a spectral $*$-algebra we write $\mcO(\X_{\full})$ for the associated full spectral $*$-algebra.

\begin{Rem}
This notion of spectral $*$-algebra is very weak, and will in practice have little value if we do not impose certain properties on $\msP$, such as being closed under arbitrary bounded direct sums. A more refined notion of `closed spectral conditions' was introduced in \cite[Definition 1.1]{DCF19}. The weak version above will however suffice for our purposes. 
\end{Rem}

If $\msA$ is a $*$-algebra and $\pi$ a bounded $*$-representation of $\msA$, we call $\pi$ irreducible if the representation space $\Hsp_{\pi}$ has no non-trivial closed $\msA$-invariant subspaces. For $\pi$ a general bounded $*$-representation of $\msA$, we denote by $\msP_{\pi}^{\Irr}$ the collection of all irreducible $*$-representations $\pi'$ which are weakly contained in $\pi$, meaning that $\pi'$ lifts to a $*$-representation of the C$^*$-algebra $\overline{\pi(\msA)}^{\|-\|}$, the norm-closure of $\pi(\msA)$. In this case we write $\pi' \preccurlyeq \pi$, so that
\[
\msP_{\pi}^{\Irr} = \{\textrm{irreducible }\pi'\mid \pi'\preccurlyeq \pi\}. 
\]

\begin{Def}
Let $\mcO(\X) = (\mcO(\X),\msP)$ be a spectral $*$-algebra. We denote 
\[
\msP^{\Irr} = \cup_{\pi \in \msP} \msP_{\pi}^{\Irr}.
\] 
We call a bounded $*$-representation $\pi$ of $\mcO(\X)$ \emph{admissible} if $\msP_{\pi}^{\Irr} \subseteq \msP^{\Irr}$, and we write 
\[
\msP_{\X} = \langle \msP \rangle = \{\pi \mid \pi \textrm{ admissible bounded }*\textrm{-representation}\}.
\] 
\end{Def}

The following lemma collects some easily verified properties.

\begin{Lem}
Let $\mcO(\X) = (\mcO(\X),\msP)$ be a spectral $*$-algebra. Then
\begin{enumerate}
\item $\msP \subseteq \msP_{\X}$ and $\msP^{\Irr}\subseteq \msP_{\X}$.
\item $(\msP_{\X})^{\Irr} = \msP^{\Irr}$.
\item $\msP_{\X} = \langle \msP_{\X}\rangle = \langle \msP^{\Irr}\rangle$.
\item The ideal of non-admissibility does not change upon passing from $\msP$ to $\msP^{\Irr}$ or $\msP_{\X}$.  
\end{enumerate}
\end{Lem}

\subsection{General theory of compact quantum groups}

We recall some general theory on compact quantum groups, see e.g.~ \cite{NT13}. 

Let $H = (H,\Delta,\varepsilon,S)$ be a Hopf $^*$-algebra with a (necessarily unique) invariant state. We view $H = \mcO(\U)$ as a full spectral $*$-algebra, and interpret $\mcO(\U)$ as the $*$-algebra of regular functions on a `compact quantum group' $\U$. 

We denote the unique invariant state on $\mcO(\U)$ as $\int_{\U}$, and refer to it as the \emph{Haar integral}. We write $L^2(\U)$ for the GNS-completion of $\mcO(\U)$ with respect to $\int_{\U}$, $C_r(\U) \subseteq B(L^2(\U))$ for the associated reduced C$^*$-algebra and $C_u(\U)$ for the associated universal C$^*$-algebra. When $\U$ is coamenable, i.e.~ when the natural $*$-homomorphism $C_u(\U) \rightarrow C_r(\U)$ is an isomorphism, we simply write $C(\U)$ for these C$^*$-algebras. We write $L^{\infty}(\U) = C_r(\U)'' \subseteq B(L^2(\U))$ for the associated von Neumann algebra. 

Recall that $\mcO(\U)$ is spanned linearly by the matrix entries of its unitary corepresentations $U \in B(\Hsp)\otimes \mcO(U)$ on finite-dimensional Hilbert spaces $\Hsp$. Endowing the linear dual $\mcO(\U)'$ with its convolution algebra structure and the $*$-operation 
\[
\omega^*(u) = \overline{\omega(S(u)^*)},
\]
we have a one-to-one correspondence between weak$^*$-continuous $*$-representations $\pi$ of $\mcO(\U)'$ on finite-dimensional Hilbert spaces and unitary corepresentations $U_{\pi}$ of $\mcO(\U)$ via 
\[
\pi(\omega) = (\id\otimes \omega)U_{\pi}.
\]
We denote by $\hat{\varepsilon}$ the trivial representation of $\mcO(\U)'$ with $U_{\hat{\varepsilon}} = 1$. 

For $U_{\pi}$ a unitary corepresentation on $\Hsp_{\pi}$ and $\xi,\eta\in \Hsp_{\pi}$ we write the associated matrix coefficient as
\[
U_{\pi}(\xi,\eta) = (\omega_{\xi,\eta}\otimes \id)U_{\pi},\qquad \omega_{\xi,\eta}(X) = \langle \xi,X\eta\rangle\textrm{ for }X\in B(\Hsp_{\pi}).
\]
Then 
\[
\omega(U_{\pi}(\xi,\eta)) = \langle \xi,\pi(\omega)\eta\rangle,\qquad \omega \in \mcO(\U)'.
\]
The Haar integral on $\mcO(\U)$ is determined by 
\[
\int_{\U} 1 = 1,\qquad \int_{\U} U_{\pi}(\xi,\eta) = 0,\qquad \pi \textrm{ irreducible and } \pi \ncong \hat{\varepsilon}.  
\]
More generally, if $\pi,\pi'$ are inequivalent irreducible unitary representations of $\U$ we have 
\[
\int_{\U} U_{\pi}(\xi,\eta)^*U_{\pi'}(\xi',\eta') = 0,
\]
and there exists a unique character 
\[
\hat{\delta}^{1/2} \in \mcO_q(\U)'
\]
with $\pi(\hat{\delta}^{1/2})$ strictly positive for all $\pi$ and
\begin{equation}\label{EqOrthoGen}
\int_{\U} U_{\pi}(\xi,\eta)^*U_{\pi}(\xi',\eta') =  \frac{\langle \eta,\eta'\rangle \langle \xi',\pi(\hat{\delta}^{1/2})\xi\rangle}{\Tr(\pi(\hat{\delta}^{1/2}))},\qquad \pi \textrm{ irreducible}.
\end{equation}
We call the associated family of functionals 
\[
\hat{\delta}^{iz} \in \mcO(\U)',\qquad \hat{\delta}^{iz}(U_{\pi}(\xi,\eta)) = \langle \xi,\pi(\hat{\delta}^{1/2})^{iz/2}\eta\rangle
\]
the family of \emph{Woronowicz characters}. They are pointwise analytic in $z$, are characters of $\mcO(\U)$ which are $*$-preserving for $z\in \R$, form a complex one-parametergroup 
\[
\hat{\delta}^{iz} \hat{\delta}^{iw} = \hat{\delta}^{i(w+z)}
\]
with respect to convolution, and satisfy 
\[
(\hat{\delta}^{iz})^* = \hat{\delta}^{-i\overline{z}}.
\]

One can express the antipode squared of $\mcO(\U)$ by means of the Woronowicz characters as
\[
S^2(u) = (\hat{\delta}^{-1/2}\otimes \id\otimes \hat{\delta}^{1/2})\Delta^{(2)}(u),\qquad \Delta^{(2)} = (\Delta\otimes \id)\Delta.
\]
We further define the \emph{modular group} $\sigma_z$ as the complex one-parameter group of (non-$^*$-preserving) algebra automorphisms of $\mcO(\U)$ determined by 
\[
\sigma_z(u) = (\hat{\delta}^{-iz/2}\otimes \id\otimes \hat{\delta}^{-iz/2})\Delta^{(2)}(u).
\]
Then $\sigma = \sigma_{-i}$ satisfies the \emph{modularity condition}
\[
\int_{\U} ab = \int_{\U} b\sigma(a).
\]

We extend in the following the convolution product to $*$-representations: if $\pi,\pi'\in \msP_{\U}$, we write
\[
\pi*\pi' = (\pi\otimes \pi')\Delta,\qquad \Hsp_{\pi*\pi'} = \Hsp_{\pi}\otimes \Hsp_{\pi'}. 
\]

\subsection{Actions of compact quantum groups on spectral $*$-algebras}\label{SecAct}

\begin{Def}
Let $\U$ be a compact quantum group, and let $\mcO(\X)$ be a spectral $*$-algebra. A \emph{spectral right coaction} of $\mcO(\U)$ on $\mcO(\X)$ is given by a $*$-preserving right coaction $\alpha$ by $\mcO(\U)$,
\[
\alpha: \mcO(\X) \rightarrow \mcO(\X) \otimes \mcO(\U),
\]
satisfying the \emph{spectrality condition}
\[
\pi \in \msP_{\X},\pi'\in \msP_{\U}\quad \Rightarrow \quad \pi*\pi' \in \msP_{\X},
\]
where
\[
\pi*\pi' = (\pi\otimes \pi')\alpha.
\]
We then call $\X$ a \emph{quantum $\U$-space}. 
\end{Def}
Note that $\alpha$ automatically descends to a coaction on $\mcO^*(\X)$. Also note that if $\mcO(\X)$ is a full spectral $*$-algebra, any right $*$-coaction is automatically spectral. 

If $\alpha$ does not satisfy the spectrality condition, we can enlarge $\msP_{\X}$ to the set
\[
\msP_{\X\U} = \langle \{\pi*\pi' \mid \pi\in \msP_{\X},\pi'\in \msP_{\U}\}\rangle \supseteq \msP_{\X},
\]
and we denote the associated spectral $*$-algebra as $\mcO(\X\U)$. In particular, $\mcO(\X\U)$ has the same underlying $*$-algebra as $\mcO(\X)$, and we can interpret $\X \subseteq \X\U\subseteq \X_{\full}$. We note the following properties. The second property states in a way that any $\U$-orbit of $\X\U$ meets $\X$ and consists of the $\U$-translates of any of its points in $\X$. This nice behaviour comes however at the price of a strong condition on $\U$, which will be satisfied for the particular $\U$ we are interested in.

\begin{Lem}\label{LemTransOrb}
Let $\U$ be a compact quantum group with $C(\U)$ a type $I$ C$^*$-algebra, and let $\mcO(\X)$ be a spectral $*$-algebra with right coaction by $\mcO(\U)$. The following holds:
\begin{enumerate}
\item The coaction $\alpha$ is spectral on $\mcO(\X\U)$.
\item If $\msP_{\X} = \langle \msP\rangle$, then 
\[
\msP_{\X\U} = \langle \{\pi*\pi'\mid \pi \in \msP,\pi'\in \msP_{\U}\}\rangle = \langle \{\pi*\pi_{\reg}\mid \pi \in \msP\}\rangle, 
\]
where $\pi_{\reg}: \mcO(\U) \rightarrow B(L^2(\U))$ is the regular representation. 
\end{enumerate} 
\end{Lem} 
Note that, under the condition of the previous lemma, $\U$ is automatically coamenable by \cite[Proposition 2.5]{CS19}, upon noting that a type $I$ C$^*$-algebra always admits a finite-dimensional $*$-representation, cf.~ \cite{BMT03}.
\begin{proof}
The first property in the lemma follows from the second, since then
\[
\msP_{(\X\U)\U} = \langle \{\pi*\pi'*\pi''\mid \pi \in \msP_{\X},\pi',\pi''\in \msP_{\U}\}= \langle \{\pi*\pi'\mid \pi \in \msP_{\X},\pi' \in \msP_{\U}\}\rangle = \msP_{\X\U}.
\]

In the second property the inclusions $\supseteq$ are obvious. Let us show that 
\[
\msP_{\X\U} \subseteq \langle \{\pi*\pi_{\reg}\mid \pi \in \msP\}\rangle.
\]
Write $B_{\pi}  = \overline{\pi(\mcO(\X))}^{\|-\|}$ and $C_{\pi} = \overline{\pi(\mcO(\U))}^{\|-\|}$ for resp.~ $\pi \in \msP_{\X}$ and $\pi \in \msP_{\U}$. Pick an irreducible $\pi$ with $\pi \preccurlyeq \pi'*\pi''$ for $\pi'\in \msP_{\X}$ and $\pi''\in \msP_{\U}$. We are to show that there exists $\widetilde{\pi} \in \msP$ with  $\pi\preccurlyeq \widetilde{\pi}*\pi_{\reg}$.

Let $\omega$ be a pure state on $B_{\pi'*\pi''}$ associated to $\pi$. Then $\omega$ can be lifted to a pure state on 
\[
\overline{\pi'(\mcO(\X))\otimes \pi''(\mcO(\U))}^{\|-\|} = B_{\pi'}\underset{\min}{\otimes} C_{\pi''}.
\]
Since $C_{\pi''}$ is type $I$, this lift must be a vector state associated to a tensor product $\pi_1\otimes \pi_2$ of irreducible representations of respectively $B_{\pi'}$ and $C_{\pi''}$. It follows that $\pi \preccurlyeq \pi_1*\pi_2$. Since $\pi_1 \preccurlyeq \pi'$, we can pick $\widetilde{\pi} \in \msP$ with $\pi_1 \preccurlyeq \widetilde{\pi}$. Since $\U$ is coamenable, we also have $\pi_2 \preccurlyeq \pi_{\reg}$. But since we have the descent 
\[
B_{\widetilde{\pi}*\pi_{\reg}} \subseteq B_{\widetilde{\pi}}\underset{\min}{\otimes} C_{\pi_{\reg}} \overset{\pi_1 \otimes \pi_2}{\rightarrow} B_{\pi_1}\underset{\min}{\otimes} C_{\pi_2}\supseteq B_{\pi_1*\pi_2},
\]
we must also have $\pi\preccurlyeq \widetilde{\pi}*\pi_{\reg}$.
\end{proof}

\subsection{Quantum homogeneous spaces}

If $\X$ is a quantum $\U$-space, we denote the $*$-algebra of coinvariants as
\[
\mcO(\X)^{\alpha} = \{f\in \mcO(\X) \mid \alpha(f)= f\otimes 1\}.
\]
We write
\begin{equation}\label{EqNotDouble1}
\mcO(\X\dbslash \U) = \mcO(\X)^{\alpha} \textrm{ as a full spectral }*\textrm{-algebra, and}
\end{equation}
\begin{equation}\label{EqNotSingle1}
\mcO(\X/\U) = \mcO(\X)^{\alpha} \textrm{ as a spectral }*\textrm{-algebra for }\msP_{\X/\U} = \langle \{\pi_{\mid \mcO(\X)^{\alpha}} \mid \pi \in \msP_{\X}\}\rangle. 
\end{equation}

\begin{Def}\cite{Boc95}
We call $\X$ a \emph{quantum homogeneous space}, or the action of $\U$ on $\X$ \emph{ergodic}, if $\mcO(\X)^{\alpha} = \C$.
\end{Def}

If $\X$ is a general quantum $\U$-space, note that any $f\in \mcO(\X)$ lies in the linear span of a finite set of elements $f_i\in \mcO(\X)$ which transform according to a unitary corepresentation $U = (u_{ij})_{ij}$ for $\mcO(\U)$. It follows that $\sum_i f_i^*f_i \in \mcO(\X)^{\alpha}$. Assuming that $\X$ is a quantum homogeneous space, each $\sum_i f_i^*f_i$ as above is scalar. In particular, any $*$-representation $\pi$ of $\mcO(\X)$ on a (pre-)Hilbert space $\Hsp$ is bounded, and there exists for each $f\in \mcO(\X)$ a $\pi$-independent constant $C_f$ such that $\|\pi(f)\|\leq C_f$ for all such $\pi$. Hence we can associate to $\mcO(\X)$ the universal C$^*$-algebra $C_u(\X)$ obtained from the seminorm 
\[
\|x\| = \sup\{\|\pi(x)\| \mid \pi \in \msP_{\X}\}.
\]
We have 
\[
\mcO(\X) \twoheadrightarrow \mcO^*(\X) \hookrightarrow C_u(\X).
\] 
and the inclusion above has dense image.

For $\X$ a quantum homogeneous space, we write $\int_{\X}$ for the associated normalized invariant functional on $\mcO(\X)$, determined by 
\[
\int_{\X}f = \left(\id\otimes \int_{\U}\right)\alpha(f).
\] 
This functional descends to a faithful state on $\mcO^*(\X)$. We write $L^2(\X)$ for the GNS-space of $\mcO^*(\X)$ with respect to $\int_{\X}$, $C_r(\X)$ for the reduced C$^*$-algebra, and $L^{\infty}(\X)$ for the associated von Neumann algebra.  These come equipped with respective coactions of $C_r(\U)$ and $L^{\infty}(\U)$. When $\U$ is coamenable, the natural projection map $C_u(\X) \rightarrow C_r(\X)$ will be an isomorphism, and we simply write $C(\X)$ in this case. 

When $\U$ is coamenable, the spectral structure of $\mcO(\X)$ will necessarily be the full one. 

\begin{Lem}\label{LemErgoIsFull}
Let $\alpha: \mcO(\X) \rightarrow \mcO(\X)\otimes \mcO(\U)$ be a spectral ergodic coaction, with $\U$ coamenable. Then $\mcO(\X)$ is a full spectral $*$-algebra. 
\end{Lem}
\begin{proof}
By coamenability, any irreducible $*$-representation of $\mcO(\X_{\full})$ is weakly contained in the regular representation $\pi_{\X,\reg}$ of $\mcO(\X)$ on $L^2(\X)$. It hence suffices to show that the latter is admissible for $\mcO(\X)$. If however $\pi\in \msP_{\X}$ is arbitrary, we have that $\pi_{\X,\reg} \subseteq \pi*\pi_{\U,\reg}$, and the latter is admissible by spectrality of the coaction.
\end{proof}

Note that, for $\X$ homogeneous, $\mcO^*(\X)$ also admits a unique modular automorphism $\sigma: \mcO^*(\X) \rightarrow \mcO^*(\X)$ such that 
\[
\int_{\X} xy = \int_{\X} y \sigma(x),\qquad x,y\in \mcO^*(\X),
\]
see \cite[Proposition 2.10]{BDRV06}. The automorphism $\sigma$ extends uniquely to a complex analytic one-parameter group of diagonalizable automorphisms $\sigma_z: \mcO^*(\X) \rightarrow \mcO^*(\X)$ with $\sigma = \sigma_{-i}$. We have
\begin{equation}\label{EqAlphaSigma}
\alpha \circ \sigma = (\sigma \otimes S^2)\circ \alpha. 
\end{equation}

\subsection{Centrally coinvariant coactions and quantum orbits}

In this paper, we will be interested in a slightly more general class of actions than just the ergodic ones. We use in the following definition the notation introduced in \eqref{EqNotDouble1} and \eqref{EqNotSingle1}.

\begin{Def}
Let $\mcO(\X)$ be a spectral $*$-algebra with a spectral coaction by $\mcO(\U)$. We say that $\mcO(\X)$ has \emph{central coinvariants} if $\mcO(\X)^{\alpha} \subseteq \msZ(\mcO(\X))$, the center of $\mcO(\X)$. 

We write in this case 
\[
\X\dbslash \U = \Spec_*(\mcO(\X)^{\alpha}) =  \{*\textrm{-characters of }\mcO(\X)^{\alpha}\},
\]
\[
\X/\U= \{\textrm{admissible }*\textrm{-characters of }\mcO(\X/ \U)\} \subseteq \X\dbslash \U.
\] 
\end{Def}

We view $\X\dbslash \U$ and $\X/\U$ as topological spaces by the topology of pointwise convergence on $\mcO(\X)^{\alpha}$. 

\begin{Def}
Let $x\in \X\dbslash \U$. We call a bounded $*$-representation $\pi$ of $\mcO(\X_{\full})$ an \emph{$x$-representation} if $\pi$ restricts to $x$ on $\mcO(\X \dbslash \U)$. 
\end{Def}

By definition we have $x\in \X/\U$ if and only if there exists an admissible $x$-representation $\pi$ of $\mcO(\X)$. By elementary C$^*$-theory, we can also always choose this $\pi$ to be irreducible. 

For $x \in  \X\dbslash\U$ we write $\widetilde{I}_x$ for the 2-sided $*$-ideal
\[
\widetilde{I}_x =  \mcO(\X)\{f - x(f)\mid f\in \mcO(\X)^{\alpha}\} \subseteq \mcO(\X).
\]
By central coinvariance $\alpha$ descends to an ergodic spectral coaction on 
\[
\mcO(\X(x)) = \mcO(\X)/\widetilde{I}_{x},
\] 
and we call $\X(x)$ the \emph{fiber} of $\X$ at $x$. Here we interpret $\mcO(\X(x))$ as a spectral $*$-algebra by 
\[
\msP_{\X(x)} = \{\pi \in \msP_{\X}\mid \pi \textrm{ is an }x\textrm{-representation}\}.
\] 
We know from Lemma \ref{LemErgoIsFull} that this is simply the full spectral $*$-algebra structure if $\U$ is coamenable. It follows that in this case $\X/\U$ consists of those $x\in \X\dbslash \U$ for which the ideal of non-admissibility for $\mcO(\X(x))$ is not the whole algebra. We then write for $x \in  \X/\U$
\[
\mcO(\X_*(x)) = \mcO^*(\X(x)),
\]
and we call $\X_*(x)$ the fiber of $\X_*$ at $x$. 

Given an arbitrary coaction $\alpha$, let $E_{\alpha}$ be the projection map 
\[
E_{\alpha}: \mcO(\X)\rightarrow \mcO(\X)^{\alpha},\quad f \mapsto \left(\id\otimes \int_{\U}\right)\alpha(f). 
\]
The following lemma gives an easy characterisation of $\X_{\full}/\U$.

\begin{Lem}
Let $\mcO(\X)$ be a full spectral $*$-algebra with a centrally coinvariant $*$-coaction by $\mcO(\U)$. Let $x\in \X\dbslash \U$. Then $x\in \X/\U$ if and only if the invariant functional 
\[
\int_{\X(x)} = x\circ E_{\alpha}
\] 
on $\mcO(\X)$ is a positive functional, i.e. 
\[
\int_{\X(x)}f^*f\geq 0\qquad  \textrm{for all }f\in \mcO(\X).
\]
\end{Lem}
\begin{proof}
If $x \in \X_{\full}/\U$, take any $x$-representation $\pi$ of $\mcO(\X)$, and take any non-zero vector $\xi\in \mcH_{\pi}$ with associated vector state $\omega_{\xi}$. Then we have 
\[
\int_{\X(x)}f^*f = \left((\omega_{\xi}\circ \pi)\otimes \int_{\U}\right)\alpha(f^*f)\geq 0,\qquad f\in \mcO(\X).
\]
Conversely, if $\int_{\X(x)}$ is positive, let $\pi_x$ be the associated GNS-representation. The same argument which showed the existence of the universal C$^*$-envelope shows that $\pi_x$ is bounded, and thus an $x$-representation of $\mcO(\X)$. 
\end{proof}

In particular, we see that $\X/\U$ is a closed subset of $\X\dbslash \U$. Note that if we start with a general spectral coaction with central coinvariants, we then have the (possibly all strict) inclusions 
\[
\X/\U \subseteq \X_{\full}/\U \subseteq \X\dbslash \U.
\]

\subsection{Two examples}\label{ExaRank1}
Most of the following material can be found, as far as the algebraic constructions are concerned, in standard reference works on quantum groups, see e.g.~ \cite{Maj95,KS97}.

Consider for $0<q<1$ the $R$-matrix
\[
R= q^{-1/2}(e_{11}\otimes e_{11} + e_{22}\otimes e_{22})+q^{1/2}(e_{11}\otimes e_{22}+ e_{22}\otimes e_{11}) + q^{1/2}(q^{-1}-q)e_{12}\otimes e_{21} \in M_2(\C)\otimes M_2(\C).
\]
Let $\mcO_q(M_2(\C))$ be the algebra generated by a matrix $M = \begin{pmatrix} a & b \\ c & d\end{pmatrix} \in M_2(\C)\otimes \mcO_q(M_2(\C))$ of indeterminates such that
\[
R_{12}M_{13}M_{23} = M_{23}M_{13}R_{12}. 
\]
By direct verification one sees that $\mcO_q(M_2(\C))$ has universal relations
\[
ac = qca,\quad ab =q ba,\quad bc = cb,\quad cd = qdc,\quad bd = qdb,\quad ad-da = (q-q^{-1})bc. 
\]
It is a bialgebra in a natural way such that $M$ becomes a corepresentation. It becomes a $*$-bialgebra by the $*$-structure 
\[
a^* = d,\qquad b^* = -qc.
\] 
With $V = \C^2$ interpreted as row vectors, we write this $*$-algebra as $\mcO_q(V_{\R}) = \mcO(V_{\R,q})$, as we view it as a quantization of the $*$-algebra of regular functions on $V$ as a real affine variety, with $a,b$ corresponding to the (holomorphic) coordinate functions and $d,c$ to their (anti-holomorphic) complex conjugates. We view here $\mcO_q(V_{\R})$ as a full spectral $*$-algebra. 

Let 
\[
D = ad - qbc = da-q^{-1}bc.
\] 
Then $D$ is a selfadjoint grouplike central element, and we write $\mcO_q(SU(2)) = \mcO(SU_q(2))$ for the Hopf $^*$-algebra obtained by imposing the extra relation $D=1$. We write in this case the fundamental corepresentation $M$ as $U$. 

Clearly the comultiplication gives rise to a coaction 
\[
\alpha: \mcO_q(V_{\R}) \rightarrow \mcO_q(V_{\R}) \otimes \mcO_q(SU(2)),\qquad (\id\otimes \alpha)M = M_{12}U_{13}.
\]
The following proposition is easily verified. 

\begin{Prop}
The coaction on $\mcO_q(V_{\R})$ is centrally coinvariant, with
\[
\mcO_q(V_{\R})^{\alpha}= \C[D].
\] 
Identifying $\Spec_*(\C[D])\cong \R$ via $\chi \mapsto \chi(D)$, we have
\[
V_{\R,q}\dbslash SU_q(2) = \R,\qquad V_{\R,q}/SU_q(2) = \R_{\geq 0}.
\] 

For $\lambda = 0$ we have $\mcO_q^*(V_{\R}(0)) = \C$ with the trivial coaction.

For $\lambda> 0$ we have $\mcO_q^*(V_{\R}(\lambda))\cong \mcO_q(SU(2))$ equivariantly via $M\mapsto \lambda^{1/2} U$.
\end{Prop}

Also the following theorem can be easily derived by hand, or deduced from the results in \cite{LS91, Koo91}. 
\begin{Theorem} \label{TheoListAllSUq2}
A full list of inequivalent irreducible $*$-representations of $\mcO_q^*(V_{\R}(\lambda))$ for $\lambda>0$ is given by the $*$-characters
\[
\chi_{\lambda,\theta}\begin{pmatrix} a & b \\ c & d \end{pmatrix} =  \begin{pmatrix} \lambda^{1/2}e^{i\theta} & 0 \\ 0 & \lambda^{1/2}e^{-i\theta}\end{pmatrix},\qquad \theta \in \T =  \R/2\pi\Z, 
\]
and the infinite-dimensional $*$-representations
\[
\mcS_{\lambda,\theta} = l^2(\N),\quad a e_n = \lambda^{1/2}(1-q^{2n})^{1/2}e_{n-1},\quad b e_n =\lambda^{1/2}e^{-i\theta} q^n e_n,\qquad \theta \in  \T.
\]
When $\theta =0$ we drop the symbol $\theta$, and when interpreting $\mcO_q(V_{\R}(1))\cong \mcO_q(SU(2))$ we drop the symbol $\lambda$. We then have
\[
\chi_{\lambda}*\chi_{\theta} = \chi_{\lambda,\theta},\quad \mcS_{\lambda} * \chi_{\theta} = \mcS_{\lambda,\theta},\qquad \chi_{\lambda}*\mcS = \mcS_{\lambda}
\]
and 
\[
\mcS_{\lambda}*\mcS \cong \int_{\T} \mcS_{\lambda,\theta} \rd \theta. 
\]
\end{Theorem}
We also note that the Haar integral of $SU_q(2)$ can be realized with respect to the $*$-representation
\[
\int_{\T} \mcS*\chi_{\theta}\rd \theta  = l^2(\N)\otimes l^2(\Z),\quad  a e_{n,m} =(1-q^{2n})^{1/2}e_{n-1,m},\quad b e_{n,m} =q^n e_{n,m-1},
\]
since then 
\[
\int_{SU_q(2)} x = (1-q^2)\sum_{n} \langle be_{n0},xbe_{n0}\rangle, 
\]
see \cite[Appendix A.1]{Wor87}. It follows easily from this that 
\[
L^{\infty}(SU_q(2)) \cong B(l^2(\N))\overline{\otimes} L^{\infty}(\T).
\]

Let us now consider a second example. Let $\mcO_q^{\br}(M_2(\C))$ be the algebra of \emph{braided} 2-by-2 matrices \cite[Definition 4.3.1]{Maj95}, i.e. the universal algebra generated by a matrix $Z = \begin{pmatrix} z & w \\ v & u\end{pmatrix}$ with relations
\[
R_{21}Z_{13}R_{12}Z_{23} = Z_{23}R_{21}Z_{13}R_{12}. 
\]
It has the $*$-structure $Z^* = Z$. As a full spectral $*$-algebra, we interpret it as $\mcO_q(H_2)= \mcO(H_{2,q})$ with $H_2$ the space of hermitian $2$-by-$2$ matrices. It carries a right coaction by $\mcO_q(SU(2))$ given as
\begin{equation}\label{EqEquivZ}
Z \mapsto U_{13}^*Z_{12}U_{13}. 
\end{equation}
Write 
\[
T = q^{-1}u + qz,\qquad D = uz -q^{-2}vw,
\]
which are called respectively the \emph{quantum trace} and \emph{quantum determinant}. We easily find by direct computation that the universal relations of $\mcO_q(H_2)$ are centrality and selfadjointness of $T,D$ together with the relations 
\[
z^* = z,\quad w^* = v,\quad zw = q^2wz,\qquad vz = q^2zv,
\]
and
\[
q^{-2}vw = -D+qTz - q^{2}z^2,\qquad q^{-2}wv = -D+q^{-1}Tz -q^{-2}z^2.
\]

The main results in the following proposition can be obtained from \cite{Pod87,MN90,MNW91}. The computation of the $*$-products of representations can be done by spectral analysis along the lines of e.g.~ \cite[Appendix B]{DeC11}.
\begin{Prop}\label{PropRank1Tensor}
The coaction on $\mcO_q(H_2)$ is centrally coinvariant with $\mcO_q(H_2)^{\alpha}= \C[D,T]$. Identifying $\Spec_*(\C[D,T])\cong \R^2$ via $\chi \mapsto (\chi(D),\chi(T)) = (d,t)$, we have that 
\[
H_{2,q}\dbslash SU_q(2) = \R^2,\qquad H_{2,q}/SU_q(2) = S_+ \cup S_0 \cup S_-
\] 
where, writing $s_n = q^{-n-1}+q^{n+1}$ for $n\in \N$,
\[
S_+ = \{(c^2, cs_n)\mid c\in \R^{\times}, n\in \N\},\qquad  S_0 = \{0\}\times \R,\qquad S_- = \R_{<0}\times \R.
\] 

Moreover, we have the following fusion rules, denoting by $\mcH$ a Hilbert space of countably infinite dimension encoding countable multiplicity:
\begin{enumerate}
\item[(1)] 
For $(c^2,cs_n)\in S_+$ we have $\mcO_q^*(H_2(c^2,cs_n))\cong M_{n+1}(\C)$ with unique irreducible $*$-representation
\[
\mcS_{c^2,n}\cong \C^{n+1},\qquad ze_k= cq^{-n+2k}e_k,\quad ve_0 = 0.
\]
Convolution with $\mcO_q(SU(2))$-representations is given by
\[
\mcS_{c^2,n}* \chi_{\theta}\cong \mcS_{c^2,n},\qquad \mcS_{c^2,n}*\mcS \cong \Hsp \otimes \mcS_{c^2,n}.
\]
The invariant state on $\mcO_q^*(H_2(c^2,cs_n))\cong M_{n+1}(\C)$ is given by 
\[
\int_{H_{2,q}(c^2,cs_n)} X = \frac{\Tr(zX)}{\Tr(z)}.
\]
\item[(2)]
For $(0,t) \in S_0$ we have the $*$-character $\chi(z) = \chi(v) = \chi(w) =0$. When $t=0$ this is the only irreducible $*$-representation, while if $t\neq 0$ the only other irreducible $*$-representation is given by 
\[
\mcS_{0,t} \cong l^2(\N),\quad ze_n = q^{2n+1}te_n,\quad ve_0 = 0.
\] 
Convolution with $\mcO_q(SU(2))$-representations is given by
\[
\chi* \chi_{\theta}\cong \chi\qquad \chi*\mcS \cong \mcS_{0,t},\qquad \mcS_{0,t}*\chi_{\theta}\cong \mcS_{0,t},\qquad \mcS_{0,t}*\mcS  \cong \Hsp \otimes \mcS_{0,t}.
\]
For $t\neq 0$ the associated von Neumann algebra can be identified as 
\[
L^{\infty}_q(H_2(0,t)) = B(\mcS_{0,t}),
\]
and under this isomorphism the invariant state on $\mcO_q^*(H_2(0,t))$ is given as 
\[
\int_{H_{2,q}(0,t)} X = \frac{\Tr(zX)}{\Tr(z)}.
\]
\item[(3)] 
For $(-c^2,t) \in S_-$ with $c>0$ and $t = (a-a^{-1})c$ for $a>0$ we have as complete list of non-equivalent irreducibles the $*$-characters $\chi_{\theta}$ with $\chi_{\theta}(z)= 0$ and $\chi_{\theta}(v) = qce^{i\theta}$ and the $*$-representations
\[
\mcS^{\pm}_{-c^2,a} \cong l^2(\N),\qquad ze_n = \pm ca^{\pm 1} q^{2n+1}e_n,\qquad ve_0 =0.
\]
Writing $\mcS_{-c^2,a} = \mcS_{-c^2,a}^+ \oplus \mcS_{-c^2,a}^-$, one has convolution with $\mcO_q(SU(2))$-representations given by
\[
\chi_{\theta}*\chi_{\theta'} = \chi_{\theta+2\theta'},\quad \chi_{\theta}* \mcS \cong \mcS_{-c^2,a},\qquad \mcS^{\pm}_{-c^2,a}*\chi_{\theta}\cong \mcS^{\pm}_{-c^2,a},\qquad \mcS^{\pm}_{-c^2,a}*\mcS \cong \Hsp \otimes \mcS_{-c^2,a}.
\]
The associated von Neumann algebra can be identified as 
\[
L^{\infty}_q(H_2(-c^2,t)) = B(\mcS_{-c^2,a}^+)\oplus B(\mcS_{-c^2,a}^-),
\] 
and under this isomorphism the invariant state on $\mcO_q^*(H_2(-c^2,t))$ is given (in obvious notation) as 
\[
\int_{H_{2,q}(-c^2,t)} X = \frac{\Tr_+(zX) - \Tr_-(zX)}{\Tr_+(z)-\Tr_-(z)}.
\]
\end{enumerate}
\end{Prop}
The main goal of this paper will be to find higher rank analogues of the above formulas for the `quantum invariant measures' on the $V_{\R,q}(\lambda)$ and $H_{2,q}(d,t)$.

\section{Twisting data}\label{SecDefDat}

In this section we fix an index set $I$ of cardinality $|I|$. We further fix a Euclidian space $V \cong \R^{|I|}$ with inner product $(-,-)$ and a reduced finite root system $\Delta\subseteq V$ which spans $V$. We choose a system of positive roots $\Delta^+$ and associated simple roots $\Phi= \{\alpha_r \mid r\in I\}\subseteq \Delta^+$ with indexation by $I$. We denote by $\Gamma$ the associated Dynkin diagram on $I$, with arrows pointing towards the shorter roots. We write $\alpha^{\vee} = \frac{2}{(\alpha,\alpha)}\alpha$ for the coroots, $\varpi_r\in V$ for the fundamental weights such that $(\varpi_r,\alpha_s^{\vee}) = \delta_{rs}$, and $\varpi_r^{\vee}$ for the fundamental coweights such that $(\varpi_r^{\vee},\alpha_s) = \delta_{rs}$. We write the Cartan matrix
\[
A = (a_{rs})_{r,s} = ((\alpha_r^{\vee},\alpha_s))_{r,s}.
\]
We write $Q^+ =  \N\{\alpha_r\mid r\in I\}$ and $Q = Q^+ -Q^+$ for resp.~ the cone of positive roots and the root lattice, and similarly $P^+ = \N\{\varpi_r\mid r \in I\}$ and $P = P^+ - P^+$ for the cone of dominant integral weights and the weight lattice.

For $J \subseteq I$ a subset, we write $\Gamma_J$ for the subdiagram of $\Gamma$ with vertex set $J$ and associated Cartan matrix $A_J = (a_{rs})_{r,s\in J}$. We call $\Gamma_J$ a \emph{full Dynkin subdiagram}. We write $\Delta_J = \Delta \cap \Z\{\alpha_r \mid r\in J\}$ for the associated root system.

\subsection{Twisting data}

\begin{Not}
For $\tau$ an involution of $I$, we write 
\begin{itemize}
\item $I^{\tau}$ for the $\tau$-fixed vertices,
\item $I^*$ for a fixed fundamental domain of $\tau$ inside $I \setminus I^{\tau}$.
\end{itemize}   
\end{Not}

Assume now moreover that $\tau$ is a Dynkin diagram involution, allowing $\tau= \id$. Then we can extend $\tau$ uniquely to an isometric involution on $V$ such that
\[
\tau(\alpha_r) = \alpha_{\tau(r)}. 
\]
This extension preserves the root system and its positive roots,
\[
\tau(\Delta) = \Delta,\quad \tau(\Delta^+) = \Delta^+.
\]

\begin{Def}
We call $\Gamma$ connected if its root system is irreducible.

We call $\Gamma$ $\tau$\emph{-connected} if $I = J \cup \tau(J)$ with $\Gamma_J$ connected.

We call a subdiagram $\Gamma'  \subseteq \Gamma$ a \emph{$\tau$-connected component} if $\Gamma'$ is a maximal $\tau$-connected full subdiagram.
\end{Def}

Any $\Gamma$ is obviously the disjoint union of its $\tau$-connected components. 
\begin{Rem}
Note that $\Gamma$ is $\tau$-connected if and only if it is connected or it is the disjoint union of two isomorphic Dynkin diagrams which are switched under $\tau$. We call the latter $(\Gamma,\tau)$ of \emph{direct product} type. 
\end{Rem}

The following notion will play an important rôle in the paper. 
\begin{Def}
We call a couple $\nu = (\tau,\epsilon) = (\tau_{\nu},\epsilon_{\nu})$ a \emph{twisting datum} if $\tau$ is an involution of the Dynkin diagram $\Gamma$ and $\epsilon$ is a $\C$-valued function on $I$ with 
\begin{equation}\label{EqDefPropEps}
\epsilon_{\tau(r)} = \overline{\epsilon_r},\qquad \textrm{for all }r\in I.
\end{equation}  
We write 
\[
\mbH = \{\nu\}\subseteq \Inv(\Gamma)\times \C^{|I|}
\]
for the space of all twisting data.

We write
\[
J_{\nu} = J_{\epsilon} =  \supp(\epsilon) = \{r\in I \mid \epsilon_r \neq 0\}
\]
for the \emph{locus of non-degeneracy}. Note that $J_{\nu}$ is $\tau$-invariant.
\end{Def}

We will need various extra conditions on twisting data, depending on the context. We gather these conditions in the following definition. 

\begin{Def}
We call a twisting datum $\nu$ 
\begin{itemize}
\item \emph{regular} if $J_{\nu}=I$,
\item \emph{positive} if $\nu$ is regular and $\epsilon_r>0$ for all $r\in I^{\tau}$, 
\item \emph{strongly positive} if $\epsilon_r>0$ for all $r\in I$,
\item of \emph{symmetric pair type} if $|\epsilon_r| = 1$ for all $r\in I$.
\item a \emph{gauge} if $\epsilon_r = 1$ for $r\in I^{\tau}$ and $|\epsilon_r|=1$ for $r\in I^*$,
\item \emph{ungauged} if $\epsilon_r\geq 0$ for $r\in I^*$,  
\item \emph{reduced} if $\epsilon_r \in \{-1,0,1\}$ for $r\in I^{\tau}$ and $\epsilon_r \in \{0,1\}$ for $r\in I^*$,
\item \emph{strongly reduced} if $\nu$ is reduced and each $\tau$-connected component of $J_{\nu}$ has at most one $r$ with $\tau(r) = r$ and $\epsilon_r= -1$.
\end{itemize}
We write the respective spaces as 
\begin{equation}\label{EqDiffdefdat}
\mbH^{\times},\quad \mbH^{>},\quad \mbH^{\gg},\quad \mbH^{\sym},\quad \mbH^{\gauge},\quad \mbH^{\ungauge},\quad \mbH^{\red},\quad \mbH^{\sred}.
\end{equation}

The \emph{trivial twisting datum} $+$ is given by $\tau = \id$ and $\epsilon_r = 1$ for all $r$. 
\end{Def}

\begin{Rem} The notion of twisting datum is very similar to that of `extended $\tau$-signature' \cite{EL01}, but we allow more flexibility. 
\end{Rem}

We identify the space of all twisting data with underlying involution $\tau$ as 
\[
\mbH_{\tau} = \{\epsilon\} \subseteq \C^{|I|},
\]
and we index in this case all notations in \eqref{EqDiffdefdat} with $\tau$. Similarly, if we fix also a subset $J \subseteq I$ as the locus of non-degeneracy, we index all notations in \eqref{EqDiffdefdat} with $J$. 

The space $\mbH_{\tau}$ forms a monoid under pointwise multiplication, with group of invertibles $\mbH^{\times}_{\tau}$.

\begin{Def}
We call two twisting data $\nu,\nu'$ \emph{multiplicatively equivalent}, and write $\nu \sim_m \nu'$, if $\tau = \tau'$ and if $\epsilon$ and $\epsilon'$ lie in the same orbit under multiplication with $\mbH_{\tau}^{>}$. 
\end{Def}
We stress that under this equivalence the values at $I^{\tau}$ can only be multiplied with strictly positive numbers, whereas the values at $I^*\cup \tau(I^*)$ can be multiplied with non-zero conjugate complex numbers. Clearly each twisting datum is multiplicatively equivalent to a unique reduced twisting datum.

If $\nu = (\tau,\epsilon) \in \mbH_{J}$, we can extend $\epsilon$ to monoid homomorphisms
\begin{equation}\label{EqEpsQ}
\epsilon_Q: (Q^{J+},+) \rightarrow (\C,\cdot),\quad \epsilon_Q(\alpha_r) = \epsilon_r,
\end{equation}
\begin{equation}\label{EqEpsP}
\epsilon_P: (P^{J+},+) \rightarrow (\C,\cdot),\quad \epsilon_P(\varpi_r) = \epsilon_r,
\end{equation}
where 
\[
Q^{J+} = \Z\{\alpha_r\mid r\in J\}+ \N\{\alpha_r\mid r\notin J\},\qquad P^{J+} = \Z\{\varpi_r\mid r\in J\}+ \N\{\varpi_r\mid r\notin J\}.
\]
We then have in general that $\epsilon_Q(\tau(\alpha))  = \overline{\epsilon_Q(\alpha)}$ and $\epsilon_P(\tau(\varpi)) = \overline{\epsilon_P(\varpi)}$ for $\alpha \in Q^{J+}$ and $\varpi \in P^{J+}$.  

\subsection{Folding}

Let $\tau$ be an involution of the Dynkin diagram $\Gamma$. We write 
\[
V \rightarrow V^{\tau},\quad v \mapsto v_{+} = \frac{1}{2}(v + \tau(v)),\qquad V \rightarrow V^{-\tau},\quad v \mapsto v_{-} = \frac{1}{2}(v - \tau(v))
\]
for the projection maps onto resp.~ the $\tau$-fixed and the $-\tau$-fixed points. 

\begin{Prop}\cite[Theorem 32]{Ste67}
The subset 
\[
\hat{\Delta} = \{\alpha_+\mid \alpha \in \Delta\}\subseteq V^{\tau}
\]
is a (possibly non-reduced) root system with associated positive roots $\{\alpha_+\mid \alpha \in \Delta^+\}$ and associated system of simple roots $\hat{\Phi} =\{\alpha_{r,+} \mid r\in I\}$.
\end{Prop} 
\begin{Rem}
When $\Gamma$ is $\tau$-connected, the case of non-reduced $\hat{\Delta}$ only appears for the Dynkin diagram $\begin{tikzpicture}[scale=.4,baseline]
\node (v1) at (0,0) {};
\node (v2) at (1.8,0) {};
\node (v3) at (3.2,0) {};
\node (v4) at (5,0) {};
\draw(0cm,0) circle (.2cm);
\draw (0.2cm,0) -- +(0.2cm,0);
\draw[dotted] (0.4cm,0) --+ (1cm,0);
\draw (1.6cm,0) --+ (0.2cm,0);
\draw (2cm,0) circle (.2cm);
\draw (2.2cm,0) -- +(0.6cm,0);
\draw(3cm,0) circle (.2cm);
\draw (3.2cm,0) -- +(0.2cm,0);
\draw[dotted] (3.4cm,0) --+ (1cm,0);
\draw (4.6cm,0) --+ (0.2cm,0);
\draw (5cm,0) circle (.2cm); 
\draw[<->]
(v1) edge[bend left] (v4);
\draw[<->]
(v2) edge[bend left=70] (v3);
\end{tikzpicture}$
of type $A_{2n}$ with non-trivial involution. 
\end{Rem}

In the following, we denote 
\[
\hat{I} = I/\tau,\qquad \hat{r} = \{r,\tau(r)\}.
\]
The simple roots in $\hat{\Phi}$ can be labeled by $\hat{I}$ upon putting
\[
\alpha_{\hat{r}} = \alpha_{r,+} = \alpha_{\tau(r),+}. 
\]
We call a simple root $\alpha_{\hat{r}}$ a \emph{folded simple root} if $|\hat{r}|= 2$.

We define the associated \emph{folded Dynkin diagram} $\hat{\Gamma} = \Gamma/\tau$ via the folded Cartan matrix 
\[
\hat{A}_{\hat{r}\hat{s}} = 2\frac{(\alpha_{\hat{r}},\alpha_{\hat{s}})}{(\alpha_{\hat{r}},\alpha_{\hat{r}})}. 
\]  

We denote by $\varpi_{\hat{r}} \in V^{\tau}$ the fundamental weights for $\hat{\Delta}$. We have the general formula
\begin{equation}\label{EqHatVarpi}
\varpi_{\hat{r}} = \sum_{s\in \hat{r}} \lambda_s \varpi_s,\qquad \lambda_r = \frac{1}{2} + \frac{1}{4}(\alpha_r^{\vee},\alpha_{\tau(r)}).
\end{equation}
Hence $\varpi_{\hat{r}} = \varpi_{r,+}$ \emph{except} for the middle two vertices in type $A_{2n}$ with non-trivial $\tau$, in which case $\varpi_{\hat{r}} = \frac{1}{2}\varpi_{r,+}$.

We list below the connected Dynkin diagrams with non-trivial involution, together with their folded version. 

\begin{table}[htpb!]
\begin{center}
\bgroup
\def\arraystretch{2.4}
{\setlength{\tabcolsep}{1.5em}
\begin{tabular}{|c|c|c|c|}
\hline
Type & Dynkin diagram & Folded Dynkin diagram & Folded type\\ \hline
$A_{2n}$&
\begin{tikzpicture}[scale=.4,baseline]
\node (v1) at (0,0) {};
\node (v2) at (1.8,0) {};
\node (v3) at (3.2,0) {};
\node (v4) at (5,0) {};

\draw (0cm,0) circle (.2cm) node[below]{\small $1$};
\draw (0.2cm,0) -- +(0.2cm,0);
\draw[dotted] (0.4cm,0) --+ (1cm,0);
\draw (1.6cm,0) --+ (0.2cm,0);
\draw (2cm,0) circle (.2cm) node[below]{\small $n$};
\draw (2.2cm,0) -- +(0.6cm,0);
\draw (3cm,0) circle (.2cm);
\draw (3.2cm,0) -- +(0.2cm,0);
\draw[dotted] (3.4cm,0) --+ (1cm,0);
\draw (4.6cm,0) --+ (0.2cm,0);
\draw (5cm,0) circle (.2cm); 

\draw[<->]
(v1) edge[bend left] (v4);
\draw[<->]
(v2) edge[bend left=70] (v3);
\end{tikzpicture}
& 
\begin{tikzpicture}[scale=.4,baseline]
\draw (1.8cm,0) circle (.2cm)node[below]{\small $\hat{1}$}; 
\draw (2cm,0) --+ (0.6cm,0);
\draw (2.8cm,0) circle (.2cm);
\draw (3cm,0) --+ (0.6cm,0);
\draw (3.8cm,0) circle (.2cm);
\draw (4.2cm,0) -- +(0.2cm,0);
\draw[dotted] (4.4cm,0) --+ (1cm,0);
\draw (5.4cm,0) --+ (0.2cm,0);
\draw (5.8cm,0) circle (.2cm);
\draw (6.8cm,0) circle (.2cm)  node[below]{\small $\hat{n}$};
\draw
(6,-.06) --++ (0.6,0)
(6,+.06) --++ (0.6,0);
\draw
(6.4,0) --++ (60:-.2)
(6.4,0) --++ (-60:-.2);
\end{tikzpicture}
& 
BC$_n$
\\ \hline
$A_{2n+1}$& 
\begin{tikzpicture}[scale=.4,baseline]
\node (v1) at (0,0) {};
\node (v2) at (2,0) {};
\node (v3) at (4,0) {};
\node (v4) at (6,0) {};

\draw (0cm,0) circle (.2cm)  node[below]{\small $1$};
\draw (0.2cm,0) -- +(0.2cm,0);
\draw[dotted] (0.4cm,0) --+ (1cm,0);
\draw (1.6cm,0) --+ (0.2cm,0);
\draw (2cm,0) circle (.2cm) node[below]{\small $n$};
\draw (2.2cm,0) -- +(0.6cm,0);
\draw(3cm,0) circle (.2cm) node[below]{\small $n'$};
\draw (3.2cm,0) -- +(0.6cm,0);
\draw (4cm,0) circle (.2cm);
\draw (4.2cm,0) -- +(0.2cm,0);
\draw[dotted] (4.4cm,0) --+ (1cm,0);
\draw (5.6cm,0) --+ (0.2cm,0);
\draw (6cm,0) circle (.2cm);

\draw[<->]
(v1) edge[bend left] (v4);
\draw[<->]
(v2) edge[bend left=50] (v3);
\end{tikzpicture}
& 
\begin{tikzpicture}[scale=.4,baseline]
\draw (0cm,0) circle (.2cm) node[below]{\small $\hat{1}$} ;
\draw (0.2cm,0) -- +(0.6cm,0);
\draw (1cm,0) circle (.2cm);
\draw (1.2cm,0) -- +(0.2cm,0);
\draw[dotted] (1.4cm,0) --+ (1cm,0);
\draw (2.4cm,0) --+ (0.2cm,0);
\draw (2.8cm,0) circle (.2cm) node[below]{\small $\hat{n}$}; 
\draw (3.8cm,0) circle (.2cm)  node[below]{\small $n'$};
\draw
(3,-.06) --++ (0.6,0)
(3,+.06) --++ (0.6,0);
\draw
(3.3,0) --++ (60:.2)
(3.3,0) --++ (-60:.2);
\end{tikzpicture}
& 
C$_{n+1}$
\\ \hline
$D_{n} $& 
\begin{tikzpicture}[scale=.4,baseline]
\node (v1) at (8.8,0.8) {};
\node (v2) at (8.8,-0.8) {};
\draw (0cm,0) circle (.2cm) node[above]{\small $1$} ;
\draw (0.2cm,0) -- +(0.6cm,0);
\draw (1cm,0) circle (.2cm);
\draw (1.2cm,0) -- +(0.2cm,0);
\draw[dotted] (1.4cm,0) --+ (1cm,0);
\draw (2.4cm,0) --+ (0.2cm,0);
\draw (2.8cm,0) circle (.2cm); 
\draw (3cm,0) --+ (0.6cm,0);
\draw (3.8cm,0) circle (.2cm) ;
\draw (4cm,0) --+ (0.6cm,0);
\draw (4.8cm,0) circle (.2cm);
\draw (5cm,0) -- +(0.2cm,0);
\draw[dotted] (5.2cm,0) --+ (1cm,0);
\draw (6.2cm,0) --+ (0.2cm,0);
\draw (6.6cm,0) circle (.2cm)  node[above]{\small $n$-$2$};
\draw (6.8cm,0) --+ (1.6,0.6);
\draw (6.8cm,0) --+ (1.6,-0.6);
\draw (8.6cm,0.6) circle (.2cm) node[above]{\small $n$-$1$}  ;
\draw (8.6cm,-0.6) circle (.2cm) node[below]{\small $n$} ;
\draw[<->]
(v1) edge[bend left] (v2);
\end{tikzpicture}
&
\begin{tikzpicture}[scale=.4,baseline]
\draw (1.8cm,0) circle (.2cm)node[below]{\small $1$}; 
\draw (2cm,0) --+ (0.6cm,0);
\draw (2.8cm,0) circle (.2cm);
\draw (3cm,0) --+ (0.6cm,0);
\draw (3.8cm,0) circle (.2cm);
\draw (4.2cm,0) -- +(0.2cm,0);
\draw[dotted] (4.4cm,0) --+ (1cm,0);
\draw (5.4cm,0) --+ (0.2cm,0);
\draw (5.8cm,0) circle (.2cm) node[below]{\footnotesize $n$-$2$};
\draw (6.8cm,0) circle (.2cm)  node[below]{\small $\hat{n}$};
\draw
(6,-.06) --++ (0.6,0)
(6,+.06) --++ (0.6,0);
\draw
(6.4,0) --++ (60:-.2)
(6.4,0) --++ (-60:-.2);
\end{tikzpicture}
&
B$_{n-1}$
\\ \hline
$E_6$ & 
 \begin{tikzpicture}[scale=.4,baseline=-6pt]
\node (v1) at (0,0.2) {};
\node (v2) at (4,0.2) {};
\node (v3) at (1,0.2) {};
\node (v4) at (3,0.2) {};
\draw (0cm,0) circle (.2cm) node[below]{\small $1$};
\draw (0.2cm,0) -- +(0.6cm,0);
\draw (1cm,0) circle (.2cm) node[below]{\small $2$};
\draw (1.2cm,0) -- +(0.6cm,0);
\draw(2cm,0) circle (.2cm) node[below]{\small $\quad3$};
\draw (2.2cm,0) -- +(0.6cm,0);
\draw (3cm,0) circle (.2cm);
\draw (3.2cm,0) -- +(0.6cm,0);
\draw (4cm,0) circle (.2cm);
\draw (2cm,-0.2) -- +(0cm,-0.6);
\draw(2cm,-1) circle (.2cm) node[below]{\small $6$};
\draw[<->]
(v1) edge[bend left] (v2);
\draw[<->]
(v3) edge[bend left] (v4);
\end{tikzpicture}
&
\begin{tikzpicture}[scale=.4,baseline]
\draw (0cm,0) circle (.2cm) node[below]{\small $\hat{1}$};
\draw (0.2cm,0) --+ (0.6cm,0);
\draw (1cm,0) circle (.2cm) node[below]{\small $\hat{2}$};
\draw (2cm,0) circle (.2cm)  node[below]{\small $3$};
\draw
(1.2,-.06) --++ (0.6,0)
(1.2,+.06) --++ (0.6,0);
\draw
(1.5,0) --++ (60:.2)
(1.5,0) --++ (-60:.2);
\draw (2.2cm,0) --+ (0.6cm,0);
\draw (3cm,0) circle (.2cm) node[below]{\small $6$};
\end{tikzpicture}
&
F$_4$\\
\hline
\end{tabular}
}
\egroup
\end{center}
\end{table}

\begin{Def}
We define the \emph{folding operation} on ungauged twisting data as
\[
\mbH^{\ungauge} \rightarrow \hat{\mbH},\quad (\tau,\epsilon)\mapsto (\id,\hat{\epsilon})\textrm{ where } \hat{\epsilon}_{\hat{r}} = \epsilon_r.
\]
\end{Def}

Note that the folding operation is well-defined since $\epsilon_r = \epsilon_{\tau(r)}$ for ungauged twisting data. It then satisfies in general the rule
\[
\hat{\epsilon}_{\hat{Q}}(\alpha_+)= \epsilon_Q(\alpha),\quad \alpha \in \hat{\Delta}.
\]
Note also that a twisting datum $(\id,\eta)$ in $\hat{\mbH}$ appears as the folding of a deformation in $\mbH^{\ungauge}$ if and only if $\eta_{\hat{r}} \geq 0$ for all folded simple roots $\hat{r}$.

\subsection{Weyl group actions on twisting data}

Let $W \subseteq \End(V)$ be the Weyl group of $\Delta$, with associated simple reflections $s_r$ for $r\in I$. For $J \subseteq I$, we denote $W_J\subseteq W$ for the parabolic subgroup  generated by the $s_r$ with $r\in J$. 

Let $\tau$ be an involution of $\Gamma$. Then $\tau$ defines a group automorphism of $W$ in the obvious way, denoted again $\tau$. We let $W^{\tau}\subseteq W$ be the subgroup of $\tau$-fixed elements. On the other hand, we let $\hat{W} \subseteq \End(V^{\tau})$ be the Weyl group of $\hat{\Gamma}$. 

\begin{Theorem}\label{LemFunctjvarpi1}\cite[Theorem 32]{Ste67}
\begin{enumerate}
\item The restriction 
\[
W^{\tau} \rightarrow \End(V^{\tau}),\qquad w\mapsto w_{\mid V^{\tau}}
\]
induces an isomorphism $W^{\tau}\cong \hat{W}$. 
\item Under this correspondence $\hat{s}_r \mapsto s_{\hat{r}}$ with $\hat{s}_r$ the longest element in the parabolic Weyl group $W_{\{r,\tau(r)\}}$.
\end{enumerate}
\end{Theorem}

 More concretely, we have
\begin{itemize}
\item $\hat{s}_r = s_r$ for $(\alpha_r^{\vee},\alpha_{\tau(r)}) = 2$, so $r\in I^{\tau}$,
\item $\hat{s}_r= s_r s_{\tau(r)}$ for $(\alpha_r^{\vee},\alpha_{\tau(r)}) = 0$,
\item $\hat{s}_r = s_r s_{\tau(r)}s_r$ for $(\alpha_r^{\vee},\alpha_{\tau(r)}) = -1$.
\end{itemize}

As a direct consequence of the above theorem, we have the following corollary. 
\begin{Cor}\label{CorOrbitPosInv}
For each $\tau$-invariant $\omega \in P$ there exists $w\in W^{\tau}$ with $w\omega \in P^+$.
\end{Cor}

In the following we will identify $W^{\tau}\cong \hat{W}$.

\begin{Def}\label{DefWnu}
Let $\nu= (\tau,\epsilon) \in \mbH_J$ be a twisting datum with locus of non-degeneracy $J$. We define the associated Weyl group as
\[
W_{\nu} = W^{\tau}_{J}:= W^{\tau}\cap W_{J} \subseteq W.
\]
We define the \emph{Weyl group action} 
\[
W_{\nu} \times \mbH_{\tau,J}\rightarrow \mbH_{\tau,J},\quad (w\epsilon)_Q(\alpha) = \epsilon_Q(w^{-1}\alpha).
\]
\end{Def}

Note that this is well-defined since $W_J$ preserves $Q^{J+}$. Note also that the action of $W_{\nu}$ descends to each $\mbH_{\tau,J'}$ with $J' \supseteq J$.

\begin{Def}
We call two twisting data $\nu = (\tau,\epsilon)$ and $\nu'=(\tau',\epsilon')$ \emph{Weyl group equivalent}, and write $\nu \sim_{W} \nu'$, if 
\begin{itemize}
\item $\tau = \tau'$,
\item $\epsilon$ and $\epsilon'$ have the same support, and 
\item there exists $w\in W_{\nu}=W_{\nu'}$  such that $\epsilon' = w\epsilon$.
\end{itemize}
\end{Def}

\begin{Lem}
The group $\mbH_{\tau}^{>}$ is stable under $W^{\tau}$. 
\end{Lem}
\begin{proof}
Let $\epsilon \in \mbH_{\tau}^{>}$. If $r\in I^{\tau}$, then $s_r$ multiplies each entry $\epsilon_t$ with a power of $\epsilon_r>0$. If $r\in I^*$ and $t\in I$, then $\hat{s}_r\alpha_t =  \alpha_t - m_{rt} (\alpha_r+\alpha_{\tau(r)})$ for some integer $m_{rt}$, hence $\hat{s}_r$ multiplies $\epsilon_t$ with a power of $|\epsilon_r|^2>0$ by the $\tau$-conjugate symmetry of $\epsilon$. 
\end{proof}

It follows that we can consider the action of $W_{\nu}\ltimes \mbH_{\tau}^>$ on $\mbH_{\tau,J}$. 

\begin{Def}
We call two twisting data $\nu,\nu'\in \mbH$ \emph{equivalent}, and write $\nu \sim \nu'$, if $\nu,\nu'\in \mbH_{\tau,J}$ for some $\tau,J$ and $\nu,\nu'$ lie in the same orbit under the $W_{\nu}\ltimes \mbH_{\tau}^>$-action.
\end{Def}

\begin{Lem}\label{LemStrongRed}
Each twisting datum $\nu$ is equivalent to a strongly reduced twisting datum.
\end{Lem} 

\begin{proof} 
By the multiplicative action of $\mbH_{\tau}^{>}$, the twisting datum $\nu$ is equivalent to a reduced one. It is then enough to prove the lemma when $I$ is connected, $\nu$ is reduced and $\supp(\epsilon) = I$, so in particular $W_{\nu} = W^{\tau}$. In this case $\nu$ can be interpreted as a Vogan diagram. The result thus follows from the Borel and de Siebenthal theorem \cite[Section 6.96]{Kna02}.
\end{proof}

In fact, the proof loc. cit.~ shows that one can achieve the last step using only $W_{I^{\tau}}$. We will at a later point refine this further, see Corollary \ref{CorRefineBS}. 

\subsection{$\nu$-compact roots}

\begin{Def}\label{DefnuCompactRoot}
Let $\nu = (\tau,\epsilon) \in \mbH_{J}^{\ungauge}$ be an ungauged twisting datum. We say that a root $\beta \in \hat{\Delta}$ is \emph{$\nu$-compact} if $\beta \in \hat{\Delta}_{\hat{J}}$ and at least one of the following two conditions is satisfied:
\begin{itemize}
\item $\hat{\epsilon}_{\hat{Q}}(\beta) >0$, or
\item  there exists $r\in I^*\cap J$ with $(\varpi_{\hat{r}}^{\vee},\beta)$ odd. 
\end{itemize} 
We write $\hat{\Delta}_c$ for the set of $\nu$-compact roots in $\hat{\Delta}$. 
\end{Def}

\begin{Rem}
Note that the second condition is in general not subsumed by the first. For example, consider $\begin{tikzpicture}[scale=.4,baseline]
\node (v1) at (0,0) {};
\node (v2) at (2.4,0) {};

\draw (0cm,0) circle (.2cm) node[below]{\small $1$};
\draw (0.2cm,0) -- +(1cm,0);
\draw[fill=black] (1.2cm,0) circle (.2cm) node[below]{\small $2$};
\draw (1.4cm,0) -- +(0.8cm,0);
\draw (2.4cm,0) circle (.2cm) node[below]{\small $3$};

\draw[<->]
(v1) edge[bend left=50] (v2);
\end{tikzpicture}$, the Dynkin diagram $A_3$ with non-trivial involution and $\epsilon_1 = \epsilon_3 =1$ and $\epsilon_2 = -1$. Put $\alpha = \alpha_1+\alpha_2$. One sees that $\hat{\epsilon}_{\hat{Q}}(\alpha_{+}) = -1$, but $\alpha_{+}= \alpha_{\hat{1}} + \alpha_{\hat{2}}$, so $\hat{\alpha}$ is $\nu$-compact.
\end{Rem}

\begin{Prop}
The subset $\hat{\Delta}_c \subseteq \hat{\Delta}$ is a root subsystem. 
\end{Prop}

\begin{proof}
We may assume that $\Gamma$ is $\tau$-connected with $J =I$. Since the proposition is automatic for $\Gamma$ of direct product type, by ungaugedness, we may assume in fact that $\Gamma$ is connected. If $\tau = \id$, then $\hat{\Delta}_c = \Delta_c$ consists of all $\alpha \in \Delta$ with $\epsilon_Q(\alpha)>0$, which clearly form a root subsystem. If $\tau \neq \id$, we may again by ungaugedness assume that $\Gamma$ is not of type $A_{2n}$. The proposition then follows from the following lemma. 
\end{proof}
When $\Delta$ is any irreducible reduced root system, we write resp.~ $\Delta_s$ and $\Delta_l$ for the root subsystems of short and long roots. 
\begin{Lem}\label{LemSubsLongShort}
Assume that $\Gamma$ is a connected Dynkin diagram, not of type $A_{2n}$, with non-trivial involution. Then
\begin{enumerate}
\item  $\hat{\Delta}_s$ consists precisely of those $\beta$ with $(\varpi_{\hat{r}}^{\vee},\beta)$ odd for at least one $r\in I^*$.
\item If $\epsilon \in \mbH_{\tau}^{\ungauge,\times}$, then $\hat{\Delta}_{c}$ is a root subsystem of $\hat{\Delta}$. 
\end{enumerate} 
\end{Lem}
\begin{proof}
Necessarily $\Gamma$ is simply laced, and the only $\hat{\Gamma}$ which can arise are those of type $B_m,C_m$ or $F_4$. Since we can only have $(\alpha_r,\alpha_{\tau(r)}) = 0$ for $\tau(r)\neq r$, as we have excluded the case $A_{2n}$, it is further clear that $r\in I^{\tau}$ if and only if $\alpha_{\hat{r}}$ is long. The first statement of the lemma thus says that $\beta \in \hat{\Delta}_s$ if and only if $\beta$ contains at least one short simple root with odd multiplicity in its decomposition $\beta = \sum_{\hat{r}} k_{\hat{r}}\alpha_{\hat{r}}$. That the latter property is true for $\beta \in \hat{\Delta}_s$ follows immediately by inspection \cite[Appendix, \S 2, Table 1]{OV90}. The converse direction can either again be verified by inspection, or one can argue that any long root is Weyl group conjugate to a simple long root, with the action of $s_{\hat{r}}$ for $\hat{r}$ short adding even multiples of $\alpha_{\hat{r}}$. 

For the second property, note that by the first part
\[
\hat{\Delta}_c = \hat{\Delta}_{l,c} \cup \hat{\Delta}_s,\qquad \textrm{where }\hat{\Delta}_{l,c}=  \{\beta \in \hat{\Delta}_l\mid \hat{\epsilon}_{\hat{Q}}(\beta)>0\}.
\]
Clearly $\hat{\Delta}_{l,c}$ and $\hat{\Delta}_s$ are root subsystems. Since $(\beta,\gamma^{\vee}) \in 2\Z$ for all $\beta \in \hat{\Delta}_l$ and $\gamma\in \hat{\Delta}_s$, we have $s_{\gamma}(\hat{\Delta}_{l,c}) = \hat{\Delta}_{l,c}$ for all $\gamma \in \hat{\Delta}_s$. Since $\hat{\Delta}_s$ is automatically preserved by the whole Weyl group $\hat{W}$, it follows that $\hat{\Delta}_c$ is a root system. 
\end{proof}

\begin{Def}
We define $W_{\nu}^+ \subseteq W_{\nu}$ to be the Weyl group of the root system $\hat{\Delta}_c$.
\end{Def}

\subsection{Twisted dot-action}

We now come to the definition of the twisted dot-action, which will be important for a Harish-Chandra-type formula later on. Note that, in contrast to the classical dot-action, we work here in the multiplicative setting, and the dot-action is with respect to a shift by $\epsilon$ in stead of by the halfsum of all positive roots.

\begin{Def}\label{DefDot}
 Let $\nu = (\tau,\epsilon)\in \mbH^{\ungauge}$ be an ungauged twisting datum. We define an action 
\[
\cdot_{\epsilon}: W_{\nu}\times \mbH_{\tau}^{\times} \rightarrow \mbH_{\tau}^{\times},\quad 
(w\cdot_{\epsilon} \lambda)_P(\omega) = \epsilon_Q(\omega - w^{-1}\omega)\lambda_P(w^{-1}\omega),\qquad \omega \in P.
\]
\end{Def}

Note that this indeed defines an action. The following lemma follows by a direct computation.

\begin{Lem}\label{LemDotAction}
The value of $\hat{s}_r\cdot_{\epsilon}\lambda$ remains everywhere the same as $\lambda$ except at the positions $t\in \{r,\tau(r)\}$, where the value $\lambda_t$ gets multiplied by
\begin{itemize}
\item $\epsilon_t  \lambda_P(\alpha_t)^{-1}$ when $(\alpha_r^{\vee},\alpha_{\tau(r)}) \in \{2,0\}$,
\item $|\epsilon_t|^2 |\lambda_P(\alpha_t)|^{-2}$ when $(\alpha_r^{\vee},\alpha_{\tau(r)}) =-1$.
\end{itemize}
\end{Lem}

\begin{Def}\label{DefWMin}
Let $\nu = (\tau,\epsilon) \in \mbH^{\ungauge}_J$. We define $W_{\nu}^-$ as the set of elements $w =s_{r_n}\ldots s_{r_1}\in W_{\nu}$ with $r_1,\ldots,r_n\in J^{\tau}$ such that  for all $0\leq k \leq n-1$
\begin{equation}\label{EqCondWminus}
(s_{r_{k+1}}\ldots s_{r_1})\cdot_{\epsilon} \mbH^{\gg}_{\tau} \neq  (s_{r_{k}}\ldots s_{r_1})\cdot_{\epsilon} \mbH^{\gg}_{\tau}.
\end{equation}
\end{Def}

Clearly \eqref{EqCondWminus} is the same as asking $(s_{r_{k+1}}\ldots s_{r_1})\cdot_{\epsilon} + \notin  (s_{r_{k}}\ldots s_{r_1})\cdot_{\epsilon} \mbH^{\gg}_{\tau}$, or that at each successive step a sign change is introduced. 

\begin{Theorem}\label{TheoSec}
\begin{enumerate}
\item The map 
\begin{equation}\label{EqBijminplus}
W_{\nu}^- \rightarrow W_{\nu}/W_{\nu}^+,\quad w \mapsto wW_{\nu}^+
\end{equation}
is bijective.
\item If $w_1,w_2\in W_{\nu}^-$ and $w_1\cdot_{\epsilon} \mbH^{\gg}_{\tau} = w_2\cdot_{\epsilon} \mbH^{\gg}_{\tau}$, then $w_1 = w_2$. 
\end{enumerate}
\end{Theorem}
\begin{proof} 
Note that all $\lambda \in W_{\nu}\cdot_{\epsilon} \mbH^{\gg}_{\tau}$ satisfy $\lambda_r>0$ when $\epsilon_r=0$. Using Lemma \ref{LemDotAction}, it is then easily seen that we may restrict to the case where $I$ is $\tau$-connected with $I= \supp(\epsilon)$. Since the theorem is trivial in case we are in the product type situation, we may in fact assume that $I$ is connected. 

Assume now first that $\tau= \id$, so in particular $W_{\nu} =W$. Note that in this case the $\nu$-compact roots in $\hat{\Delta} = \Delta$ are simply the $\alpha$ with $\epsilon_Q(\alpha)>0$.

\textbf{Claim}: $W^+$ is the global stabilizer of $\mbH_{\id}^{\gg}$ under the $\epsilon$-twisted dot-action. 

Choosing a compact root $\alpha$ and $\lambda\in \mbH_{\id}^{\gg}$, we have 
\[
(s_{\alpha}\cdot_{\epsilon} \lambda)_r = \epsilon_Q(\alpha)^{(\varpi_r,\alpha^{\vee})}\lambda_P(s_{\alpha}(\varpi_r)) >0,
\] 
hence $W^+$ leaves $\mbH_{\id}^{\gg}$ globally stable. 

For the converse direction, we first fix a connected and simply connected compact Lie group $U$ with complexification $G=U_{\C}$, together with a maximal torus $T\subseteq U$ and Borel subgroup $T_{\C} \subseteq B \subseteq G$ realizing the root system $\Delta$ with its particular choice of positive roots. Let $\Theta$ be the Cartan involution of $G$ with fixed points $U$, and put $g^* = \Theta(g)^{-1}$. We further make $\nu$ into a complex Lie group involution on $G$ uniquely determined by $\nu(X_{\alpha}) = \sgn (\epsilon_Q(\alpha))X_{\alpha}$ for $X_{\alpha} \in \mfg$ a root vector at weight $\alpha$ in the Lie algebra $\mfg$ of $G$. Let 
\[
G_{\nu} = \{g\in G \mid \nu(g)^* = g^{-1}\},
\] 
so that $G_{\nu}$ is the real semisimple Lie group associated to the Vogan diagram $\sgn(\epsilon)$. As is well-known, 
\begin{equation}\label{EqIdWeylGroup}
W^+ = W(T:G_{\nu}) = N_{G_{\nu}}(T)/T,
\end{equation}
see e.g.~ \cite[Chapter IX.\S7]{Kna01}.

Pick $w\in W$ which leaves $\mbH_{\id}^{\gg}$ globally stable, and pick a representative $\widetilde{w} \in N_{U}(T)$.  It is by \eqref{EqIdWeylGroup} sufficient to show that also $\widetilde{w} \in G_{\nu}$. But let $V_{\varpi}$ be an irreducible unitary representation of $U$ at highest weight $\varpi$, and let $\msE_{\varpi}$ be the operator on $V_{\varpi}$ determined by $\msE_{\varpi} v = \sgn(\epsilon_Q(\varpi - \wt(v)))v$ for $v$ a weight vector at weight $\wt(v)\in P$.  Then $\pi_{\varpi}(\nu(u)) = \msE_{\varpi} \pi_{\varpi}(u)\msE_{\varpi}$ for $u\in U$. It follows that
\[
\pi_{\varpi} (\nu(\widetilde{w})^*\widetilde{w})v = \msE_{\varpi} \pi_{\varpi}(\widetilde{w})^{-1} \msE_{\varpi} \pi_{\varpi}(\widetilde{w})v  = \sgn(\epsilon_Q(\wt(v) - w\wt(v)))v = v
\]
by the assumption on $w$. This proves $\nu(\widetilde{w})^*=\widetilde{w}^{-1}$, hence $\widetilde{w}\in G_{\nu}$ and so $w\in W^+$. This finishes the proof of the claim. 

Let us show now that the map \eqref{EqBijminplus} is bijective. 

To see that the map \eqref{EqBijminplus} is surjective, write any $w \in W$ as a product of $s_r$ and remove those $s_r$ which act by multiplication with $\mbH_{\id}^{\gg}$ within the expression $w\cdot_{\epsilon} +$. Then we obtain an element $w_-$ such that $w_-\cdot_{\epsilon} \mbH_{\id}^{\gg} =w\cdot_{\epsilon}\mbH_{\id}^{\gg}$. By the previous claim, $w = w_-w_+$ for $w_+\in W^+$, proving surjectivity. 

To see that the map \eqref{EqBijminplus} is injective, we again first make a claim. 

\textbf{Claim}: $W^+ \cap W^- = \{e\}$. 

Indeed, assume that $w\in W^-$, and write $w = s_{r_n}\ldots s_{r_1}$ as in \eqref{EqCondWminus}. We may clearly take this to be a reduced expression. Consider the roots 
\[
\gamma_{k} = s_{r_1}\ldots s_{r_{k-1}}(\alpha_{r_k}).
\]
By the assumption on $w$, we have
\[
\R_{<0} =  \frac{(s_{r_k} \ldots s_{r_1}\cdot +)_{r_k}}{(s_{r_{k-1}}\ldots s_{r_1}\cdot +)_{r_k}} \R_{>0} = \epsilon_Q(\alpha_{r_k}) (s_{r_{k-1}} \ldots s_{r_1}\cdot+)_P(\alpha_{r_k})^{-1} \R_{>0} = \epsilon_Q(\gamma_k)\R_{>0},
\]
hence all the $\gamma_k$ are non-compact. Since 
\begin{equation}\label{EqIntPosComp}
w^{-1}\Delta^+ \cap \Delta^+ = \Delta^+\setminus \{\gamma_1,\ldots,\gamma_n\},
\end{equation} 
we can not have $w\in W^+$ unless $w$ is trivial. This proves the claim. 

Assume now that $w_1,w_2 \in W^-$ with $w_1W^+ = w_2W^+$, or equivalently $w_1\cdot_{\epsilon} \mbH^{\gg}_{\id} = w_2\cdot_{\epsilon}\mbH^{\gg}_{\id}$. Write $w_1 = s_{r_n}\ldots s_{r_1}$ and $w_2 = s_{t_m}\ldots s_{t_1}$ as in \eqref{EqCondWminus}. Then $w = s_{t_1}\ldots s_{t_m}s_{r_n}\ldots s_{r_1}$ is seen to be an expression still satisfying \eqref{EqCondWminus}, hence $w \in W^-$. On the other hand, since $w$ stabilizes $\mbH^{\gg}_{\id}$ we have $w \in W^+$. From the above claim, it follows that $w= e$ and hence $w_1 = w_2$. 

This ends the proof that \eqref{EqBijminplus} is bijective, and it also follows at once that property (2) in the theorem holds. 

Let us turn now to the case $\tau \neq \id$. We may assume that $\Gamma$ is not of type $A_{2n}$, since in this case $W_{\nu}^-$ is trivial and $W_{\nu}^+  = W^{\tau}$. Consider the Dynkin subdiagram $\Gamma' = \Gamma_{I^{\tau}}$ with the twisting datum $\nu' = (\id,\epsilon')$ where $\epsilon'$ is the restriction of $\epsilon$. Since any $\alpha_r$ with $r \in I^{\tau}$ has $(\alpha_r,\alpha_s^{\vee}) = (\alpha_r,\alpha_{\tau(s)}^{\vee})$ for all $s$, it follows by Lemma \ref{LemDotAction} and \eqref{EqDefPropEps} that $W_{\nu}^-$ coincides with its counterpart $W_{\nu'}^-$ for $\Gamma'$. By the first part of the proof, we have a bijection
\begin{equation}\label{EqPartialBij}
W_{\nu}^- \times W_{\nu'}^+ \rightarrow W_{I^{\tau}}. 
\end{equation}
The first part of the proof also shows already that property (2) in the theorem holds. 

Let us show now that \eqref{EqBijminplus} is injective. It is sufficient to show that 
\begin{equation}\label{EqIntRoots}
W_{\nu'}^+  = W_{\nu}^+ \cap W_{I^{\tau}}.
\end{equation}
Since $W_{I^{\tau}}$ is parabolic, it follows by Chevalley's lemma that $W_{\nu}^+ \cap W_{I^{\tau}}$ is generated by reflections $s_{\beta}$ across roots $\beta \in \hat{\Delta}_c\cap \Delta_{I^{\tau}} = (\Delta_{I^{\tau}})_c$, implying \eqref{EqIntRoots}.

To see that \eqref{EqBijminplus} is surjective, it is by \eqref{EqPartialBij} sufficient to show that 
\[
W_{I^{\tau}} \times W_{\nu}^+ \rightarrow W_{\nu}
\]
is surjective. However, we have by Lemma \ref{LemSubsLongShort} that $W_{\nu}^+ \supseteq \hat{W}_s$, the Weyl group of the  short roots $\hat{\Delta}_s$. Since $\hat{W}_s$ is normal in $\hat{W}$, we find 
\[
W_{I^{\tau}} W_{\nu}^+ \supseteq W_{I^{\tau}}\hat{W}_s  = W_{I^{\tau}}\hat{W}_s W_{I^{\tau}}.
\]
Hence $W_{I^{\tau}}\hat{W}_s$ is a group containing all $s_{\hat{r}}$, so $W_{I^{\tau}} W_{\nu}^+ = W_{I^{\tau}}\hat{W}_s = W_{\nu}$. 
\end{proof} 

We end this section by recording the following results. Fix $\nu$ an ungauged deformation datum.

\begin{Lem}\label{LemDecompRho}
Let $w\in W_{\nu}^-$ and $v\in W_{\nu}^+\setminus \{e\}$. Let 
\[
\rho = \frac{1}{2}\sum_{\alpha \in \Delta^+} \alpha = \sum_{r\in I}\varpi_r. 
\]
Then $w^{-1}\rho -(wv)^{-1}\rho \in \hat{Q}_c^+\setminus\{0\}$. 
\end{Lem} 

\begin{proof}
If $v = s_{\beta_{r_1}}\ldots s_{\beta_{r_n}}$ is a reduced decomposition of $v$, with $\beta_k$ the simple roots of $\hat{\Delta}_c$, then  
\[
w^{-1}\rho -(wv)^{-1}\rho = \sum_{p=1}^n (w^{-1}\rho, s_{\beta_{r_1}}\ldots s_{\beta_{r_{p-1}}}\beta_{r_p}^{\vee})\beta_{r_p}.
\]
Since all $s_{\beta_{r_n}}\ldots s_{\beta_{r_{p+1}}}\beta_{r_p}\in \hat{\Delta}_c^+$, and since $(\rho,\alpha_+^{\vee})\in \Z$ for all $\alpha \in \Delta$, it is hence sufficient to show that $(w^{-1}\rho,\beta) > 0$ for all $\beta \in \hat{\Delta}_c^+$. In turn, this will follow if we can show $w\beta \in \hat{\Delta}^+$. But assume the latter were not the case. Writing $w = \hat{s}_{r_m}\ldots \hat{s}_{r_1}$ as in \eqref{EqCondWminus}, it follows that we must then have $\hat{s}_{r_p}\ldots \hat{s}_{r_1}\beta = -\alpha_{\hat{r}_p}$ for some $p$. This however implies $\beta = \hat{s}_{r_1}\ldots \hat{s}_{r_{p-1}}\alpha_{\hat{r}_p}$, which is non-compact by the proof of Theorem \ref{TheoSec}. This contradiction ends the proof.  
\end{proof}

\begin{Lem}\label{LemPosWPlus}
For all $w\in W_{\nu}^+$ and $\omega \in P$ one has
\begin{equation}\label{EqPosWPlus}
\epsilon_Q(\omega - w\omega) >0. \tag{*}
\end{equation}
Conversely, if $\beta \in \hat{\Delta}$ is such that \eqref{EqPosWPlus} holds for $w=s_{\beta}$ and all $\omega \in P^{\tau}$, then $\beta$ is $\nu$-compact.
\end{Lem}

\begin{proof}
It is clear that \eqref{EqPosWPlus} will hold once it holds for all $w = s_{\beta}$ with $\beta  \in \hat{\Delta}_c$. In this case, \eqref{EqPosWPlus} states that 
\begin{equation}\label{EqPosWPlusAlt}
\hat{\epsilon}_{\hat{Q}}(\beta)^{(\omega,\beta^{\vee})} >0, \qquad \forall \omega \in P^{\tau}.\tag{**}
\end{equation}
We may assume that $\supp(\epsilon) = I$ and that $\Gamma$ is connected and not of type $A_{2n}$ with non-trivial involution. 

The condition is clearly satisfied for the first case in Definition \ref{DefnuCompactRoot}. In the second case we have that $\beta$ is a short root in $\hat{\Delta}$, and so $(\omega,\beta^{\vee})\in 2\Z$. 

Assume now conversely that $\beta\in \hat{\Delta}$ is such that \eqref{EqPosWPlus} holds for $s_{\beta}$, or equivalently \eqref{EqPosWPlusAlt} holds for all $\omega \in P^{\tau}$. Then clearly $\beta$ must be supported on $\hat{J} = \supp(\hat{\epsilon})$, so we may suppose that $J = I$. If then $\beta$ contains $\alpha_r$ with $r\in I^{\tau}$ with odd multiplicity, it follows upon chosing $\omega = \varpi_r$ that $\hat{\epsilon}_{\hat{Q}}(\beta) >0$, hence $\beta$ is $\nu$-compact. If on the other hand $\beta$ contains all $\alpha_r$ with $r\in I^{\tau}$ with even multiplicity, then necessarily $\beta$ must contain a folded root with odd multiplicity. Since the latter are short, we are done.
\end{proof}

\begin{Cor}\label{CorRefineBS}
If $\nu \sim \nu'$ are equivalent ungauged twisting data, there exists $w\in W_{\nu}^-$ and $f\in \mbH_{\tau}^{>}$ such that $\epsilon' = f w\epsilon$.
\end{Cor}
\begin{proof}
Note first that if $\beta\in \hat{\Delta}_c^+$, then for $\alpha \in \Delta_+$ we have 
\[
(s_{\beta}\epsilon)_Q(\alpha) = \epsilon_Q(\alpha) \hat{\epsilon}_{\hat{Q}}(\beta)^{(\alpha,\beta^{\vee})}.
\]
Now except for type $A_{2n}$ with non-trivial involution, we always have $\alpha_+ \in P$ for $\alpha \in \Delta^+$. Since $\epsilon$ is however positive anyway for type $A_{2n}$ by the ungaugedness assumption, we deduce from the previous lemma that $s_{\beta}\epsilon = g\epsilon$ with $g\in \mbH_{\tau}^{\gg}$. 

Assume now that $\epsilon ' = f (w\epsilon)$ for $w\in W_{\nu}$ and $f\in \mbH_{\tau}^{>}$. We can write $w = uv$ with $u\in W_{\nu}^-$ and $v\in W_{\nu}^+$. By the above we see that $\epsilon' = f'(u\epsilon)$ for some $f'\in \mbH_{\tau}^{>}$. 
\end{proof}

\section{Quantized enveloping algebras and their deformations}

\subsection{Quantized enveloping algebras}

Let $\mfg$ be a complex semisimple Lie algebra of rank $l$. We fix a set of Chevalley generators
\[
\mcS = \{h_r,e_r,f_r\mid r\in I\} \subseteq \mfg,
\] 
indexed by a set $I$ with $|I|=l$. We write $\mfb^{\pm}$ for the associated Borel subalgebras, $\mfh$ for the Cartan subalgebra and $\mfn^{\pm}$ for the associated nilpotent subalgebras. We write $\mfu \subseteq \mfg_{\R}$ for the compact form of $\mfg$ determined by the above choice of Chevalley generators, and we write $\mft = \mfh \cap \mfu$ for the maximal torus of $\mfu$. We write $\mfa = i \mft \subseteq \mfh_{\R}$. When viewing $\mfg$ as a real Lie algebra by forgetting the complex structure, we write $\mfg_{\R}$. 

We use the notation of the previous section for the associated root system $\Delta$ on $\mfa\cong \mfa'$. We will also use the shorthand notation
\[
d_r = \frac{1}{2}(\alpha_r,\alpha_r).
\]

Fix in the following $0<q<1$. Algebraically there is no issue allowing also $q>1$, but this will influence spectral conditions later on. 
\begin{Def}
The quantized enveloping algebra $U_q(\mfb^+)$ is the algebra generated by $K_{\omega}^+,E_r$ for $r\in I$ and $\omega\in P$, with $K_0^+ = 1$  and commutation relations
\[
K_{\omega}^+K_{\chi}^+ = K_{\omega+ \chi}^+, \qquad K_{\omega}^+ E_r = q^{(\omega,\alpha_r)} E_rK_{\omega}^+
\]
and the \emph{quantum Serre relations} whose precise form we will not need in what follows, see e.g. \cite[Section 6.1.2]{KS97}.
\end{Def}
Similarly one defines $U_q(\mfb^-)$. We write its generators as $K_{\omega}^-$ and $F_r^-$, with in particular $K_0^{-} = 1$ and
\[
K_{\omega}^-K_{\chi}^- = K_{\omega+ \chi}^-, \qquad K_{\omega}^- F_r = q^{-(\omega,\alpha_r)} F_rK_{\omega}^-.
\]
The algebras $U_q(\mfb^{\pm})$ form Hopf algebras for the comultiplication
\begin{equation}\label{EqComultBorel}
 K_{\omega}^{\pm}\mapsto K_{\omega}^{\pm}\otimes K_{\omega}^{\pm},\quad E_r \mapsto E_r\otimes 1+ K_r^+\otimes E_r,\quad F_r \mapsto F_r\otimes (K_r^-)^{-1} + 1\otimes F_r. 
\end{equation}

We let $U_q(\mfg)$ be the associated quantized enveloping algebra of $\mfg$, which is generated by a copy of the Hopf algebras $U_q(\mfb^+)$ and $U_q(\mfb^-)$ with the relations $K_{\omega}^+  = K_{\omega}^- =: K_{\omega}$ and interchange relations
\[
\lbrack E_r,F_s\rbrack =\delta_{r,s} \frac{ K_{r} -  K_r^{-1}}{q_r-q_r^{-1}},
\]
where we use the shorthand $K_r = K_{\alpha_r}$ and $q_r = q^{d_r}$. Endowed with the $*$-structure 
\[
K_{\omega}^* = K_{\omega},\qquad E_r^* = F_rK_r
\]
we obtain a Hopf $*$-algebra which we denote $U_q(\mfu)$. Its antipode squared is given by 
\begin{equation}\label{EqAntiSquared}
S^2 = \Ad(K_{2\rho}^{-1}).
\end{equation}
Occasionally, we will use $U_q'(\mfu)$ as generated by the $E_r,F_r$ and $K_r$. 

We view $U_q(\mfu)$ as a spectral $*$-algebra with admissible finite-dimensional $*$-representations those for which all $K_{\omega}$ have positive spectrum. Up to unitary equivalence the irreducible admissible $*$-representations are parametrized by $P^+$, and we write $V_{\varpi}$ for a $*$-representative associated to $\varpi \in P^+$. The correspondence is determined by the fact that $V_{\varpi}$ has a highest weight vector $\xi_{\varpi}$, which we normalize by $\|\xi_{\varpi}\|=1$, with 
\[
K_{\omega}\xi_{\varpi} = q^{(\omega,\varpi)}\xi_{\varpi},\qquad E_r \xi_{\varpi} = 0 \textrm{ for all }r.
\] 
When $V_{\pi}$ is a finite-dimensional admissible $*$-representation and $v\in V_{\pi}$ is a joint eigenvector of the $K_{\omega}$, we write $\wt(v)$ for the unique element of $P$ such that 
\[
K_{\omega}v = q^{(\omega,\wt(v))}v. 
\]
We write $U_{\varpi}(\xi,\eta)$ for the matrix coefficients of a highest weight $*$-representation $\pi_{\varpi}$ of highest weight $\varpi$.

\begin{Def} 
We write
\[
\mcO_q(U)  = \mcO(U_q) = \{U_{\pi}(\xi,\eta)\}
\] 
for the Hopf $*$-algebra of matrix coefficients of finite-dimensional admissible $*$-representations of $U_q(\mfu)$.
\end{Def} 
Here the $*$-structure is determined by
\[
(f^*,X) = \overline{(f,S(X)^*)},\qquad f\in \mcO_q(U),X\in U_q(\mfu).
\]
The definition remains the same if we replace $U_q(\mfu)$ by $U_q'(\mfu)$.

We view $\mcO_q(U)$ as a $q$-deformation of the $*$-algebra $\mcO(U)$ of regular functions on the simply connected compact Lie group $U$ integrating $\mfu$. The algebra $\mcO_q(U)$ separates elements in $U_q(\mfu)$, and we may identify $U_q(\mfu)\subseteq \mcO_q(U)'$, the linear dual of $\mcO_q(U)$. Moreover, $\mcO_q(U)$ has a positive Haar integral, and is hence associated to a compact quantum group which we write $U_q$. We have 
\[
\int_{U_q} U_{\varpi}(\xi,\eta) = \delta_{\varpi,0}\langle \xi,\eta\rangle,
\]
and the Woronowicz character is given by 
\begin{equation}\label{EqDeltaUq}
\hat{\delta}^{iz}(U_{\varpi}(\xi,\eta)) = q^{4iz(\rho,\wt(\eta))}\langle \xi,\eta\rangle,
\end{equation}
where 
\[
\rho = \frac{1}{2}\sum_{\alpha \in \Delta^+} \alpha = \sum_{r=1}^l \varpi_r.
\] 
Formally, we can write $\hat{\delta}^{iz} = K_{\rho}^{4iz}$. We further note the following.

\begin{Prop}[{\cite[Section 4.6.7]{LS91}, \cite[Theorem 2.7.14]{NT13}}]\label{PropCoamTyp1}
The compact quantum group $U_q$ is coamenable, and $C_q(U)  = C(U_q)$ is a type $I$ C$^*$-algebra. 
\end{Prop}

We write $L^{\infty}_q(U) = L^{\infty}(U_q)$ for the associated von Neumann algebra.

When forgetting the $*$-structure, we write $\mcO_q(U)$ rather as 
\[
\mcO_q(G) = \mcO(G_q),
\] 
with $G_q$ viewed as a $q$-deformation of the simply connected \emph{complex} algebraic group $G$ integrating $\mfg$.

Note that when $l=1$, we indeed obtain that $\mcO_q(SU(2))$, as defined in Section \ref{ExaRank1}, is dual to $U_q(\mfsu(2))$. More precisely, the duality is determined by the 2-dimensional $*$-representation $\pi_{1/2}$ with associated corepresentation $U$ such that 
\[
(U,K_{\varpi}) = \begin{pmatrix} q & 0 \\0 & q^{-1}\end{pmatrix},\qquad (U,E_{\alpha}) = \begin{pmatrix} 0 & q^{1/2} \\ 0 & 0 \end{pmatrix},
\]
where $\alpha$ is the simple root, normalized such that $(\alpha,\alpha)=2$, and $\varpi = \alpha/2$ is the fundamental weight. In the following, we will in this case write $E = E_{\alpha}$ and $K = K_{\alpha}$. Then $U_q'(\mfsu(2))$ is generated by $K,E,E^*$.

\subsection{Deformations of quantized enveloping algebras from twisting data, and associated coactions}

We keep the setting of the previous section. 

Let $\mu,\nu$ be twisting data. One can create a $(\mu,\nu)$-twisted Drinfeld double of $U_q(\mfb^+)$ and $U_q(\mfb^-)$ as follows, see \cite[Section 3.4]{DCNTY19} and \cite[Section 2.1]{DCM18}.\footnote{We allow somewhat more generality by allowing complex values for the parameters, but this extra generality is inessential and will be removed in a moment.}

\begin{Def}
For $\mu,\nu\in \mbH$ twisting data, the \emph{$(\mu,\nu)$-deformed quantized enveloping algebra} $U_q^{\mu,\nu}(\wmfg)$, with $\wmfg = \mfg \oplus \mfh$, is generated by $U_q(\mfb^+)$ and $U_q(\mfb^-)$ with interchange relations 
\[
K_{\omega}^+ K_{\chi}^- = q^{(\omega,\tau_{\nu}(\chi)-\tau_{\mu}(\chi))}K_{\chi}^-K_{\omega}^+, 
\]
\begin{equation}\label{EqCommRelBR}
K_{\omega}^- E_r = q^{(\alpha_r,\tau_{\mu}(\omega))}E_rK_{\omega}^-,\qquad K_{\omega}^+F_r =  q^{-(\alpha_r,\tau_{\nu}(\omega))} F_r K_{\omega}^+,
\end{equation}
\[
\lbrack E_r,F_s\rbrack = \frac{\delta_{r,\tau_{\nu}(s)} \epsilon_{\nu,r} K_{r}^+ - \delta_{r,\tau_{\mu}(s)} \overline{\epsilon_{\mu,s}} (K_{s}^-)^{-1}}{q_r-q_r^{-1}}.
\]
\end{Def}

It is easy to see that the multiplication maps 
\[
U_q(\mfb^+) \otimes U_q(\mfb^-) \rightarrow U_q(\wmfg),\qquad U_q(\mfb^-)\otimes U_q(\mfb^+) \rightarrow U_q(\wmfg)
\]
are vector space isomorphisms. The above algebra becomes a $*$-algebra for 
\[
(K_{\omega}^+)^* = K_{\omega}^-,\qquad E_r^* = F_rK_r^-, 
\]
in which case we write it as $U_q^{\mu,\nu}(\wmfu)$, where $\wmfu$ is interpreted as $\mfu \oplus \mfa$. The $U_q^{\mu,\nu}(\wmfu)$ form a $*$-compatible Hopf-Galois system, also called co-groupoid \cite{Bic14}, for the maps 
\[
\Delta: U_q^{\mu,\nu}(\wmfu) \rightarrow U_q^{\mu,\kappa}(\wmfu)\otimes U_q^{\kappa,\nu}(\wmfu),
\]
given by \eqref{EqComultBorel} on the components $U_q(\mfb^{\pm})$. In particular, the $U_q^{\nu,\nu}(\wmfu)$ are Hopf $*$-algebras. We write $U_q(\wmfu) = U_q^{++}(\wmfu)$. We see that $U_q(\mfu)$ coincides with the  quotient of $U_q(\wmfu)$ by putting $K_{\omega}^+ = K_{\omega}^-$.

In the following, we will only be interested in the case where $\mu=+$ is the trivial twisting datum. We write in this case the associated deformed quantized enveloping $*$-algebra 
\[
U_q^{\nu}(\wmfu) = U_q^{+,\nu}(\wmfu)
\]
with $\nu = (\tau,\epsilon)$. It comes equipped with a left coaction by $U_q^{++}(\widetilde{\mfu})$ and hence by $U_q(\mfu)$.

Fix now $\nu= (\tau,\epsilon)$. We use again the following shorthand notations for $\omega \in \mfa'\cong Q\otimes_{\Z}\R$ and $r\in I$,
\begin{equation}\label{EqShorthandPM}
\omega_{\pm} = \frac{1}{2}(\id\pm \tau)\omega,\qquad c_{\omega} = q^{(\omega_-,\omega_-)},\qquad c_r = c_{\alpha_r}. 
\end{equation}
Write $U_q^{\nu}(\mfu)$ for the $*$-subalgebra of $U_q^{\nu}(\wmfu)$ generated by the $E_r,E_r^*$ and the 
\[
T_{\omega} := c_{\omega}^{-1}  K_{-\tau(\omega)}^+K_{-\omega}^-.
\]
We have for $U_q^{\nu}(\mfu)$ the universal relations $T_0 = 1$ and
\[
T_{\omega + \chi}  =  T_{\omega}T_{\chi},\qquad T_{\omega}^* = T_{\tau(\omega)},
\]
\begin{equation}\label{CommTE}
T_{\omega} E_r = q^{-2(\omega_+,\alpha_r)}E_rT_{\omega} ,\qquad T_{\omega}E_r^* = q^{2(\omega_+,\alpha_r)}E_r^*T_{\omega} 
\end{equation}
\[
E_r E_s^*  - q^{-(\alpha_r,\alpha_s)}  E_s^*E_r = \frac{\delta_{r,\tau(s)} \epsilon_r c_r T_{-\alpha_s} - \delta_{r,s}}{q_r-q_r^{-1}}, 
\]
together with the usual Serre relations between the $E_r$. As one notices, the Cartan part of $U_q^{\nu}(\mfu)$ should rather be interpreted as $U_q(\mfh_{\nu})$, with
\[
\mfh_{\nu}  = \mft^{\tau} \oplus \mfa^{-\tau}.
\]
Note also that when $\nu = +$, we obtain an embedding $U_q^+(\mfu) \hookrightarrow U_q(\mfu)$, but it is not surjective as we only reach elements in the Cartan algebra with weights in $2P$. 

It is easily checked that the left coaction of $U_q(\wmfu)$ on $U_q^{\nu}(\wmfu)$  gives rise to a $*$-preserving left coaction
\[
\gamma: U_q^{\nu}(\mfu)\rightarrow U_q(\mfu) \otimes U_q^{\nu}(\mfu),
\]
\begin{equation}\label{EqDefGamma}
T_{\omega} \mapsto K_{-\omega -\tau(\omega)}\otimes T_{\omega},\quad E_r \mapsto E_r\otimes 1 + K_r \otimes E_r. 
\end{equation}
Moreover, the $*$-algebra $U_q^{\nu}(\wmfu)$ is through its Galois object structure endowed with a left $U_q(\wmfu)$-module $*$-algebra structure, called the \emph{Miyashita-Ulbrich action}, given in this concrete case by 
\begin{equation}\label{EqMiyaUlbrich}
X \rhdb Y = X_{(1)}YS(X_{(2)}),\qquad Y \in U_q^{\nu}(\wmfu), X \in U_q(\mfb)\cup U_q(\mfb^-).
\end{equation}
This is easily seen to descend to a left $U_q(\mfu)$-module $*$-algebra structure on $U_q^{\nu}(\mfu)$ (although the latter action will not be inner anymore if $\tau \neq \id$). The triple $(U_q^{\nu}(\mfu),\gamma,\rhdb)$ forms a $*$-compatible Yetter-Drinfeld $U_q(\mfu)$-module algebra. 

The next lemma shows that all $U_q^{\nu}(\mfu)$, with $\nu$ ranging in an equivalence class of twisting data, are isomorphic as Yetter-Drinfeld $U_q(\mfu)$-module $*$-algebras. The isomorphism can moreover be chosen such that it is compatible with certain `spectral conditions', which will be explored further in the next sections. 

\begin{Lem}\label{LemWeakEq}
Let $\nu,\nu' \in \mbH_{\tau}$ be equivalent twisting data. Then there exists $f \in \mbH_{\tau}^{\times}$ such that 
\[
\chi_f: U_q^{\nu'}(\mfu) \rightarrow U_q^{\nu}(\mfu), \qquad T_{\omega} \mapsto f_P(\tau(\omega))T_{\omega},\quad E_r \mapsto E_r
\]
extends to an isomorphism of Yetter-Drinfeld $U_q(\mfu)$-module $*$-algebras. Moreover, if $\epsilon' = \lambda (w\epsilon)$ for $w\in W_{\nu}$ and $\lambda \in \mbH_{\tau}^{>}$, then $f$ can be chosen so that the following condition with respect to $(w,\epsilon)$ holds:
\begin{equation}\label{EqSpecCondIso}
f_P(\omega) \epsilon_Q(\omega - w^{-1}\omega)>0,\qquad \forall \omega \in P^{\tau}.
\end{equation}
\end{Lem} 
\begin{proof}
It is clear that $\chi_f$ will be a $*$-isomorphism of Yetter-Drinfeld module $*$-algebras as soon as $f \in \mbH_{\tau}^{\times}$ with $\epsilon_r' = f_P(\alpha_r) \epsilon_r$ for all $r\in I$. 

Assume now that $\epsilon' = \lambda (w\epsilon)$ for $w\in W_{\nu}$ and $\lambda \in \mbH_{\tau}^{>}$. Choose $N\in \N$ such that $P \subseteq \frac{1}{N}Q$, and pick $\chi \in \mbH_{\tau}^{>}$ such that $\chi_r^N = \lambda_r$ for all $r\in I$. Define $g\in \mbH_{\tau}^{\times}$ by $g_P(\omega) = \prod \chi_r^{k_r}$ if $\omega = \sum_r k_r (N^{-1}\alpha_r)$ with $k_r\in \N$. Then clearly $g_P(\alpha_r) =\lambda_r$, and $g_P(\omega)>0$ for $\omega \in P^{\tau}$. Define now
\[
f_P(\omega) = g_P(\omega) \epsilon_Q(\omega -w^{-1}\omega)^{-1},\qquad \omega \in P.
\]
Then $f$ clearly satisfies the requirements. 
\end{proof}

Remark that this allows us to restrict ourselves to deal with strongly reduced twisting data only. The general case is however convenient to keep around for using scaling arguments, though we will from now on always assume that at least $\nu \in \mbH^{\ungauge}$ when considering $U_q^{\nu}(\mfu)$.

\subsection{Interpretation of $U_q^{\nu}(\mfu)$ as quantized function algebra}

We want to introduce certain spectral conditions on $U_q^{\nu}(\mfu)$. These are more easily justified for the quantized function algebra interpretation of $U_q^{\nu}(\mfu)$ which we now introduce.

\begin{Def}
Let $\nu= (\tau,\epsilon)$ be an ungauged twisting datum. We define $\mcO_q(Z_{\nu}^{\reg})$ as the $U_q(\mfu)$-module $*$-algebra generated by elements $a_{\omega}$ for $\omega \in P$ and $x_r,x_r^*$ with $r\in I$ such that the map 
\[
\iota_{\nu}: \mcO_q(Z_{\nu}^{\reg}) \rightarrow U_q^{\nu}(\mfu),\quad a_{\omega}\mapsto T_{\omega},\quad x_r \mapsto E_r
\]
is a $U_q(\mfu)$-equivariant $*$-isomorphism.
\end{Def}

\begin{Rem}
\begin{itemize}
\item Note that, with respect to the notation in \cite[(2.36)]{DCM18}, we have rescaled the $a_{\omega}$ so that $\omega \mapsto a_{\omega}$ is a homomorphism. 
\item As we are working in the concrete setting where $q$ is a definite number, the above renaming process is of course purely formal. Within the deformation setting where $q$ is a parameter, the choice of generators for $\mcO_q(Z_{\nu}^{\reg})$ has to be made more judiciously as to obtain the correct limit as a function algebra when $q$ is set equal to $1$.
\end{itemize}
\end{Rem}

The following theorem collects some of the results in \cite[Section 2]{DCM18}. Note first that $\tau$ can be interpreted as a Hopf $*$-algebra isomorphism of $U_q(\mfu)$ by 
\begin{equation}\label{EqTauHopfIso}
\tau(K_{\omega}) = K_{\tau(\omega)},\qquad \tau(E_r) = E_{\tau(r)},\qquad \tau(F_r) = F_{\tau(r)}. 
\end{equation}
By duality, we also have that $\tau$ determines a Hopf $*$-algebra automorphism of $\mcO_q(U)$. 

\begin{Theorem}\label{TheoOqZ}
Endow $\mcO_q(U)$ with the $\tau$-twisted adjoint $U_q(\mfu)$-module structure 
\begin{equation}\label{EqTwistedAdj}
X \rhdb U_{\pi}(\xi,\eta) = U_{\pi}(\tau(S(X_{(1)}))^*\xi,X_{(2)}\eta).
\end{equation}
Then there exists a unique injective $U_q(\mfu)$-module map 
\begin{equation}\label{EqInclusioninu}
i_{\nu}: \mcO_q(U) \rightarrow \mcO_q(Z_{\nu}^{\reg})
\end{equation}
such that 
\begin{equation}\label{EqInclU}
i_{\nu}(U_{\varpi}(\xi_{\varpi},\xi_{\varpi})) = c_{\varpi}a_{\varpi}.
\end{equation}
\end{Theorem}

\begin{Def}\label{DefOqZ}
We define 
\[
\mcO_q(Z_{\nu}) = i_{\nu}(\mcO_q(U)) \subseteq \mcO_q(Z_{\nu}^{\reg})
\]
to be the image of $i_{\nu}$.
\end{Def}

\begin{Theorem}\label{TheoPropOqZ}
The following properties hold for $\mcO_q(Z_{\nu})$.
\begin{itemize}
\item $\mcO_q(Z_{\nu})$ is a $*$-subalgebra of $\mcO_q(Z_{\nu}^{\reg})$,
\item $\mcO_q(Z_{\nu})\cap \C[a_{\omega}\mid \omega \in P] = \C[a_{\varpi}\mid \varpi \in P^+]$,
\item $\mcO_q(Z_{\nu}^{\reg})$ is generated as an algebra by $\mcO_q(Z_{\nu})$ and $a_{\rho}^{-1}$, and
\item $\mcO_q(Z_{\nu})$ is generated as a $U_q(\mfu)$-module by the $a_{\varpi}$. 
\end{itemize}
\end{Theorem} 

It follows in particular from the above theorem that the $\rhdb$-action on $\mcO_q(Z_{\nu}^{\reg})$ restricts to a module $*$-algebra structure on $\mcO_q(Z_{\nu})$ which integrates to an $\mcO_q(U)$-coaction 
\[
\rho_{\nu}: \mcO_q(Z_{\nu}) \rightarrow \mcO_q(Z_{\nu})\otimes \mcO_q(U),\quad f \mapsto \rho_{\nu}(f)
\]
where
\[
(\id\otimes X)\rho_{\nu}(f) = X\rhdb f,\qquad X \in U_q(\mfu).
\]
In the following, we will also write
\[
Z_{\pi}(\xi,\eta) = i_{\nu}(U_{\pi}(\xi,\eta)),
\]
and we pair $\mcO_q(Z_{\nu})$ with $U_q(\mfu)$ along $i_{\nu}$. Then 
\begin{equation}\label{EqGlobalCoactRho}
(\id\otimes \rho_{\nu})Z_{\pi} = (\id\otimes \tau)(U_{\pi})_{13}^*Z_{\pi,12}U_{\pi,13}.
\end{equation}

We can also transport the coaction $\gamma$ of \eqref{EqDefGamma} through the isomorphism $\iota_{\nu}$ into a coaction $\widetilde{\gamma}$ by $U_q(\mfu)$ on $\mcO_q(Z_{\nu}^{\reg})$, so
\begin{equation}\label{EqDefGammaTilde}
\widetilde{\gamma}: \mcO_q(Z_{\nu}^{\reg}) \rightarrow U_q(\mfu)\otimes  \mcO_q(Z_{\nu}^{\reg}),\qquad a_{\omega} \mapsto K_{-\omega -\tau(\omega)}\otimes a_{\omega},\quad x_r \mapsto E_r\otimes 1 + K_r \otimes x_r. 
\end{equation}

\begin{Lem}\label{LemYD}
The coaction $\widetilde{\gamma}$ restricts to a coaction of $U_q(\mfu)$ on $\mcO_q(Z_{\nu})$. 
\end{Lem}
\begin{proof}
Let us write 
\[
\widetilde{\gamma}(f) = f_{(-1)}\otimes f_{(0)},\qquad f\in \mcO_q(Z_{\nu}^{\reg}).
\]

Using the Yetter-Drinfeld module structure, we have
\begin{equation}\label{EqYDStruct}
\widetilde{\gamma}(X\rhdb f) = X_{(1)} f_{(-1)}S(X_{(3)}) \otimes (X_{(2)}\rhdb f_{(0)}),\qquad f\in \mcO_q(Z_{\nu}^{\reg}),X\in U_q(\mfu).
\end{equation}
In particular,
\[
\widetilde{\gamma}(X\rhdb a_{\varpi}) = X_{(1)} K_{-\varpi -\tau(\varpi)}S(X_{(3)}) \otimes (X_{(2)}\rhdb a_{\varpi}),\qquad \varpi \in P^+,X\in U_q(\mfu).
\]
Since the $a_{\varpi}$ generate $\mcO_q(Z_{\nu})$ as a $U_q(\mfu)$-module, this proves the lemma.
\end{proof}

On the subalgebra $\mcO_q(Z_{\nu})$ the map $\iota_{\nu}$ can also be given by a global formula. Let
\[
\msR \in U_q(\mfb^+)\hat{\otimes}U_q(\mfb^-)
\] 
be the universal $R$-matrix of $U_q(\mfg)$ in some completed tensor product, normalized such that 
\[
\msR(\xi\otimes \eta) = q^{-(\wt(\xi),\wt(\eta))}\xi\otimes \eta
\]
for $\xi$ a highest weight and $\eta$ a lowest weight vector.\footnote{Note that the normalization of the $R$-matrix in Section \ref{ExaRank1} was chosen compatibly with this convention.} Viewing $\mcO_q(Z_{\nu})$ again as functionals on $U_q(\mfg)$,  we have the following. Let $\msE$ be the unique element in the completion $\mcO_q(U)'\cong \prod_{\varpi} B(V_{\varpi})$ of $U_q(\mfg)$ such that in the highest weight $*$-representation $V_{\varpi}$ we have 
\begin{equation}\label{DefMsE}
\msE v = \epsilon_Q(\varpi - \wt(v))v.
\end{equation}
Let $\msR_{\tau} = (\tau\otimes \id)\msR = (\id\otimes \tau)\msR$.
\begin{Prop}\cite[Proposition 2.30]{DCM18}
For all $f\in \mcO_q(Z_{\nu})$ one has
\begin{equation}\label{EqGlobalIota}
\iota_{\nu}(f) = (f\otimes \id)(\msR_{\tau,21}(\msE\otimes 1)\msR).
\end{equation}
\end{Prop}
Here the right hand side is easily seen to be a well-defined element in $U_q(\mfb^+)U_q(\mfb^-) \subseteq U_q^{\nu}(\wmfg)$.

We record for the computations in the following section that, writing 
\begin{equation}\label{EqKappa1}
\kappa_-(f) = (f\otimes \id)\msR,\qquad \kappa_+(f) = (\id\otimes f)\msR^{-1},\qquad f\in O_q(G),
\end{equation}
we have the following general commutation rule
\begin{equation}\label{EqSwitch}
\kappa_+(g)\kappa_-(f) = \mbr_{\nu}(f_{(1)},g_{(1)}) \kappa_-(f_{(2)})\kappa_+(g_{(2)}) \mbr(f_{(3)},S^{-1}(g_{(3)}))
\end{equation}
inside $U_q(\wmfu)$, where $\mbr(f,g) = (f\otimes g,\msR)$ and $\mbr_{\nu}(f,g) =(f\otimes g,\msR_{\nu})$.

\section{Centrally coinvariant coactions from twisting data and the Harish-Chandra morphism}

\subsection{Centrally coinvariant coactions}

Fix an ungauged twisting datum $\nu = (\tau,\epsilon) \in \mbH^{\ungauge}$, and recall the integrable $U_q(\mfu)$-module $*$-algebra $\mcO_q(Z_{\nu})$ introduced in Definition \ref{DefOqZ}, with associated $\mcO_q(U)$-coaction $\rho_{\nu}$. Let 
\[
\mcO_q(Z_{\nu}\dbslash U) = \mcO_q(Z_{\nu})^{\rho_{\nu}}
\] 
be the $\rho_{\nu}$-coinvariants, or equivalently the $U_q(\mfu)$-invariant part of $\mcO_q(Z_{\nu})$ under $\rhdb$. In using this notation, we interpret $\mcO_q(Z_{\nu})$ as a full spectral $*$-algebra.

\begin{Lem}\label{LemInvPartZ}
We have 
\[
\mcO_q(Z_{\nu}\dbslash U)  \subseteq \msZ(\mcO_q(Z_{\nu})),
\]
where the right hand side denotes the center of $\mcO_q(Z_{\nu})$. Hence the coaction $\rho_{\nu}$ is centrally coinvariant. 
\end{Lem} 
\begin{proof}
By means of the equivariant embedding $\iota_{\nu}$ it is sufficient to show that the $U_q(\mfu)$-invariant elements $U_q^{\nu}(\mfu)_{\inv}$ in $U_q^{\nu}(\mfu)$ for the action \eqref{EqMiyaUlbrich} are central. But for the $U_q(\wmfu)$-module $*$-algebra $U_q^{\nu}(\wmfu)$ we have that $U_q^{\nu}(\wmfu)_{\inv} = \msZ(U_q^{\nu}(\wmfu))$ since the associated $U_q(\wmfu)$-action is the Miyashita-Ulbrich action for a Galois object, see e.g.~ \cite[Theorem 3.2.7]{CGW06}. One can also check this property directly in the case at hand .
\end{proof}

Consider the projection map
\begin{equation}\label{DefE}
E = E_{\rho_{\nu}}: \mcO_q(Z_{\nu}) \rightarrow \mcO_q(Z_{\nu}\dbslash U),\qquad E(f) =  \left(\id\otimes \int_{U_q}\right)\rho_{\nu}(f).
\end{equation}

To have a more concrete expression for $E$, we first make the following observations. 

Recall that $\tau$ defines a Hopf $*$-algebra involution of $U_q(\mfu)$ by \eqref{EqTauHopfIso}. Fix $\varpi \in P^+$. We have that the $Z_{\varpi}(\xi,\eta)$ span a $U_q(\mfu)$-representation isomorphic to $V_{\tau(\varpi)}^* \otimes V_{\varpi}$. We hence see that this span contains no invariant element unless $\tau(\varpi) = \varpi$, in which case it contains a one-dimensional subspace of invariant elements. Let for $\varpi = \tau(\varpi)$ the operator
\[
J_{\varpi}: V_{\varpi} \rightarrow V_{\varpi}
\]
be uniquely determined by the requirements
\[
J_{\varpi}\xi_{\varpi} = \xi_{\varpi},\qquad J_{\varpi}Xv = \tau(X)J_{\varpi}v,\qquad X\in U_q(\mfu),v\in V_{\varpi}.
\]
One easily sees that $J_{\varpi}$ is a selfadjoint operator with $J_{\varpi}^2 = 1$. For $\omega\in P$ with $\tau(\omega) = \omega$, we write 
\begin{equation}\label{EqForj}
j_{\varpi}(\omega) = \Tr(J_{\varpi}{}_{\mid V_{\varpi}(\omega)}),
\end{equation}
with $V_{\varpi}(\omega)$ the weight $\omega$ subspace.

\begin{Lem}\label{LemFunctjvarpi}
\begin{enumerate}
\item The functions $\omega \mapsto j_{\varpi}(\omega)$ are independent of $q$.
\item Each function $\omega \mapsto j_{\varpi}(\omega)$ is $W^{\tau}$-invariant. 
\end{enumerate}
\end{Lem}
\begin{proof}
The first statement follows since $j_{\varpi}(\omega) \in \Z$ depends continuously on $q$. The second statement then follows (for $q=1$) by \cite[(5.26)]{FSS96}.
\end{proof}

\begin{Lem}\label{LemImE}
With $\{e_i\}$ an orthonormal basis of $V_{\varpi}$, we have
\begin{equation}\label{EqEonZ}
E(Z_{\varpi}(\xi,\eta)) = \delta_{\varpi,\tau(\varpi)}\frac{\langle \xi,J_{\varpi}\eta\rangle}{\dim_q(V_{\varpi})}\sum_{i,j} \langle e_j,K_{2\rho}J_{\varpi}e_i\rangle Z_{\varpi}(e_i,e_j),
\end{equation}
where  $\dim_q(V_{\varpi}) = \Tr(K_{2\rho})$. 
\end{Lem}
\begin{proof}

Using \eqref{EqGlobalCoactRho}, we obtain
\[
E(Z_{\varpi}(\xi,\eta))= \sum_{i,j} \left(\int_{U_q} \tau(U_{\varpi}(e_i,\xi))^*U_{\varpi}(e_j,\eta)\right)Z_{\varpi}(e_i,e_j).
\]
From the orthogonality relations \eqref{EqOrthoGen} and \eqref{EqDeltaUq} for $\mcO_q(U)$, we obtain that this is $0$ for $\tau(\varpi)\neq \varpi$, while for $\tau(\varpi) = \varpi$ this reduces to  
\[ 
\sum_{i,j} \left(\int_{U_q} U_{\varpi}(J_{\varpi}e_i,J_{\varpi}\xi)^*U_{\varpi}(e_j,\eta)\right)Z_{\varpi}(e_i,e_j),
\]
from which \eqref{EqEonZ} follows immediately. 
\end{proof}

Using \eqref{EqInclU}, we obtain
\begin{Cor}\label{CorIntroz}
Let $\varpi\in P^+$ with $\tau(\varpi) =\varpi$ and define
\[
z_{\varpi} = \dim_q(V_{\varpi}) E(a_{\varpi}) =\sum_{i,j} \langle e_j,K_{2\rho}J_{\varpi}e_i\rangle Z_{\varpi}(e_i,e_j).
\]
Then the $z_{\varpi}$ form a linear basis for $\mcO_q(Z_{\nu}\dbslash U)$.
\end{Cor}

We recall now from \cite[Lemma 2.33]{DCM18} that the $a_{\varpi}$ satisfy the commutation
\begin{equation}\label{EqCommapi}
a_{\varpi} Z(\xi,\eta) = q^{2(\varpi_+,\wt(\xi)-\wt(\eta))}Z(\xi,\eta)a_{\varpi}.
\end{equation}
This still holds with $\varpi$ replaced by a general weight in $P$ when interpreting the identity within $\mcO_q(Z_{\nu}^{\reg})$. 

\begin{Lem}
Define 
\begin{equation}\label{EqSigmaForm}
\sigma_z: \mcO_q(Z_{\nu}) \rightarrow \mcO_q(Z_{\nu}),\quad \sigma_z(Z(\xi,\eta)) = q^{iz(2\rho,\wt(\xi)-\wt(\eta))}Z(\xi,\eta),
\end{equation}
and write $\sigma = \sigma_{-i}$. Then the $\sigma_z$ are algebra automorphisms with
\[
E(fg) = E(g\sigma(f)),\qquad \forall f,g\in \mcO_q(Z_{\nu}).
\] 
\end{Lem}
\begin{proof}
Within $\mcO_q(Z_{\nu}^{\reg})$ we can write 
\begin{equation}\label{EqSigmaAda}
\sigma = \Ad(a_{\rho})
\end{equation} 
It follows from this that all $\sigma_z$ are algebra automorphisms. 

By \eqref{EqCommapi}, Lemma \ref{LemImE} and $\tau$-invariance of the $\varpi_+$, we have that 
\begin{equation}\label{EqModulara}
E(a_{\varpi}z) = E(za_{\varpi}),\qquad \forall \varpi\in P^+,\forall z\in \mcO_q(Z_{\nu}). 
\end{equation}
Note further that we can write 
\[
\sigma(Z(\xi,\eta)) = Z(K_{2\rho}\xi,K_{2\rho}^{-1}\eta).
\]
Using \eqref{EqTwistedAdj} and \eqref{EqAntiSquared} it follows that
\begin{equation}\label{EqSigmaS}
\sigma(X \rhdb Y) = S^2(X) \rhdb \sigma(Y),\qquad X\in U_q(\mfu),Y\in \mcO_q(Z_{\nu}).
\end{equation}
From invariance of $E$, we also have
\begin{equation}\label{EqStrongInv}
E((X\rhdb Y)Z) = E(Y (S(X)\rhdb Z)).
\end{equation}
Hence
\begin{multline*}
E((X\rhdb a_{\varpi})Y) \underset{\eqref{EqStrongInv}}{=} E(a_{\varpi}(S(X)\rhdb Y)) \underset{\eqref{EqModulara}}{=} E((S(X)\rhdb Y)a_{\varpi})  \\ \underset{\eqref{EqStrongInv}}{=} E(Y (S^2(X) \rhdb a_{\varpi})) \underset{\eqref{EqSigmaS}}{=} E(Y\sigma(X\rhdb a_{\varpi})). 
\end{multline*}
Since the $X\rhdb a_{\varpi}$ span $\mcO_q(Z_{\nu})$ linearly, we obtain the lemma. 
\end{proof}

\subsection{Harish-Chandra homomorphism} 

Let $U_q(\mfh_{\nu})\subseteq U_q^{\nu}(\mfu)$ be the  algebra linearly spanned by the $T_{\omega}$. Let $U_q(\mfn)$ be the unital algebra generated by the $E_r$, and $U_q(\mfn^-)$ the algebra generated by the $E_r^*$. 

\begin{Lem}\label{LemCommNilp}
An element in $U_q(\mfn)$ commutes with the $T_{\omega}$ if and only if it is a scalar.
\end{Lem}
\begin{proof}
If $X \in U_q(\mfn)$ commutes with the $T_{\omega}$, we may assume it is a sum of elements of the form $E_{\alpha_{i_1}}\ldots E_{\alpha_{i_n}}$, each of which commutes with $T_{\omega}$. Putting $\alpha = \sum \alpha_{i_k} \in Q^+$, we then have by \eqref{CommTE} that $(\omega_+,\alpha) = 0$ for all $\omega \in P$, hence $\alpha_+ =0$. But then $\alpha = -\tau(\alpha)$. Since $\alpha\in Q^+$, this entails $\alpha = 0$. Hence $X$ is scalar. 
\end{proof}

Let 
\begin{equation}\label{EqProjection}
P: U_q^{\nu}(\mfg)  \rightarrow U_q(\mfh_{\nu}),\quad XYZ \rightarrow \varepsilon(X)Y\varepsilon(Z),\qquad X\in U_q(\mfn^-),Y\in U_q(\mfh_{\nu}),Z\in U_q(\mfn) 
\end{equation}
be the projection map on the Cartan part with respect to this particular triangular decomposition. Using Lemma \ref{LemCommNilp}, the standard argument gives that $P$ is a $*$-homomorphism on the degree zero-part of $U_q^{\nu}(\mfu)$ with respect to the natural $P$-grading. Hence we have in particular the \emph{Harish-Chandra $*$-homomorphism}
\[
\chi_{\HC}: \msZ(U_q^{\nu}(\mfu)) \rightarrow U_q(\mfh_{\nu}),\quad X\mapsto P(X).
\] 
From the proof of Lemma \ref{LemInvPartZ}, we see that $\iota_{\nu}$ maps $\mcO_q(Z_{\nu}\dbslash U)$ into $\msZ(U_q^{\nu}(\mfu))$. Hence we can consider the $*$-homomorphism
\[
\chi_{\HC}\circ \iota_{\nu}: \mcO_q(Z_{\nu}\dbslash U) \rightarrow U_q(\mfh_{\nu}).
\]

Recall the elements $z_{\varpi}$ introduced in Corollary \ref{CorIntroz}.

\begin{Theorem}[Harish-Chandra formula]
For $\varpi \in P^+$ with $\tau(\varpi) = \varpi$ we have
\begin{equation}\label{EqHCForm}
\chi_{\HC}(\iota_{\nu}(z_{\varpi})) = \underset{\tau(\omega) = \omega}{\underset{\varpi -\omega \in Q^+}{\sum_{\omega\in P^+}}} j_{\varpi}(\omega)  |\Stab_{W_{\nu}}(\omega)|^{-1} \epsilon_{Q}(\varpi - \omega) \left(\sum_{w\in W_{\nu}} \epsilon_Q(\omega - w\omega) q^{-2(\rho,w\omega)}T_{w\omega}\right).
\end{equation}
\end{Theorem}

\begin{proof}
We would like to use formula \eqref{EqGlobalIota}, but unfortunately the expression in that formula is with respect to the opposite triangular decomposition of the one for the Harish-Chandra projection \eqref{EqProjection}. We hence have to first bring \eqref{EqGlobalIota} in the right form. 

Endow $\mcO_q(Z_{\nu})$ with the coalgebra structure inherited from $\mcO_q(U)$ through the map $i_{\nu}$ in \eqref{EqInclusioninu}. If $f\in \mcO_q(Z_{\nu}\dbslash U)$ we have that
\[
f_{(2)}\otimes S(\tau(f_{(1)}))f_{(3)} = f\otimes 1,
\]
from which one obtains
\[
\Delta(f) = (\id\otimes S^2\tau)\Delta^{\cop}(f)
\]
and hence 
\begin{equation}\label{EqTwistQuantTrace}
f(XY) = f(\tau(S^2(Y))X),\qquad X,Y \in U_q(\mfg).
\end{equation}
Using the shorthand notation $\msR = \msR_1\otimes \msR_2$ for the universal $R$-matrix, and using the maps $\kappa_{\pm}$ from \eqref{EqKappa1}, we find using \eqref{EqGlobalIota} that for $f\in \mcO_q(Z_{\nu}\dbslash U)$
\begin{eqnarray*}
\iota_{\nu}(f) &=& f_{(1)}(\tau(\msR_{2}))f_{(2)}(\msE)f_{(3)}(\msR_{1'})) \msR_{1}\msR_{2'} \\ 
&=& f_{(2)}(\msE) \kappa_+(S(\tau(f_{(1)}))) \kappa_-(f_{(3)}) \\
&\underset{\eqref{EqSwitch}}{=}& f_{(4)}(\msE) \mbr_{\nu}(f_{(5)},S(\tau(f_{(3)}))) \kappa_-(f_{(6)})\kappa_+(S(\tau(f_{(2)}))) \mbr(f_{(7)},\tau(f_{(1)})) \\
&=& f(\tau(\msR_2\msR_{2'}S(\nu(\msR_{2''})))\msE \msR_{1''}\msR_{1^{'''}}\msR_1) \msR_{2^{'''}} \msR_{1'}\\
&\underset{\eqref{EqTwistQuantTrace}}{=}&  f(\tau(\msR_{2'}S(\nu(\msR_{2''})))\msE \msR_{1''}\msR_{1^{'''}}\msR_1S^{-2}(\msR_2)) \msR_{2^{'''}} \msR_{1'} \\ 
&=&  f(\tau(\msR_{2'})\msE S(\msR_{2''})\msR_{1''}\msR_{1^{'''}}\msR_1S^{-2}(\msR_2)) \msR_{2^{'''}} \msR_{1'},
\end{eqnarray*}
where in the last line we used $\tau(\nu(Y))\msE = \msE Y$ for $Y \in U_q(\mfb^-)$. 

Let us now write
\[
u = S(\msR_{2})\msR_1,\quad v = S(u) = \msR_1S(\msR_2),
\]
so that also
\[
v^{-1} = \msR_1 S^{-2}(\msR_2),\quad u^{-1} = \msR_2S^2(\msR_1).
\]
We have that $\Ad(u)$ implements $S^2$ and $uv^{-1} = K_{2\rho}^{-2}$, the grouplike implementing $S^4$, see e.g.~ \cite[Proposition 2.1.8 and Corollary 2.1.9]{Maj95}.

Since we can write
\begin{eqnarray*}
\iota_{\nu}(f) &=&  f(\tau(\msR_{2})\msE u\msR_{1^{'}}v^{-1}) \msR_{2^{'}} \msR_{1},
\end{eqnarray*}
we thus obtain 
\[
\iota_{\nu}(f) = f(\tau(\msR_{2})\msE S^2(\msR_{1'})K_{2\rho}^{-2})\msR_{2'}\msR_1.
\]
Since $\msR = \widetilde{\msR}\msQ$ with $\widetilde{\msR} = \sum_{\alpha\in Q^+}\widetilde{\msR}_{\alpha}$, where $\wmsR_0 = 1\otimes 1$ and $\widetilde{\msR}_{\alpha} \in  U_q(\mfn)_{\alpha}\otimes U_q(\mfn^-)_{-\alpha}$, and where 
\[
\msQ(\xi\otimes \eta) = q^{-(\wt(\xi),\wt(\eta))}\xi\otimes \eta,
\]
we see that 
\begin{eqnarray} 
\nonumber \chi_{\HC}(\iota_{\nu}(z_{\varpi})) &=& \sum_{i,j} \langle e_j,K_{2\rho}J_{\varpi} e_i\rangle P(\iota_{\nu}(Z_{\varpi}(e_i,e_j)))\\
\nonumber &=& \sum_{i,j} \langle e_j,K_{2\rho}J_{\varpi} e_i\rangle \langle e_i, \tau(\msR_2)\msE S^2(\msR_{1'})K_{2\rho}^{-2}e_j\rangle P(\msR_{2'}\msR_1)\\
\nonumber &=& \sum_{i,j} \langle e_j,K_{2\rho}J_{\varpi} e_i\rangle \langle e_i, \tau(\msQ_2)\msE S^2(\msQ_{1'})K_{2\rho}^{-2}e_j\rangle \msQ_{2'}^-\msQ_1^+\\
\nonumber &=& \sum_{i} \langle e_i,J_{\varpi} e_i\rangle\epsilon_Q(\varpi -\wt(e_i))q^{-2(\rho,\wt(e_i))}K_{- \wt(e_i)}^-K_{-\tau(\wt(e_i))}^+\\
\label{EqExtraId}&=& \underset{\tau(\omega) = \omega}{\underset{\varpi - \omega \in Q^+}{\sum_{\omega \in P}}} j_{\varpi}(\omega) \epsilon_Q(\varpi - \omega) q^{-2(\rho,\omega)}  T_{\omega}.
\end{eqnarray}
Now if $\omega \in P$ with $\tau(\omega) = \omega$, $\varpi- \omega \in Q^+$ and $\epsilon_Q(\varpi - \omega) \neq 0$, it follows that $\varpi - \omega \in Q_J^+$, where $J = \supp(\epsilon)$. From Corollary \ref{CorOrbitPosInv}, we can find $w\in W_{\nu}$ such that $w\omega$ is dominant integral on $Q_J$. Since however we also must have $\varpi -w\omega \in Q_J^+$ and $\varpi\in P^+$, we see that in fact $w\omega \in P^+$. It follows that 
\[
\chi_{\HC}(\iota_{\nu}(z_{\varpi})) =  \underset{\tau(\omega) = \omega}{\underset{\varpi - \omega \in Q^+}{\sum_{\omega \in P^+}}} |\Stab_{W_{\nu}}(\omega)|^{-1}  \epsilon_Q(\varpi -\omega)\sum_{w\in W_{\nu}} j_{\varpi}(w\omega) \epsilon_Q(\omega - w\omega) q^{-2(\rho,w\omega)}  T_{w\omega}.
\]
We can now deduce \eqref{EqHCForm} from Lemma \ref{LemFunctjvarpi}.(2).
\end{proof}

Let us now transport the above triangular decomposition and Harish-Chandra homomorphism to $\mcO_q(Z_{\nu}^{\reg})$. Write $\msN^+$ for the unital algebra generated by the $x_r$, $\msN^-$ for the unital algebra generated by the $x_r^*$, and $\msA$ for the algebra generated by the $a_{\omega}$. Consider the associated triangular decomposition
\begin{equation}\label{EqTriangZ}
\msN^- \otimes \msA \otimes \msN^+ \cong \mcO_q(Z_{\nu}^{\reg})
\end{equation}
where the isomorphism is a vector space isomorphism induced by multiplication. The projection to the $\msA$-component gives the $*$-homomorphism 
\begin{equation}\label{EqFormHCZ}
\widetilde{\chi}_{\HC}: \mcO_q(Z_{\nu}\dbslash U) \rightarrow \msA,\quad z_{\varpi} \mapsto \underset{\tau(\omega) = \omega}{\underset{\varpi -\omega \in Q^+}{\sum_{\omega\in P^+}}} j_{\varpi}(\omega)  |\Stab_{W_{\nu}}(\omega)|^{-1} \epsilon_Q(\varpi - \omega) \left(\sum_{w\in W_{\nu}} \epsilon_Q(\omega - w\omega) q^{-2(\rho,w\omega)}a_{w\omega}\right).
\end{equation}

View $\msA$ as functions on $\mbH_{\tau}^{\times}$ by 
\begin{equation}\label{EqCorrAFunc}
a_{\omega}(\mu) = \mu_P(\omega),
\end{equation}
using the notation of \eqref{EqEpsP}. We will use a $q$-deformation of the $\epsilon$-twisted dot action introduced in Definition \ref{DefDot}.

\begin{Def}\label{DefDotq}
We define the $q$-deformed $\epsilon$-twisted dot action of $W_{\nu}$ on $\mbH_{\tau}^{\times}$ by $\cdot_{\epsilon,q}= \cdot_{q^{2\rho}\epsilon}$, so
\[
(w\cdot_{\epsilon,q} \lambda)_P(\omega) = \epsilon_Q(\omega - w^{-1}\omega)q^{(2\rho,\omega -w^{-1}\omega)}\lambda_P(w^{-1}\omega)
\]
\end{Def}

\begin{Cor}\label{CorOrbitHC}
Let $\mu,\mu' \in \mbH_{\tau}^{\times}$, and assume that 
\[
\widetilde{\chi}_{\HC}(z_{\varpi})(\mu) = \widetilde{\chi}_{\HC}(z_{\varpi})(\mu'),\qquad \textrm{for all } \varpi \in P^+\textrm{ with }\tau(\varpi)= \varpi.
\] 
Then there exists $w\in W_{\nu}$ and a gauge $\gamma \in \mbH_{\tau}^{\gauge}$ such that 
\[
\mu' = \gamma(w\cdot_{\epsilon,q} \mu).
\] 
\end{Cor} 
\begin{proof}
For $\varpi \in (P^+)^{\tau}$  we write
\[
\hat{a}_{\varpi} = \sum_{w\in W_{\nu}} \epsilon_Q(\varpi - w\varpi) q^{-2(\rho,w\varpi-\varpi)}a_{w\varpi},
\]
and we let $\msA_{\tau}^{W_{\nu}}$ be the subspace of $\msA$ spanned by the $\hat{a}_{\varpi}$. By triangularity of the map 
 \[
(P^+)^{\tau}\rightarrow \msA_{\tau}^{W_{\nu}},\quad \varpi \mapsto \widetilde{\chi}_{\HC}(z_{\varpi}) = \underset{\tau(\omega) = \omega}{\underset{\varpi -\omega \in Q^+}{\sum_{\omega\in P^+}}} j_{\varpi}(\omega)  |\Stab_{W_{\nu}}(\omega)|^{-1} \epsilon_Q(\varpi - \omega) q^{-2(\rho,\omega)}\hat{a}_{\omega}
\]
with respect to the natural partial order on $(P^+)^{\tau}$ given by $\varpi \geq \varpi'$ if and only if $\varpi - \varpi' \in Q^+$, it follows that $\widetilde{\chi}_{\HC}$ is surjective onto $\msA_{\tau}^{W_{\nu}}$. 

Note now that the functions $\hat{a}_{\varpi}$ can be written as 
\[
\mu \mapsto  \hat{a}_{\varpi}(\mu)  = \sum_{w\in W_{\nu}} a_{\varpi}(w\cdot_{\epsilon,q} \mu).
\]  
Consider the action of $\mbH_{\tau}^{\gauge} \cong \T^{|I^*|}$ on $\mbH^{\times}_{\tau}$ by multiplication. Then the associated action on $\msA$ satisfies
\[
\int_{\mbH_{\tau}^{\gauge}} a_{\varpi}(\chi-)\rd \chi = \delta_{\varpi,\tau(\varpi)}a_{\varpi}. 
\]
Since the $a_{\varpi}$ for $\varpi \in P^+$ separate points in $\mbH_{\tau}^{\times}$, we must thus have that the range of $\widetilde{\chi}_{\HC}$ separates the (compact) $\mbH_{\tau}^{\gauge}\times W_{\nu}$-orbits of $\mbH_{\tau}^{\times}$, from which the corollary follows. 
\end{proof}

\section{Highest weight $*$-representations for $\mcO_q(Z_{\nu})$ and $\mcO_q(G_{\nu}\backslash G_{\R})$}\label{SecAdmis}

\subsection{The spectral $*$-algebra $\mcO_q(G_{\nu}\backslash G_{\R})$}\label{SubSecSpecOrConc}

Let $\nu = (\tau,\epsilon)\in \mbH^{\ungauge}$ be an ungauged twisting datum. We recall that we view $\mcO_q(Z_{\nu})$ as a full spectral $*$-algebra. In the following we also write this as 
\[
\mcO_q(Z_{\nu}) =  \mcO(Z_{\nu,q}) = \mcO_q(G_{\nu}\dbbackslash G_{\R}) = \mcO(G_{\nu,q}\dbbackslash G_{\R,q}).
\]
Here $G$ is again the connected, simply connected complex Lie group integrating $\mfg$, written as $G_{\R}$ when viewing it specifically as a real Lie group in stead of a complex one. The symbol $G_{\nu}$ evokes a closed real Lie subgroup of $G_{\R}$ determined by $\nu$, see \cite{DCM18} for more information on this interpretation. When $\nu$ is a reduced symmetric twisting datum, $G_{\nu}$ is the real form of $G$ determined by $\nu$ as its Vogan diagram.

Recall from Lemma \ref{LemInvPartZ} that the coaction $\rho_{\nu}$ is centrally coinvariant. We can hence consider the topological spaces
\[
 G_{\nu,q}\dbbackslash G_{\R,q}/ U_q = Z_{\nu,q}/ U_q \subseteq Z_{\nu,q}\dbslash U_q = G_{\nu,q}\dbbackslash G_{\R,q}\dbslash U_q.
\]

However, as one can glean from this notation, there is also a more refined spectral condition one can impose.

\begin{Def}
If $\pi$ is a bounded $*$-representation of $\mcO_q(Z_{\nu})$ on a Hilbert space $\Hsp_{\pi}$, we call a joint eigenspace of the $\pi(a_{\varpi})$ a \emph{weight space}. We call $\pi$ a \emph{weight $*$-representation} if $\Hsp_{\pi}$ is the closure of the sum of its weight spaces.  

If $\xi$ lies in a weight space, we call $\lambda \in \mbH_{\tau}$ with  
\[
a_{\varpi}\xi = \lambda_P(\varpi)\xi,\qquad \forall \varpi \in P^+
\]
the associated weight.
\end{Def}
Note that $*$-compatibility ensures that $\lambda$ is indeed a $\tau$-twisting datum.

\begin{Lem}\label{LemIrrRepZZ}
Any irreducible $*$-representation $\pi$ of $\mcO_q(Z_{\nu})$ is a weight $*$-representation, and all its weights lie in the same $\mbH_{\tau}^{\gg}$-orbit for the multiplication action.   
\end{Lem}
\begin{proof}
It  follows from \eqref{EqCommapi} that the absolute values $|\pi(a_{\varpi})|$ obey $q$-commutation relations with $\pi(\mcO_q(Z_{\nu}))$. By irreducibility, we then either have $|\pi(a_{\varpi})|=0$ or $\Ker(|\pi(a_{\varpi})|) =0$ with the spectrum of $|\pi(a_{\varpi})|$ a set with $0$ as its only possible accumulation point. It also follows from \eqref{EqCommapi} and irreducibility that the polar parts of the $\pi(a_{\varpi})$ are central, hence scalar. It is now immediate that $\pi$ is a weight $*$-representation.
\end{proof}

\begin{Def}
We call an irreducible $*$-representation $\pi$ of $\mcO_q(Z_{\nu})$ 
\begin{itemize}
\item \emph{of regular type} if $\pi(a_{\rho})$ has zero kernel,
\item \emph{positive} if $\pi(a_{\varpi_r})\geq0$ for all $r\in I^{\tau}$.
\end{itemize}
\end{Def} 

Note that if $\pi$ is an irreducible $*$-representation of regular type, it follows by Lemma \ref{LemIrrRepZZ} that there exists a dense subspace $V_{\pi}\subseteq \Hsp_{\pi}$, namely the algebraic direct sum of the weight spaces, on which $\pi$ extends to a (non-bounded) $*$-representation of $\mcO_q(Z_{\nu}^{\reg})$.

\begin{Def}
We write $\mcO_q(Z_{\nu}^+)$ for $\mcO_q(Z_{\nu})$ as a spectral $*$-algebra with respect to the family $\msP^+$ of all irreducible, positive representations of regular type. 
\end{Def}

Borrowing notation from the introduction, the irreducible positive $*$-representations of regular type of $\mcO_q(Z_{\nu}^+)$ correspond to the `open cell' $Z_{\nu,q}^{+,\reg}$. 

The coaction $\rho_{\nu}$ will in general not be spectral on $\mcO_q(Z_{\nu}^+)$. We therefore apply the construction described in 
Lemma \ref{LemTransOrb}, noting that the hypotheses are satisfied by Proposition \ref{PropCoamTyp1}.

\begin{Def}
We define $\mcO_q(G_{\nu}\backslash G_{\R})$ as the spectral $*$-algebra $\mcO_q(Z_{\nu}^+U)$ with spectral coaction $\rho_{\nu}$.
\end{Def} 

We can then also consider
\[
Z_{\nu,q}^+U_q/U_q  = G_{\nu,q}\backslash G_{\R,q}/U_q. 
\]
It is not true of course that any $x$-representation for $x\in  G_{\nu,q}\backslash G_{\R,q}/U_q$ will be either positive or of regular type. We do however have the following.

\begin{Lem}\label{LemExPosReg}
For any $x\in G_{\nu,q}\backslash G_{\R,q}/U_q$ there exists a positive $x$-representation of regular type. 
\end{Lem}
\begin{proof}
Let $x\in G_{\nu,q}\backslash G_{\R,q}/U_q$, and let $\pi$ be an irreducible admissible $x$-representation of $\mcO_q(G_{\nu}\backslash G_{\R})$. By Lemma \ref{LemTransOrb} we can find $\pi' \in \msP^+$ and a $*$-representation $\pi''$ of $\mcO_q(\U)$ such that $\pi \preccurlyeq \pi'*\pi''$. Clearly $\pi'$ must be an $x$-representation. 
\end{proof}

For $x\in  G_{\nu,q}\backslash G_{\R,q}/U_q$ we let $\mcO_q(O_x^{\nu})$ be the associated fiber of $\mcO_q(G_{\nu}\backslash G_{\R})$, which by Lemma \ref{LemErgoIsFull} may be interpreted as a full spectral $*$-algebra. We let $C_q(O_x^{\nu})$ be the associated C$^*$-envelope. We denote by $L^{\infty}_q(O_{x}^{\nu})$ the von Neumann algebraic envelope of $\mcO_q(O_x^{\nu})$. 

Our goal is to understand the $L^{\infty}_q(O_{x}^{\nu})$ for $x \in G_{\nu,q}\backslash G_{\R,q}/U_q$, together with their invariant states $\int_{O_{x,q}^{\nu}}$. The final result will be presented in the next section. For now, we will develop some of the representation theory of $\mcO_q(Z_{\nu})$ and $\mcO_q(Z_{\nu}^{\reg})$ as preparation.

\subsection{Admissible representations}

Recall the triangular decomposition $\mcO_q(Z_{\nu}^{\reg}) = \msN^-\msA\msN^+$ from \eqref{EqTriangZ}. Recall also the correspondence \eqref{EqCorrAFunc}. 

\begin{Def}
Let $\lambda \in \mbH_{\tau}^{\times}$. We call \emph{highest weight $*$-representation} of $\mcO_q(Z_{\nu}^{\reg})$ at weight $\lambda$ any non-zero $\mcO_q(Z_{\nu}^{\reg})$-module with $*$-compatible pre-Hilbert space structure for which there exists a cyclic weight vector $\xi_{\lambda}$ vanishing under $\msN^+$ and such that 
\[
a_{\omega}\xi_{\lambda} = a_{\omega}(\lambda)\xi_{\lambda},\qquad \omega \in P.
\]
\end{Def}

\begin{Lem}\label{LemIsoHW}
For each $\lambda  \in \mbH_{\tau}^{\times}$ there can exist up to isomorphism at most one highest weight $*$-representation at weight $\lambda$.
\end{Lem} 
\begin{proof}
The argument is standard, see e.g.~ \cite{JL92}. By the triangular decomposition \eqref{EqTriangZ}, one can form the highest weight Verma module $M_{\lambda}\cong \msN^-$ at weight $\lambda$, say with highest weight vector $\xi_{\lambda}\cong 1$. There then exists a unique invariant Hermitian form on $M_{\lambda}$, determined by 
\[
\langle x\xi_{\lambda},y\xi_{\lambda}\rangle = \widetilde{P}(x^*y)(\lambda),\qquad x,y\in \msN^-, 
\]  
with $\widetilde{P} = \iota_{\nu}^{-1}\circ P \circ \iota_{\nu}$ and $P$ the projection map \eqref{EqProjection}. The subspace 
\[
N_{\lambda} = \{v\in M_{\lambda}\mid \langle v,-\rangle = 0\} \subseteq M_{\lambda}
\]
must be a maximal proper submodule by cyclicity of $\xi_{\lambda}$ and invariance of the hermitian form. The hermitian form must then descend to a non-degenerate hermitian form on $V_{\lambda} = M_{\lambda}/N_{\lambda}$. 

If now $W_{\lambda}$ is any highest weight $*$-representation at $\lambda$, we obtain a unique-up-to-scalar non-zero morphism $\pi: M_{\lambda}\rightarrow W_{\lambda}$ by universality. Since the positive-definite form on $W_{\lambda}$ must retract to a multiple of the unique invariant hermitian form on $M_{\lambda}$, it follows that $N_{\lambda}= \Ker(\pi)$, and $\pi$ must descend to an injective map $V_{\lambda}\rightarrow W_{\lambda}$, which is then bijective by cyclicity of the highest weight vector. In particular, the hermitian form on $V_{\lambda}$ must be positive-definite, and the map $V_{\lambda}\cong W_{\lambda}$ a multiple of a unitary.
\end{proof}

In the following we will denote by $V_{\lambda}$ the highest weight $*$-representation at $\lambda$, if it exists.

\begin{Lem}
Let $V_{\lambda}$ be a highest weight $*$-representation of $\mcO_q(Z_{\nu}^{\reg})$ at weight $\lambda\in \mbH_{\tau}^{\times}$. Then $V_{\lambda}$ can be completed to an irreducible $\mcO_q(Z_{\nu})$-representation $(\pi_{\lambda},\Hsp_{\lambda})$ of regular type, and any irreducible $\mcO_q(Z_{\nu})$-representation of regular type arises in this way for a unique $\lambda$. 
\end{Lem} 

\begin{proof}
The fact that $V_{\lambda}$ completes to a bounded $*$-representation of $\mcO_q(Z_{\nu})$ follows by the same argument as in Section \ref{SecAct}. Clearly $\Hsp_{\lambda}$ is then irreducible and of regular type. Conversely, assume $(\pi,\Hsp_{\pi})$ is an irreducible $\mcO_q(Z_{\nu})$-representation of regular type. As in the proof of Lemma \ref{LemIrrRepZZ}, $a_{\rho}$ must necessarily have a bounded and discrete spectrum in any irreducible $*$-representation. Since $a_{\varpi_r}x_r$ is a multiple of $E_r \rhdb a_{\varpi_r} \in \mcO_q(Z_{\nu})$, and $a_{\rho} a_{\varpi_r}x_r = q^{-2(\rho,\alpha_r)}a_{\varpi_r}x_r a_{\rho}$, we must have that $a_{\varpi_r} x_r \xi = 0$ for all $r$ for some non-zero weight vector $\xi$, say at weight $\lambda$. If $\pi$ is of regular type, we must also have $x_r\xi= 0$ for all $r$. Hence $\xi$ is a highest weight vector. By irreducibility, $\xi$ must be cyclic for $\Hsp_{\pi}$ and hence also for $V_{\pi}$, so $V_{\pi}\cong V_{\lambda}$. 
\end{proof}

Now clearly any $V_{\lambda}$ will factor over a character of the center of $\mcO_q(Z_{\nu})$, and in particular over a character $x$ of $\mcO_q(Z_{\nu}/ U)$. We call $x = x_{\lambda}$ the \emph{associated central character}. 

\begin{Def}
We call $\lambda \in \mbH_{\tau}^{\times}$ a \emph{highest weight} (associated to the twisting datum $\nu$) if there exists a highest weight $*$-representation $V_{\lambda}$ of $\mcO_q(Z_{\nu}^{\reg})$ at weight $\lambda$. We call $\lambda$ an \emph{admissible highest weight} if moreover the associated central character $x_{\lambda}$ lies in $G_{\nu,q}\backslash G_{\R,q}/U_q$. We then also say that $V_{\lambda}$ is an admissible highest weight $*$-representation. 

We denote
\[
\widetilde{\Lambda}_{\nu} = \{\lambda\textrm{ a highest weight associated to }\nu\}\subseteq \mbH_{\tau}^{\times},
\]
\[
\Lambda_{\nu} = \{\textrm{admissible highest weights}\}\subseteq \widetilde{\Lambda}_{\nu},
\]
and 
\[
\Lambda_{\nu}^{>} = \Lambda_{\nu}\cap \mbH_{\tau}^{>},\qquad \Lambda_{\nu}^{\gg} = \Lambda_{\nu}\cap \mbH_{\tau}^{\gg}.
\]
For $x\in G_{\nu,q} \dbbackslash G_{\R,q}/U_q$ we further define 
\[
\Lambda_{\nu}(x) = \{\lambda \in \widetilde{\Lambda}_{\nu}\mid x = x_{\lambda}\}.
\]
\end{Def}

The following proposition shows that there is no clash of terminology. 

\begin{Prop}
Let $\lambda \in \widetilde{\Lambda}_{\nu}$. Then the following are equivalent:
\begin{enumerate}
\item $\Hsp_{\lambda}$ is admissible as a $*$-representation of $\mcO_q(G_{\nu}\backslash G_{\R})$
\item $\lambda \in \Lambda_{\nu}$ 
\item $x_{\lambda} \in G_{\nu,q}\backslash G_{\R,q}/U_q$. 
\end{enumerate}
\end{Prop} 
\begin{proof}
The equivalence between (2) and (3) is simply the definition of $\Lambda_{\nu}$. To see that (1) $\Rightarrow$ (3), we can use Lemma \ref{LemTransOrb} together with the fact that tensoring with an $\mcO_q(U)$-representation does not change the central character by Lemma \ref{LemInvPartZ}. Together with Lemma \ref{LemErgoIsFull}, the latter lemma also gives the reverse implication (3) $\Rightarrow$ (1).
\end{proof}

\subsection{Properties of the set $\Lambda_{\nu}$ of admissible highest weights}

We are not yet able to give a complete description of the sets $\Lambda_{\nu}$ for general $\mfu$ and general $\nu$, but see Section \ref{SecExist} for some partial results. For now, we derive some structural properties of these sets.

\begin{Lem}\label{LemCloseGauge}
For any $x\in G_{\nu,q}\dbbackslash G_{\R,q}/U_q$ the set $\Lambda_{\nu}(x)$ is closed under multiplication by $\mbH_{\tau}^{\gauge}$. 
\end{Lem}
\begin{proof}
By the descent of the coaction $\rho_{\nu}$ to $\mcO_q(O_{x}^{\nu})$, we have an action of the maximal torus $T$ on $\mcO_q(O_x^{\nu})$. More precisely, the character 
\[
e^{iH}: U_{\varpi}(\xi,\eta) \mapsto e^{i (H,\wt(\eta))},\qquad H \in \mfa
\]
acts on $a_{\varpi}$ by multiplication with $e^{i (H-\tau(H),\varpi)}$, from which the lemma follows. 
\end{proof}

From Corollary \ref{CorOrbitHC}, Lemma \ref{LemExPosReg} and the fact that 
\[
\widetilde{\chi}_{\HC}(z_{\varpi})(\lambda)\xi_{\lambda} = \widetilde{\chi}_{\HC}(z_{\varpi})\xi_{\lambda} = x(z_{\varpi})\xi_{\lambda},\qquad \varpi \in P^{\tau,+},
\] 
we obtain the following. 

\begin{Cor}\label{CorTwoWWW}
\begin{enumerate}
\item For $x\in G_{\nu,q}\dbbackslash G_{\R,q}/U_q$ the set $\Lambda_{\nu}(x)$ consists of finitely many $\mbH_{\tau}^{\gauge}$-orbits. 
\item $\Lambda_{\nu}$ is contained in the $W_{\nu}$-orbit of $\Lambda_{\nu}^{>}$ under the $\cdot_{\epsilon,q}$-action.
\end{enumerate}
\end{Cor} 

\begin{Cor}\label{CorInt}
Let $\mcH$ be a bounded representation of $\mcO_q(O_x^{\nu})$, and assume that all $a_{\varpi_r}^*a_{\varpi_r}$ have purely discrete spectrum not containing zero. Then there exists a measurable field of multiplicity Hilbert spaces $\lambda \mapsto \mcG_{\lambda}$ over $\Lambda_{\nu}(x)$ such that 
\[
\mcH \cong \int_{\Lambda_{\nu}(x)}^{\oplus} \rd \lambda \; \mcH_{\lambda}\otimes \mcG_{\lambda}. 
\]
\end{Cor}
Here we view $\Lambda_{\nu}(x)$ as a closed subset of $\C^{|I|}$.
\begin{proof}
Let $u_s$ be the polar part of $\pi_{\mcH}(a_{\varpi_s})$ for $s\in I^{*}$. Then the $u_s$ commute amongst themselves and with the operators in $\pi_{\mcH}(\mcO_q(O_x^{\nu}))$. It follows that 
\[
\mcH \cong \int_{\T^{|I^*|}}^{\oplus} \rd \theta \; \mcH_{\theta}. 
\]
where each $\mcH_{\theta}$ is a representation of $\mcO_q(O_x^{\nu})$ still satisfying the given assumption. However, as now the polar part of $a_{\varpi_s}$ for $s\in I^*$ is a fixed scalar, it follows from Corollary \ref{CorTwoWWW} that each $\mcH_{\theta}$ must be a direct sum with components from a finite family of highest weight representations $\Hsp_{\lambda}$. The corollary is now clear. 
\end{proof}

In the upcoming Theorem \ref{TheoBijCorr}, we aim to provide a more refined version of Corollary \ref{CorTwoWWW}.(2). We first need to examine some low-rank cases.

Recall $\mcO_q(H_2)$ with its $\mcO_q(SU(2))$-coaction introduced in Section \ref{ExaRank1}. 

\begin{Lem}\label{LemUniH2}
Consider the localisation $\mcO_q(H_2^{\reg}) := \mcO_q(H_2)[z^{-1}]$. Then $\mcO_q(H_2)\subseteq \mcO_q(H_2^{\reg})$, and $\mcO_q(H_2^{\reg})$ is the universal $*$-algebra generated by elements $a^{\pm 1},x,D$ with $a,D$ selfadjoint, with $D$ central and with $ax = q^{-2}xa$ and 
\[
xx^* - q^{-2}x^*x = \frac{D a^{-2} - 1}{q-q^{-1}}.
\]
The precise correspondence is given by 
\[
z = a,\qquad v = q^{-1/2}(q^{-1}-q)xa.
\]
Moreover, the infinitesimal action of $U_q'(\mfsu(2))$ on $\mcO_q(H_2)$ extends to an inner action on $\mcO_q(H_2^{\reg})$, determined by 
\begin{equation}\label{EqInfActH2}
K \rhdb X =  a^{-1}Xa,\qquad E \rhdb X = x X -a^{-1}Xax,\qquad E^* \rhdb X = x^*X -a^{-1}Xax^*.
\end{equation}
\end{Lem}
\begin{proof}
The fact that $\mcO_q(H_2)\subseteq \mcO_q(H_2^{\reg})$ is straightforward, as is the fact that $\mcO_q(H_2^{\reg})$ is presented by the above generators and relations under the given correspondence. 

Consider now $\mcO_q(Z_{\nu})$ in rank 1 for $\nu = (\id,\epsilon)$ with $\epsilon$ a formal (central) variable. The fact that $\epsilon$ is here a central abstract element, and not a scalar, is inessential to apply the theory from the previous sections. In particular, we know from \cite[Lemma 2.20]{DCM18} that the generating matrix $Z_{\varpi}$ of $\mcO_q(Z_{\nu})$ satisfies the reflection equation by \cite[Lemma 2.20]{DCM18}. We hence obtain a $U_q'(\mfsu(2))$-linear $*$-homomorphism 
\begin{equation}\label{EqMapH2}
\pi: \mcO_q(H_2) \rightarrow \mcO_q(Z_{\nu}),\qquad Z \mapsto Z_{\varpi}.
\end{equation}
We claim that
\begin{equation}\label{EqFormulaPi}
\pi(a) = a_{\varpi},\quad \pi(x) = x_{\alpha},\qquad \pi(D) = \epsilon,
\end{equation}
Then from the universal relations of source and target it is clear that $\pi$ must be an isomorphism. Moreover, it will then follow from the $U_q'(\mfsu(2))$-modularity of the inclusion \eqref{EqInclusioninu} that indeed the infinitesimal $U_q'(\mfsu(2))$-module $*$-algebra structure on $\mcO_q(H_2)$ extends to $\mcO_q(H_2^{\reg})$ by means of the given inner actions. 

To see that \eqref{EqFormulaPi} holds, note that automatically $\pi(a) = \pi(z) = a_{\varpi}$ by definition of $a_{\varpi}$. We then also find $\pi(v) = q^{-1/2}(q^{-1}-q)x_{\alpha}a_{\varpi}$ upon writing out $\pi(E \rhdb z) = E\rhdb \pi(z)$. Hence $\pi(x) = x_{\alpha}$. Now writing out $\pi(E \rhdb w) = E \rhdb \pi(w)$, we also find that 
\begin{equation}\label{EqFormu}
\pi(u) = \pi(z) - q^{1/2} \pi(E \rhdb w) = a_{\varpi} - (q^{-1}-q) (x_{\alpha}a_{\varpi}x_{\alpha}^* - x_{\alpha}^* a_{\varpi}x_{\alpha}) = \epsilon a_{\varpi}^{-1}  + q(q^{-1}-q)^2x_{\alpha}x_{\alpha}^*a_{\varpi}.
\end{equation}
Hence
\[
\pi(D) = \pi(uz - q^{-2}vw)  =\epsilon  + q(q^{-1}-q)^2x_{\alpha}x_{\alpha}^*a_{\varpi}^2 - q^{-3}(q^{-1}-q)^2 x_{\alpha}a_{\varpi}^2x_{\alpha}^* = \epsilon
\]
\end{proof}

\begin{Cor}\label{LemRank1CompZ}
Consider $\mfu = \mfsu(2)$ with simple root $\alpha$ and fundamental weight $\varpi = \frac{1}{2}\alpha$. Let $\nu = (\id,\epsilon)$ be a twisting datum with $\epsilon_Q(\alpha) = \epsilon \in \R$. Then there exists a unique $\mcO_q(SU(2))$-colinear $*$-homomorphism 
\[
\pi: \mcO_q(H_2) \rightarrow \mcO_q(Z_{\nu})
\]
such that $z \mapsto a_{\varpi}$. The kernel of this map is the ideal generated by $D-\epsilon$.  
\end{Cor}

From Proposition \ref{PropRank1Tensor} we now obtain the following. 

\begin{Cor}\label{CorRank1Comp}
Consider $\mfu = \mfsu(2)$ with twisting datum $\epsilon \in \R$. Let $W_{\epsilon}$ be the Weyl group, so $W_0$ is trivial and $W_{\epsilon} = W = \{1,s\}$ if $\epsilon\neq0$. Identify $\mbH_{\id}^{\times}\cong \R^{\times}$.
\begin{itemize}
\item If $\epsilon >0$, then $\Lambda_{\nu} = \Lambda_{\nu}^{>} = q^{-\N}\epsilon^{1/2} $ while $\widetilde{\Lambda}_{\nu} = \pm \Lambda_{\nu}$. Moreover, $(s \cdot_{\epsilon,q}\widetilde{\Lambda}_{\nu}) \cap \widetilde{\Lambda}_{\nu} = \emptyset$.
\item If $\epsilon =0$, then $\Lambda_{\nu} = \Lambda_{\nu}^{>} =\R_{>0}$ while $\widetilde{\Lambda}_{\nu} = \R^{\times}$.
\item If $\epsilon<0$, then $\Lambda_{\nu} = \widetilde{\Lambda}_{\nu} =  \R^{\times}$ with $\Lambda_{\nu}^{>} = \R_{>0}$ and $s\cdot_{\epsilon,q} \Lambda_{\nu}^{>}= -\Lambda_{\nu}^{>}$. 
\end{itemize}
\end{Cor}

\begin{Lem}\label{LemTypeA1A1}
Consider $\mfu = \mfsu(2) \oplus \mfsu(2)$ with roots $\alpha_1,\alpha_2$. Let $\nu  = (\tau,\epsilon)$ be the ungauged twisting datum given by $\tau(1) = 2$ and $\epsilon = \epsilon_1 = \epsilon_2 \geq 0$. 
\begin{enumerate}
\item[(1)] If $\epsilon>0$, then $\lambda\in \Lambda_{\nu}$ if and only if $\lambda_P(\varpi_1) = q^{\frac{1-n}{2}}\epsilon^{1/2}e^{i\theta}$ for $n \in \N$ and $\theta\in \R$.
\item[(2)] If $\epsilon =0$, then  $\Lambda_{\nu} = \mbH_{\tau}^{\times} \cong \C^{\times}$. 
\end{enumerate}
\end{Lem} 
\begin{proof}
See Appendix, Theorem \ref{TheoAppA1A1}.
\end{proof}

\begin{Cor}\label{CorTypeA1A1}
In the situation of Lemma \ref{LemTypeA1A1} with $\epsilon \neq 0$, one has $W_{\nu} = \{1,\hat{s}\}$ with $\hat{s} = s_1s_2$ and $(\hat{s}\cdot_{\epsilon,q}  \Lambda_{\nu}) \cap \Lambda_{\nu} = \emptyset$.
\end{Cor}

\begin{Lem}\label{LemTypeA2}
Consider $\mfu = \mfsu(3)$ with roots $\alpha_1,\alpha_2$. Assume that $\tau(1) = 2$, and let $\epsilon = \epsilon_1 =\epsilon_2 \geq 0$ be a twisting datum. 
\begin{enumerate}
\item[(1)] If $\epsilon >0$, then $\lambda\in \Lambda_{\nu}$ if and only if $\lambda_P(\varpi_1) = q^{\frac{3-n}{2}} \epsilon e^{i\theta}$ for $n \in \N$ and $\theta \in \R$.
\item[(2)] If $\epsilon =0$, then  $\Lambda_{\nu} = \mbH_{\tau}^{\times}\cong \C^{\times}$. 
\end{enumerate}
\end{Lem} 
\begin{proof}
See Appendix, Theorem \ref{TheoClassIrrepA2}.
\end{proof}

\begin{Cor}\label{CorTypeA2}
In the situation of Lemma \ref{LemTypeA2} with $\epsilon \neq 0$, one has $W_{\nu} = \{1,\hat{s}\}$ with $\hat{s} = s_1s_2s_1$ and $(\hat{s}\cdot_{\epsilon,q} \Lambda_{\nu}) \cap \Lambda_{\nu} = \emptyset$.
\end{Cor}

We now consider $\mfu$ of general rank. Recall again the $\mcO_q(SU(2))$-coactions introduced in Section \ref{ExaRank1}.

\begin{Lem}\label{LemRank1Prep}
For $r= \tau(r)$, there exists a unique $U_{q_r}(\mfsu(2))$-equivariant $^*$-homomorphism
\[
\kappa_r: \mcO_{q_r}(H_{2}) \rightarrow \mcO_q(Z_{\nu}),\quad \begin{pmatrix} z & w \\ v & u \end{pmatrix} \mapsto \begin{pmatrix} z_r & w_r\\ v_r & u_r\end{pmatrix} 
\]
such that $z_r = a_{\varpi_r}$. Then 
\[
v_r = q_r^{-1/2}(q_r^{-1}-q_r)x_ra_{\varpi_r},
\]
and furthermore $\kappa_r(D) = D_r := \epsilon_r a_{2\varpi_r-\alpha_r}$ and
\[
\kappa_r(T) =  T_r := q_r^{-2}(q_r-q_r^{-1})^2 x_r^*x_ra_{\varpi_r} + q_r^{-1}a_{\varpi_r} + \epsilon_rq_ra_{\varpi_r-\alpha_r}. 
\]
\end{Lem}
\begin{proof}
The existence of a $*$-homomorphism
\[
\kappa_r: \mcO_{q_r}(H_{2}^{\reg}) \rightarrow \mcO_q(Z_{\nu}^{\reg}),\quad z \mapsto a_{\varpi_r},\quad v \mapsto q_r^{-1/2}(q_r^{-1}-q_r)x_ra_{\varpi_r},\quad D \mapsto \epsilon_r a_{2\varpi_r-\alpha_r}
\]
follows immediately from the defining relations in Lemma \ref{LemUniH2}. This map is then seen to be $U_q'(\mfsu(2))$-linear, and hence $U_q(\mfsu(2))$-linear, by the explicit formula \eqref{EqInfActH2}. Since $z$ generates $\mcO_q(H_2)$ as an $U_q(\mfsu(2))$-module $*$-algebra, we obtain that $\kappa_r$ maps $\mcO_q(H(2))$ equivariantly into $\mcO_q(Z_{\nu})$, and that moreover this is the unique such map sending $z$ to $a_{\varpi_r}$. The formula for $\kappa_r(T)$ now follows by an easy explicit computation, recycling for example the formula \eqref{EqFormu} (upon replacing $\epsilon$ by $\epsilon_r a_{2\varpi_r-\alpha_r}$).
\end{proof}

\begin{Lem}\label{LemRank1Prep2}

For $r \neq \tau(r)$, there exists a unique $U_{q_r}(\mfsu(2))$-equivariant $*$-homomorphism
\[
\kappa_r: \mcO_{q_r}(V_{\R}) \rightarrow \mcO_q(Z_{\nu}),\quad \begin{pmatrix} a & b \\ c & d \end{pmatrix} \mapsto \begin{pmatrix} a_r & b_r\\ c_r & d_r\end{pmatrix} 
\]
such that $c_r = a_{\varpi_r}$. Then 
\[
d_r = q_r^{-1/2}(q_r^{-1}-q_r)a_{\varpi_r}x_r^* 
\] 
and $\kappa_r(D) = D_r$ with 
\begin{equation}\label{EqFormDV}
D_r := (1+q_r^{-1}(q_r^{-1}-q_r)^2x_r^*x_r)a_{\varpi_r+\varpi_{\tau(r)}}. 
\end{equation}
\end{Lem}
\begin{proof}
It is immediate from the defining relations of $\mcO_{q_r}(V_{\R})$ and $\mcO_q(Z_{\nu}^{\reg})$ that there exists a unique $*$-homomorphism
\[
\kappa_r: \mcO_{q_r}(V_{\R}) \rightarrow \mcO_q(Z_{\nu}),\qquad c \mapsto a_{\varpi_r},\quad d \mapsto  q_r^{-1/2}(q_r^{-1}-q_r)a_{\varpi_r}x_r^*.
\]
It is then also immediate that $\kappa_r(D)$ is given by \eqref{EqFormDV}. 

To see that $\kappa_r$ is equivariant (and is then uniquely determined by its value on $c$), we note that
\[
K\rhdb \begin{pmatrix} c & d \end{pmatrix} = \begin{pmatrix} q_rc & q_r^{-1}d\end{pmatrix},\quad E\rhdb \begin{pmatrix} c & d \end{pmatrix} = \begin{pmatrix} 0 & q_r^{1/2}c\end{pmatrix}, \quad E^*\rhdb \begin{pmatrix} c & d \end{pmatrix} = \begin{pmatrix} q_r^{1/2}d & 0\end{pmatrix}.
\]
It is straightforwardly verified that $\kappa_r$ is $U_{q_r}(\mfsu(2))$-equivariant on the two-dimensional module spanned by $c,d$. Hence $\kappa_r$ is $U_{q_r}(\mfsu(2))$-equivariant as $c,d$ generate $\mcO_{q_r}(V_{\R})$ as an $U_{q_r}(\mfsu(2))$-module $*$-algebra.
\end{proof}

We introduce the following $*$-subalgebras on $\mcO_q(Z_{\nu})$. 

\begin{Def}
We let $\msA_r$ be the $*$-algebra generated by the range of $\kappa_r$ in Lemma \ref{LemRank1Prep} and Lemma \ref{LemRank1Prep2}. We let $\msA_r^{\loc}$ be the $*$-algebra generated by $\msA_r$ and the $a_{\alpha_r}^{\pm 1},a_{\varpi_r}^{\pm 1}$. For $r\neq \tau(r)$, we let $\msA_{r,\tau(r)}^{\loc}$ be the $*$-algebra generated by $\msA_r^{\loc}$ and $\msA_{\tau(r)}^{\loc}$.  
\end{Def}

Note that for $r\in I^{\tau}$, $\msA_r^{\loc}$ is simply the unital $*$-algebra generated by $x_r,a_{\alpha_r}^{\pm 1}$ and $a_{\varpi_r}^{\pm 1}$, corresponding up to the Cartan part with the rank one case. When $r\notin I^{\tau}$, $\msA_{r,\tau(r)}^{\loc}$ is the unital $*$-algebra generated by the $a_{\varpi_r}^{\pm 1},a_{\alpha_r}^{\pm 1}$, $x_r$ and $x_{\tau(r)}$, corresponding up to the Cartan part with a rank two case with non-trivial involution.

Consider now the `big cell' $*$-representation $(\mbs,\mcS)$ of the $*$-algebra $\mcO_q(SU(2))$, given by $\mcS = l^2(\N)$ and 
\[
\mbs(a)e_n = (1-q^{2n})^{1/2}e_{n-1},\qquad \mbs(b)e_n = q^n e_n,
\]
as well as  the $*$-characters
\[
\chi_{\theta}: a \mapsto e^{i\theta},\qquad b \mapsto 0
\]
for $\theta \in \R/{2\pi}\Z$, cf. Section \ref{ExaRank1}.

Let
\[
\pi_r: \mcO_q(U) \rightarrow \mcO_{q_r}(SU(2)_r)
\]
be the natural restriction map obtained by restricting $\mcO_q(U)$ to functionals on the isomorphic copy of $U_{q_r}(\mfsu(2))$ generated by $E_r,F_r,K_{\alpha_r} \in U_q(\mfu)$. By these factorisations, the above $*$-representations can be interpreted as $*$-representations $(\mbs_r,\mcS_r)$ and $\chi_{\theta}^{(r)}$ of $\mcO_q(U)$. 

For $\Hsp_{\lambda}$ an irreducible $*$-representation of regular type of $\mcO_q(G_{\nu}\backslash G_{\R})$, the following theorem describes its fusion rules with the $*$-representations $\mbs_r$ and $\chi_{\theta}^{(r)}$. We will use the following notation: for $r\in I^*\cup \tau(I^*)$ and $\theta \in \R$ we define $\eta_r(\theta) \in \mbH_{\tau}^{\gauge}$ by $\eta_r(\theta)_s = 1$ for $s\notin \{r,\tau(r)\}$ and
\[
\eta_r(\theta)_r = e^{i\theta},\quad \eta_r(\theta)_{\tau(r)} = e^{-i\theta}.
\] 

\begin{Theorem}\label{TheoFusion}
Let $x \in G_{\nu,q}\backslash G_{\R,q}/U_q$, and let $\lambda \in \Lambda_{\nu}(x)$. We denote by $\Hsp$ some generic multiplicity Hilbert space.
\begin{enumerate}
\item If $r \in I^{\tau}$ and $\lambda_P(\alpha_r) \epsilon_r \geq 0$, then 
\[
\pi_{\lambda}*\mbs_r \cong \Hsp \otimes \pi_{\lambda},\qquad \pi_{\lambda}*\chi_{\theta}^{(r)}\cong \pi_{\lambda}.
\] 
\item If $r \in I^{\tau}$ and $\lambda_P(\alpha_r)\epsilon_r < 0$, then also $s_r\cdot_{\epsilon,q} \lambda \in \Lambda_{\nu}(x)$ and
\[
 \pi_{\lambda}*\mbs_r \cong \Hsp \otimes (\pi_{\lambda}  \oplus \pi_{s_r\cdot_{\epsilon,q} \lambda}),\qquad \pi_{\lambda}*\chi_{\theta}^{(r)}\cong \pi_{\lambda}.
\] 
\item If $r\in I^*\cup \tau(I^*)$, then 
\[
\pi_{\lambda}*\mbs_r\cong \Hsp \otimes  \int_{\T} \pi_{\lambda \eta_r(\theta)} \rd\theta,\qquad \pi_{\lambda}*\chi_{\theta}^{(r)}\cong \pi_{\lambda \eta_r(\theta)}.
\]
\end{enumerate} 
\end{Theorem}

\begin{proof}
The fusion rules with the $\chi_{\theta}^{(r)}$ are immediate, so we concentrate on the fusion rules with the $\mbs_r$.

In all cases, we have that 
\[
\pi_{\lambda}* \mbs_r(a_{\varpi_s}) = a_{\varpi_s}\otimes 1,\qquad s\notin \{r,\tau(r)\}
\]
by \eqref{EqTwistedAdj} and the definition of $a_{\varpi_s}$. On the other hand, by Lemma \ref{LemRank1Prep} and Lemma \ref{LemRank1Prep2} we see that $\pi_\lambda \circ \kappa_r$ is a $*$-representation of either $\mcO_{q_r}(H_2)$ or $\mcO_{q_r}(V_{\R})$ under which respectively $z$ and $c$ get sent to operators with discrete spectrum not containing zero. It then follows from the computations in Section \ref{ExaRank1} that $\pi_{\lambda} \circ\kappa_r$ is a direct sum of irreducible $*$-representations satisfying the same property. From the fusion rules computed there, we moreover see that $\pi_{\lambda}* \mbs_r(a_{\varpi_r}a_{\varpi_r}^*)$ must have discrete spectrum not containing $0$. 

Since $\pi_{\lambda}*\mbs_r$ still has the same central character $x$, it follows from Corollary \ref{CorInt} that  $\pi_{\lambda}*\mbs_r$ is a direct integral over amplifications of the $\pi_{\lambda'}$ with $\lambda'$ ranging inside the $W_{\nu} \ltimes \mbH_{\tau}^{\gauge}$-orbit of $\lambda$.

Now if $r \in I^{\tau}$ the spectrum of the $a_{\varpi_s}$ for $s\in I^* \cup \tau(I^*)$ remains the same, so we must actually have a direct sum (possibly with repetition) over the elements in the $W_{\nu}$-orbit of $\lambda$ under the $\cdot_{\epsilon,q}$-action. For each $\lambda'$ which appears, we  have a surjective $\mcO_q(Z_{\nu})$-intertwiner $\Hsp_{\lambda}\otimes \mcS_r \twoheadrightarrow \Hsp_{\lambda'}$. Since $\Hsp_{\lambda} = \msN^- \xi_{\lambda}$, and since $0<q<1$, it follows that $|\lambda_s'|\leq |\lambda_s|$ for $s\neq r$. On the other hand, if this is a strict inequality for some $s\neq r$, the above surjection must be zero on $\xi_{\lambda}\otimes \mcS_r$, and must hence be zero everywhere by using the general identity
\[
\rho_{\nu}(\mcO_q(Z_{\nu}))(1\otimes \mcO_q(U)) = \mcO_q(Z_{\nu})\otimes \mcO_q(U).
\]
It follows that $\lambda_s' = \lambda_s$ for $s\neq r$, and that a highest weight vector $\xi_{\lambda'}$ of $\mcH_{\lambda} \otimes \mcS_r$ lies in $\overline{\msA_r\xi_{\lambda}}  \otimes \mcS_r$, the latter space vanishing under all $x_s$ with $s\neq r$. Since $\lambda_P(\alpha_r)\epsilon_r \geq0$ if and only if $\lambda_P(2\varpi_r - \alpha_r)\epsilon_r \geq0$, it now easily follows from Lemma \ref{LemRank1Prep} and Proposition \ref{PropRank1Tensor} that we must have $\lambda' \in \{\lambda,s_{\alpha_r}\cdot_{\epsilon,q} \lambda\}$, with $\lambda' = \lambda$ if $\lambda_P(\alpha_r)\epsilon_r \geq0$ and both weights $\lambda$ and $s_{r}\cdot_{\epsilon,q} \lambda$ appearing as highest weights if $\lambda_P(\alpha_r)\epsilon_r <0$. 

Assume now that $r\neq \tau(r)$. Then again we can conclude as above that any vector in $\Hsp_{\lambda}\otimes \mcS_r$ vanishing under all $x_s$ with $s\notin \{r,\tau(r)\}$ must lie in $\overline{\msA_{r,\tau(r)}^{\loc}\xi_{\lambda}}\otimes \mcS_r$. If now $\Hsp_{\lambda'}$ appears in the direct integral decomposition of $\Hsp_{\lambda}\otimes \mcS_r$, we must have $|\lambda_r'| \in \{|\lambda_r|,|(\hat{s}_r\cdot_{\epsilon,q} \lambda)_r|\}$ by applying Corollary \ref{CorTwoWWW} to the rank 2 case with non-trivial involution. By Corollary \ref{CorTypeA1A1} and Corollary \ref{CorTypeA2}, we see however that only the value $|\lambda_r|$ can appear. 

Finally, to see that we must have a direct integral over the full circle, we simply note that $\Hsp_{\lambda}$ must decompose as a direct sum of $\msA_r$-representations, each of which has full spectrum with uniform multiplicity for the unitary part of $a_{\varpi_r}$ inside its tensor product with $\mcS_r$  by Lemma \ref{LemRank1Prep2} and the fusion rules described in Theorem \ref{TheoListAllSUq2}.
\end{proof}

\begin{Theorem}\label{TheoBijCorr}
Let $x \in G_{\nu,q}\backslash G_{\R,q}/U_q$. The map 
\[
\mbH^{\gauge}_{\tau}\times W_{\nu}^-  \times \Lambda_{\nu}^{\gg}(x) \rightarrow \Lambda_{\nu}(x), \quad (f,w,\lambda) \mapsto f(w\cdot_{\epsilon,q} \lambda)
\]
is a well-defined bijection.
\end{Theorem}
\begin{proof}
Recall that $W_{\nu}^-$ was defined in Definition \ref{DefWMin}. Using then Lemma \ref{LemDotAction}, we see from Lemma \ref{LemCloseGauge} and Theorem \ref{TheoFusion} that the above map is well-defined. 

To see that it is surjective, pick $\lambda \in \Lambda_{\nu}$ with associated central character $x= x_{\lambda}$. By definition of admissible weight and Lemma \ref{LemTransOrb}, there exists a $\lambda' \in \Lambda_{\nu}^{\gg}(x)$ such that $\pi_{\lambda}$ is weakly contained in $\pi_{\lambda'}*\pi_{\reg}$, where $\pi_{\reg}$ is the regular representation of $\mcO_q(U)$ on $L^2_q(U)$. However, with $w_0 = s_{r_1}\ldots s_{r_N}$ a longest word in $W$, one has by \cite{LS91}
\[
\pi_{\reg} \cong \Hsp \otimes  \mbs_{r_1}*\ldots *\mbs_{r_N}* \int_{T} \chi_{\theta} \rd \theta
\]
with $\Hsp$ a multiplicity Hilbert space and with 
\[
\chi_{\theta}(U_{\varpi}(\xi,\eta)) = e^{2\pi i (\wt(\eta),\theta)} \langle \xi,\eta\rangle,\qquad \theta \in T \cong \mfa/Q^{\vee}.
\]
Hence $\pi_{\lambda}$ must be weakly contained in some $\pi_{\lambda'}*\mbs_{r_1}*\ldots *\mbs_{r_N}* \chi_{\theta}$ with $\lambda'\in \Lambda_{\nu}^{\gg}$. By Theorem \ref{TheoFusion}, there must exist a finite subsequence $r_{i_1},\ldots,r_{i_K}$ in $J_{\nu}^{\tau}$ such that  $\pi_{\lambda} = \pi_{s_{r_{i_K}}\ldots s_{r_{i_1}}\cdot_{\epsilon,q} \lambda'}*\chi_{\theta}$, where in the finite subsequence we may assume that at each step, there is a place in $J_{\nu}^{\tau}$ where the sign is changed. Hence we may assume $s_{r_{i_K}}\ldots s_{r_{i_1}} \in W^-_{\nu}$, and so our map is surjective. 

To see that the map is injective, note that the sign patterns of $\lambda$ and $f(w\cdot_{\epsilon,q} \lambda)$ determine $w$ and $f$ by Theorem \ref{TheoSec}.
\end{proof}

From Theorem \ref{TheoSec} we obtain also the following. 

\begin{Cor}
For each $x \in G_{\nu,q}\backslash G_{\R,q}/U_q$ the set $\Lambda_{\nu}^{\gg}(x)$ consists of at most one point. 
\end{Cor}
We will in the following write 
\[
\Lambda_{\nu}^{\gg}(x) = \{\lambda_x\},
\]
and call $\lambda_x$ the \emph{positive highest weight associated to} $x$. 

The above results on $\Lambda_{\nu}(x)$ were obtained using the coaction by $\mcO_q(U)$. For the next results, we will obtain information about the total set $\Lambda_{\nu}$ using the coaction $\widetilde{\gamma}$ by $U_q(\mfu)$ on $\mcO_q(Z_{\nu})$ defined in \eqref{EqDefGammaTilde}. 

\begin{Lem}\label{LemGoFurther}
If $\lambda \in \Lambda_{\nu}^{\gg}$, then also $\lambda q^{-(\varpi + \tau(\varpi))} \in \Lambda_{\nu}^{\gg}$ for all $\varpi \in P^+$.
\end{Lem}
\begin{proof}
Choosing $\pi_{\lambda}$ with $\lambda\in \Lambda_{\nu}^{\gg}$, we see that $(\pi_{\varpi}\otimes \pi_{\lambda})\circ \widetilde{\gamma}$ contains the highest weight vector $\xi_{\varpi}\otimes \xi_{\lambda}$ at weight $\lambda q^{-(\varpi + \tau(\varpi))}$.
\end{proof} 

Also $\Lambda_{\nu}$ is stable under the above operation, but we need one extra argument. 

\begin{Prop}
If $\lambda \in \Lambda_{\nu}$, then also $\lambda q^{-(\varpi + \tau(\varpi))} \in \Lambda_{\nu}$ for all $\varpi \in P^+$.
\end{Prop}
\begin{proof}
We have to show that $(\pi_{\varpi}\otimes \pi_{\lambda})\circ \widetilde{\gamma}$ is again admissible as an $\mcO_q(G_{\nu}\backslash G_{\R})$-representation. By definition of admissibility, it is actually sufficient to show that 
\[
(\pi_{\varpi}\otimes \pi_{\lambda} \otimes \pi_{\reg})\circ (\id\otimes \rho_{\nu})\widetilde{\gamma}
\]
is admissible for all $\lambda \in \Lambda_{\nu}^{\gg}$, where $\pi_{\reg}$ is the regular representation of $\mcO_q(U)$ on $L^2_q(U)$. However, by the Yetter-Drinfeld compatibility \eqref{EqYDStruct} we have
\[
(\pi_{\varpi} \otimes \id\otimes \id)(\id\otimes \rho_{\nu})\widetilde{\gamma}(f) = U_{\varpi,13}^*((\pi_{\varpi} \otimes \id\otimes \id)(\widetilde{\gamma}\otimes \id)\rho_{\nu}(f))U_{\varpi,13},
\]
from which the result is obvious.
\end{proof}

Combining this with Theorem \ref{TheoBijCorr}, we obtain the following useful corollary extending Lemma \ref{LemGoFurther}.

\begin{Cor}
If $\lambda \in \Lambda_{\nu}^{\gg}$, then also $\lambda q^{-w^{-1}(\varpi + \tau(\varpi))} \in \Lambda_{\nu}^{\gg}$ for all $\varpi \in P^+$ and $w\in W_{\nu}^-$.
\end{Cor}

To end this section, we will show that a certain positivity assumption is unperturbed by equivalence of twisting data. In the next section, we will in fact show that this positivity assumption is satisfied for all twisting data, see Proposition \ref{PropAlwaysPositive}. Let us write again $\lambda = q^{2\gamma}$ for $\gamma \in \mfa^{\tau}$ with $\lambda_r = q^{(2\gamma,\varpi_r)}$.

\begin{Lem}\label{LemInvPosDefDat}
Consider the following property for a reduced twisting datum $\nu$: 
\begin{equation}\label{EqPosnu}
\textrm{For all }\gamma \in \mfa^{\tau} \textrm{ with } \lambda = q^{2\gamma}\textrm{ in } \Lambda_{\nu}^{\gg}\textrm{ one has }(\rho-\gamma,\beta)>0,\quad \forall \beta \in \hat{\Delta}_c^+.\tag{*}
\end{equation}
Then Property \eqref{EqPosnu} is preserved under equivalence of twisting data.  
\end{Lem}  
\begin{proof}
In the proof, let us write $\hat{\Delta}_c = \hat{\Delta}_{\nu,c}$ for emphasis on the $\nu$-dependence. 

Assume that $\nu,\nu'$ are equivalent, and assume that Property \eqref{EqPosnu} holds for $\nu$. By Corollary \ref{CorRefineBS}, we have $\epsilon' = f w\epsilon$ for $w\in W_{\nu}^-$ and $f$ a gauge. By Lemma \ref{LemWeakEq} and Theorem \ref{TheoBijCorr} we then have for $\gamma \in \mfa^{\tau}$ and $\lambda = q^{2\gamma} \in \Lambda_{\nu}^{\gg}$ as above that 
\[
|w \cdot_{\epsilon,q} \lambda| = q^{2(w(\gamma-\rho)+\rho)} \in \Lambda_{\nu'}^{\gg},
\]
establishing a bijection 
\[
\Lambda_{\nu}^{\gg} \rightarrow \Lambda_{\nu'}^{\gg},\quad \lambda \mapsto |w \cdot_{\epsilon,q} \lambda|.
\]
We are now to verify that $(w\rho - w\gamma,\beta) >0$ for all $\beta\in \hat{\Delta}_{\nu',c}^+$. But this is clear since $\hat{\Delta}_{\nu',c}^+ =w \hat{\Delta}_{\nu,c}^+$, cf. \eqref{EqIntPosComp}.
\end{proof}

\section{The von Neumann algebra $L^{\infty}_q(O_x^{\nu})$ and its invariant integral $\int_{O_{x,q}^{\nu}}$}

Fix again an ungauged twisting datum $\nu \in \mbH_{\tau}^{\ungauge}$, and let $x\in G_{\nu,q}\backslash G_{\R,q}/U_q$. Recall from the end of Section \ref{SubSecSpecOrConc} that $\mcO_q(G_{\nu}\backslash G_{\R}) \rightarrow \mcO_q(O_x^{\nu})$ denotes the associated fiber at $x$, with von Neumann algebraic completion $L^{\infty}_q(O_x^{\nu})$ and associated invariant integral $\int_{O_{x,q}^{\nu}}$. Our goal is to present concrete models for the couple 
\[
\left(L^{\infty}_q(O_x^{\nu}),\int_{O_{x,q}^{\nu}}\right).
\]

The structure of $L^{\infty}_q(O_x^{\nu})$ is easily obtained from the results of the previous section.

\begin{Theorem}\label{TheoDecompvN}
Let $\lambda_x \in \Lambda_{\nu}^{\gg}(x)$ be the positive highest weight associated to $x$. Then 
\begin{equation}\label{EqIsovN}
L^{\infty}_q(O_x^{\nu}) \cong \left( \oplus_{w\in W_{\nu}^-} B(\Hsp_{w\cdot_{\epsilon,q} \lambda_x})\right) \overline{\otimes} L^{\infty}(\T^{s}), 
\end{equation}
where $s = \dim \mfa^{-\tau}$.
\end{Theorem} 
\begin{proof}
For $\theta \in \T^s \cong \mfa^{-\tau}/(Q^{\vee})^{-\tau}$, consider $f_{\theta} \in \Hsp_{\tau}^{\gauge}$ given by $(f_{\theta})_P(\omega) = e^{2\pi i (\omega,\theta)}$. From Theorem \ref{TheoBijCorr}, it follows that all $\pi_{f_{\theta}(w\cdot_{\epsilon,q}\lambda_x)}$ are mutually inequivalent irreducible $*$-representations of $\mcO_q(G_{\nu}\backslash G_{\R})$. Hence also the direct integral
\[
\Hsp_x = \oplus_{w\in W_{\nu}^-} \int_{\T^s} \Hsp_{f_{\theta}(w\cdot_{\epsilon,q}\lambda_x)} \rd \theta
\]
is an admissible representation of $\mcO_q(G_{\nu}\backslash G_{\R})$. On the other hand, it follows from Theorem \ref{TheoFusion} that the operations $*\mbs_r$ and $*\chi_{\theta}$ leave $\Hsp \otimes \Hsp_x$ stable (with $\Hsp$ some multiplicity Hilbert space of countable dimension). Since the regular representation $\pi_{\reg}$ of $\mcO_q(U)$ is a product of the various $\mbs_r$ and $\chi_{\theta}$ up to multiplicity, and since $L^{\infty}_q(O^{\nu}_x)$ is the $\sigma$-weak closure of $(\pi_{\lambda_x}*\pi_{\reg})(\mcO_q(O_x^{\nu}))$, it follows that $L^{\infty}_q(O^{\nu}_x)$ is the $\sigma$-weak closure of $\mcO_q(O_x^{\nu})$ in its representation on $\Hsp_x$. Since however the $\pi_{f_{\theta}(w\cdot_{\epsilon,q}\lambda_x)}$ are mutually inequivalent, the latter von Neumann algebra is simply 
\[
L^{\infty}_q(O_x^{\nu}) \cong \oplus_{w\in W_{\nu}^-} \int_{\T^s} B(\Hsp_{f_{\theta}(w\cdot_{\epsilon,q}\lambda_x)})d\theta \cong \left( \oplus_{w\in W_{\nu}^-} B(\Hsp_{w\cdot_{\epsilon,q} \lambda_x})\right) \overline{\otimes} L^{\infty}(\T^{s}).
\]
\end{proof}

Let us now describe in more detail the associated invariant integral $\int_{O_{x,q}^{\nu}}$ where we will use the isomorphism  \eqref{EqIsovN} implicitly. We will also abbreviate
\[
\pi_w = \pi_{w\cdot_{\epsilon,q}\lambda_x},\qquad w\in W_{\nu}^-,
\]
and use the notation
\[
\wt(Z_{\varpi}(\xi,\eta)) = \wt(\eta) - \tau(\wt(\xi)).
\]
Then 
\[
(\id\otimes \chi_{\theta})\rho_{\nu}(f) = e^{2\pi i (\wt(f), \theta)}f,\qquad \theta \in T \cong \mfa/Q^{\vee}.
\]

\begin{Prop}
There exist $c_w =  c_w(\lambda_x) >0$ for $w\in W_{\nu}^-$ such that 
\begin{equation}\label{EqIntf}
\int_{O_{x,q}^{\nu}}f = \sum_{w\in W_{\nu}^-} c_w \Tr(\pi_{w}(f)A_w) \int_{\T^s} e^{2\pi i (\wt(f),\theta)}\rd \theta, 
\end{equation}
where $A_w = |\pi_w(a_{\rho})|$ and where we identify once more $\T^s \cong \mfa^{-\tau}/(Q^{\vee})^{-\tau}$.  
\end{Prop} 
\begin{proof}
From \eqref{EqSigmaForm}, it follows that the modular automorphism of $\mcO_q(O_x^{\nu})$ is implemented as $\Ad(\oplus_{w\in W_{\nu}^-} A_w^{it} \otimes 1)$ on $\oplus_{w\in W_{\nu}^-} \Hsp_{w}\otimes L^2(\T^s)$. Since $\int_{O_{x,q}^{\nu}}$ is faithful, this already implies that
\[
\int_{O_{x,q}^{\nu}}f = \sum_{w\in W_{\nu}^-} c_w \Tr(\pi_{w}(f)A_w) \int_{\T^s} e^{2\pi i (\wt(f),\theta)} g_w(\theta)\rd \theta, 
\]
for $c_w >0$ and certain strictly positive $g_w\in L^1(\T^s,\rd \theta)$. Since however $\int_{O_{x,q}^{\nu}}*\chi_{\theta} = \int_{O_{x,q}^{\nu}}$ for all $\theta \in \mfa/Q^{\vee}$, it follows that $g_w$ is a constant, which can then be absorbed into $c_w$.
\end{proof}

In the remainder of this section, we will determine the $c_w$ more explicitly. To simplify some of the computations in what follows, \emph{we will from now on make the standing assumption that $\epsilon$ is reduced.}

\begin{Lem}
Assume $\varpi \in P^+$ with $\tau(\varpi) = \varpi$. Let $\lambda = \lambda_x = q^{2\gamma} \in \Hsp_{\tau}^{\gg}$, where $\gamma \in \mfa^{\tau}$. Then
\begin{equation}\label{EqIntAlt}
\int_{O_{x,q}^{\nu}}a_{\varpi} = \frac{1}{\dim_q(V_{\varpi})}\frac{\sum_{w\in W_{\nu}} \hsgn(w) \epsilon_Q(\varpi + \rho - w(\varpi + \rho)) q^{2(\gamma - \rho,w(\varpi +\rho))}}{\sum_{w\in W_{\nu}} \hsgn(w) \epsilon_Q(\rho - w\rho)q^{2(\gamma - \rho,w\rho)}}. 
\end{equation}
\end{Lem} 
Note that the sign function $\hsgn$ on $\hat{W}$ is to be taken with respect to the simple reflections $s_{\hat{r}}$ for $\hat{r}\in I/\tau$. Also note that the proof will show that the right hand side, up to the factor $\dim_q(V_{\varpi})^{-1}$, is actually a polynomial in $q$, so that there is no issue evaluating the quotient. 
\begin{proof}
Recall first the projection map $E$ from \eqref{DefE}, the Harish-Chandra homomorphism $\widetilde{\chi}_{\HC}$ from \eqref{EqFormHCZ} and the functions $j_{\varpi}$ introduced in \eqref{EqForj}. Then since $z \xi_{\lambda} = \widetilde{\chi}_{\HC}(z) \xi_{\lambda}$ for $z\in \msZ(\mcO_q(Z_{\nu}))$, we have
\[
\int_{O_{x,q}^{\nu}}a_{\varpi}  =  ((\widetilde{\chi}_{\HC}\circ E)(a_{\varpi}))(\lambda), 
\]
which by \eqref{EqExtraId} becomes
\[
\int_{O_{x,q}^{\nu}}a_{\varpi}  = \frac{1}{\dim_q(V_{\varpi})} \sum_{\alpha \in (Q^+)^{\tau}} j_{\varpi}(\varpi - \alpha) \epsilon_Q(\alpha) q^{2(\gamma - \rho,\varpi - \alpha)}.
\]
It now suffices to show that 
\begin{equation}\label{EqTempExpr}
\sum_{\alpha \in (Q^+)^{\tau}} j_{\varpi}(\omega - \alpha) \epsilon_Q(\alpha) q^{2(\gamma - \rho,\varpi - \alpha)}= \frac{\sum_{w\in W_{\nu}} \hsgn(w) \epsilon_Q(\varpi + \rho - w(\varpi + \rho)) q^{2(\gamma - \rho,w(\varpi +\rho))}}{\sum_{w\in W_{\nu}} \hsgn(w) \epsilon_Q(\rho - w\rho)q^{2(\gamma - \rho,w\rho)}}. 
\end{equation}
Assume first that $\supp(\epsilon) = I$. Then we can take $\widetilde{\epsilon} \in \mbH_{\tau}^{\times}$ such that $\widetilde{\epsilon}_P(\alpha) = \epsilon_Q(\alpha)$ for $\alpha \in Q^+$. We can hence write
\[
\sum_{\alpha \in (Q^+)^{\tau}} j_{\varpi}(\omega - \alpha) \epsilon_Q(\alpha) q^{2(\gamma - \rho,\varpi - \alpha)} =\widetilde{\epsilon}_P(\varpi) \sum_{\alpha \in (Q^+)^{\tau}} j_{\varpi}(\varpi - \alpha) \widetilde{\epsilon}_P(\varpi - \alpha)^{-1} q^{2(\gamma - \rho,\varpi - \alpha)}.
\]
Using now the `twining Weyl character formula' \cite{Jan73}, where we use the form as given in \cite[(5.56)]{FSS96}, we find
\[
\sum_{\alpha \in (Q^+)^{\tau}} j_{\varpi}(\varpi - \alpha) \widetilde{\epsilon}_P(\varpi - \alpha)^{-1} q^{2(\gamma - \rho,\varpi - \alpha)} = \frac{\sum_{w\in W^{\tau}} \hsgn(w) \widetilde{\epsilon}_P(w(\varpi + \rho))^{-1} q^{2(\gamma - \rho,w(\varpi +\rho))}}{\sum_{w\in W^{\tau}} \hsgn(w) \widetilde{\epsilon}_P(w\rho)^{-1}q^{2(\gamma - \rho,w\rho)}},
\]
which after multiplication with $\widetilde{\epsilon}_P(\varpi)$ leads to 
\[
\sum_{\alpha \in (Q^+)^{\tau}} j_{\varpi}(\omega - \alpha) \epsilon_Q(\alpha) q^{2(\gamma - \rho,\varpi - \alpha)}= \frac{\sum_{w\in W^{\tau}} \hsgn(w) \epsilon_Q(\varpi + \rho - w(\varpi + \rho)) q^{2(\gamma - \rho,w(\varpi +\rho))}}{\sum_{w\in W^{\tau}} \hsgn(w) \epsilon_Q(\rho - w\rho)q^{2(\gamma - \rho,w\rho)}}. 
\]
Since all sums concerned are finite, this formula is by continuity still valid for arbitrary $\epsilon \in \mbH_{\tau}^{\ungauge}$. Since moreover any $\varpi + \rho - w(\varpi + \rho)$ contains $\alpha_r$ with a strictly positive coefficient whenever $w$ contains $s_r$ in its reduced expression, it follows that we can take the above sums over $W_{\nu}$, obtaining \eqref{EqTempExpr}.
\end{proof}

Let now $V_{w\cdot_{\epsilon,q}q^{2\gamma}}(\alpha_+)$ be the weight space inside $V_{w\cdot_{\epsilon,q}q^{2\gamma}}$ at weight $(w\cdot_{\epsilon,q} q^{2\gamma}) q^{2\alpha_+}$ for $\alpha \in Q^+$. Comparing \eqref{EqIntAlt} with \eqref{EqIntf}, we find for $\varpi \in P^+$ with $\varpi = \tau(\varpi)$ and $\lambda_x = q^{2\gamma} \in \Lambda_{\nu}^{\gg}$, by using the usual Weyl character formula and noting that $|\epsilon_Q(\rho-w^{-1}\rho)|=1$ for $w\in W_{\nu}$ by the assumption that $\epsilon$ is reduced,
\begin{eqnarray*}
\int_{O_{x,q}^{\nu}}a_{\varpi} &=& \sum_{w\in W_{\nu}^-}\sum_{\alpha_+\in \hat{Q}^+}  c_w \epsilon_Q(\varpi - w^{-1}\varpi) q^{2(\varpi +\rho,\alpha_+ + \rho + w(\gamma -\rho))} \dim(V_{w\cdot_{\epsilon,q} q^{2\gamma}}(\alpha_+)) \\&=&  \frac{1}{\dim_q(V_{\varpi})}\frac{\sum_{w\in W_{\nu}} \hsgn(w) \epsilon_Q(\varpi + \rho - w(\varpi + \rho)) q^{2(\gamma - \rho,w(\varpi +\rho))}}{\sum_{w\in W_{\nu}} \hsgn(w) \epsilon_Q(\rho - w\rho)q^{2(\gamma - \rho,w\rho)}}\\
&=& \frac{\sum_{w\in W} \sgn(w) q^{-2(\rho,w\rho)}}{\sum_{w\in W_{\nu}} \hsgn(w) \epsilon_Q(\rho - w\rho)q^{2(\gamma - \rho,w\rho)}}\frac{\sum_{w\in W_{\nu}} \hsgn(w) \epsilon_Q(\varpi + \rho - w(\varpi + \rho)) q^{2(\gamma - \rho,w(\varpi +\rho))}}{\sum_{w\in W} \sgn(w) q^{-2(\rho,w(\varpi +\rho))}} \\
&=& \frac{e_{\gamma}}{G(\varpi)} H_w^{(\gamma)}(\varpi)
\end{eqnarray*}
where 
\[
G(\varpi) = \sum_{w\in W} \sgn(w) q^{-2(\rho,w(\varpi +\rho))},\qquad e_{\gamma} =\frac{\sum_{w\in W} \sgn(w) q^{-2(\rho,w\rho)}}{\sum_{w\in W_{\nu}} \hsgn(w) \epsilon_Q(\rho - w\rho)q^{2(\gamma - \rho,w\rho)}}
\]
and 
\[
H_w^{(\gamma)}(\varpi) = \sum_{w\in W_{\nu}} \hsgn(w) \epsilon_Q(\varpi + \rho - w(\varpi + \rho)) q^{2(\gamma - \rho,w(\varpi +\rho))}.
\]
On the other hand, by Theorem \ref{TheoSec} we can write 
\[
H_w^{(\gamma)}(\varpi)=  
\sum_{w\in W_{\nu}^-} \epsilon_Q(\varpi -w^{-1}\varpi) \epsilon_Q(\rho - w^{-1} \rho)\hsgn(w)   F_w^{(\gamma)}(\varpi) 
\]
where 
\[
F_w^{(\gamma)}(\varpi)  =  \sum_{v\in W_{\nu}^+} \hsgn(v) q^{2(\gamma - \rho,v^{-1}w^{-1}(\varpi +\rho))}.
\]
Note that in the formula for $F_w^{(\gamma)}$ we could remove a factor $\epsilon_Q(w^{-1}(\varpi + \rho) - v^{-1}w^{-1}(\varpi + \rho))$ by reducedness of $\epsilon$ and the positivity result in  Lemma \ref{LemPosWPlus}.

We have obtained so far the identity
\begin{multline*}
 \sum_{w\in W_{\nu}^-}\sum_{\alpha_+\in \hat{Q}^+}  c_w \epsilon_Q(\varpi - w^{-1}\varpi) q^{2(\varpi +\rho,\alpha_+ + \rho + w(\gamma -\rho))} \dim(V_{w\cdot_{\epsilon,q} q^{2\gamma}}(\alpha_+)) \\
=\frac{e_{\gamma}}{G(\varpi)}\sum_{w\in W_{\nu}^-} \epsilon_Q(\varpi -w^{-1}\varpi) \epsilon_Q(\rho - w^{-1} \rho)\hsgn(w)   F_w^{(\gamma)}(\varpi). 
\end{multline*}
Consider now for a function $g: W_{\nu}^- \rightarrow \{\pm\}$ the set
\[
P_{g}^+ = \{\varpi \in (P^+)^{\tau}\mid \forall w\in W_{\nu}^-: g(w) \epsilon_Q(\varpi -w^{-1}\varpi)>0\}.
\]
Then $P_g^+$ is either empty or contains a translate of $2(P^+)^{\tau}$ in $\mfa^{\tau}$. As all expressions involved in the above identity consist of finite linear combinations of analytic functions on the open unit ball of $\exp(\mfa^{\tau})$ multiplied with functies of the form $q^{\omega}\mapsto q^{(\omega,\kappa)}$ for $\kappa \in \mfa^{\tau}$, it then follows by the above property of each non-empty $P_g^+$ that in fact
\begin{multline*}
\sum_{w\in W_{\nu}^-}\sum_{\alpha_+\in \hat{Q}^+}  c_w \epsilon_Q(\varpi - w^{-1}\varpi)  q^{2(\omega +\rho,\alpha_+ + \rho + w(\gamma -\rho))} \dim(V_{w\cdot_{\epsilon,q} q^{2\gamma}}(\alpha_+))\\
=  \frac{e_{\gamma}}{G(\omega)} \sum_{w\in W_{\nu}^-} \epsilon_Q(\varpi -w^{-1}\varpi) \epsilon_Q(\rho - w^{-1} \rho)\hsgn(w)   F_{w}^{(\lambda)}(\omega)
\end{multline*}
for all $\varpi\in (P^+)^{\tau}$ and all dominant $\omega \in \mfa^{\tau}$. In particular, since the characters $\varpi \mapsto \epsilon_Q(\varpi - w^{-1}\varpi)$ are all distinct by Theorem \ref{TheoSec}, they must be linearly independent, and hence  
\[
c_w \sum_{\alpha_+\in \hat{Q}^+}   q^{2(\omega +\rho,\alpha_+ + \rho + w(\gamma -\rho))} \dim(V_{w\cdot_{\epsilon,q} q^{2\gamma}}(\alpha_+))
=  \frac{e_{\gamma}}{G(\omega)} \epsilon_Q(\rho - w^{-1} \rho)\hsgn(w)  F_{w}^{(\lambda)}(\omega)
\]
for all dominant $\omega \in \mfa^{\tau}$ and all $w\in W_{\nu}^-$. Rearranging, and using that we already know the $c_w$ to be positive, this becomes
\begin{equation}\label{EqFormcDepOm}
c_w \sum_{\alpha_+\in \hat{Q}^+}   q^{2(\omega +\rho,\alpha_+)} \dim(V_{w\cdot_{\epsilon,q} q^{2\gamma}}(\alpha_+))
=  |e_{\gamma}| \frac{\left|\sum_{v\in W_{\nu}^+} \hsgn(v) q^{2(\rho - \gamma,w^{-1}(\omega +\rho)-v^{-1}w^{-1}(\omega+\rho))}\right|}{\left|\sum_{w\in W} \sgn(w) q^{2(\omega + \rho,\rho- w^{-1}\rho)}\right|}
\end{equation}

We now want to take the limit of both sides, putting $\omega = n\rho$ and letting $n$ tending to infinity. However, we first want to exclude the possibility of having singularities on the right hand side which cancel each other out. We therefore state first the following proposition. 

\begin{Prop}\label{PropAlwaysPositive}
Let $\nu$ be a reduced twisting datum, and let $\lambda = q^{2\gamma} \in \Lambda_{\nu}^{\gg}$. Then 
\begin{equation}\label{EqPosBeta}
(\rho- \gamma,\beta)>0 \qquad \textrm{for all }\beta \in \hat{\Delta}_c^{+}.
\end{equation}
\end{Prop}

Before we prove this proposition, let us first show how it can be used to prove the following theorem. 

\begin{Theorem}\label{TheoFormInvFunct}
Assume that $\epsilon\in \mbH_{\tau}^{\red}$ is a reduced twisting datum. Then
\[
\int_{O_{x,q}^{\nu}}f = \frac{\sum_{w\in W_{\nu}^-} \Tr(\pi_{w}(f)A_w) \int_{\T^s} e^{2\pi i (\wt(f),\theta)}\rd \theta}{\sum_{w\in W_{\nu}^-} \Tr(A_w)},\qquad \forall f\in \mcO_q(Z_{\nu}). 
\]
\end{Theorem} 
\begin{proof}
Consider in \eqref{EqFormcDepOm} the case where $\omega = (n-1)\rho$ for $n\in \Z_{>0}$, 
\begin{equation}\label{EqFormcDepOmRho}
c_w \sum_{\alpha_+\in \hat{Q}^+}   q^{2n(\rho,\alpha_+)} \dim(V_{w\cdot_{\epsilon,q} q^{2\gamma}}(\alpha_+))
=  |e_{\gamma}| \frac{\left|\sum_{v\in W_{\nu}^+} \hsgn(v) q^{2n(\rho - \gamma,w^{-1}\rho-v^{-1}w^{-1}\rho)}\right|}{\left|\sum_{w\in W} \sgn(w) q^{2n(\rho,\rho- w^{-1}\rho)}\right|}
\end{equation}
Since 
\[
\dim(V_{w\cdot_{\epsilon,q} q^{2\gamma}}(\alpha_+)) \leq \sum_{\alpha'_+ = \alpha_+} P(\alpha'),
\]
with $P$ Kostant's partition function, we have that
\[
x \mapsto \sum_{\alpha_+ \in \hat{Q}^+} x^{(\rho,\alpha_+)}  \dim(V_{w\cdot_{\epsilon,q} q^{2\gamma}}(\alpha_+))
\]
is uniformly convergent on $(-1,1)$. We hence see that, letting $n$ tend to infinity, the left hand side of \eqref{EqFormcDepOmRho} tends to $c_w$. On the other hand, by Lemma \ref{LemDecompRho} and Proposition \ref{PropAlwaysPositive}, we obtain upon letting $n$ tend to infinity that the right hand side tends to $|e_{\gamma}|$. We thus obtain
\[
c_w = |e_{\gamma}|,
\]
from which the theorem immediately follows. 
\end{proof}

Let us now come back to the proof of Proposition \ref{PropAlwaysPositive}. 

\begin{proof}[Proof (of Proposition \ref{PropAlwaysPositive})]

For $\Gamma' \subseteq \Gamma$ a $\tau$-stable Dynkin subdiagram we let $\nu'$ be the twisting datum on $\Gamma'$ obtained by restriction. For $\lambda \in \mbH_{\tau}^{\gg}$ we let $\lambda' \in \mbH_{\tau'}^{\gg}$ be the unique element such that $\lambda_{P'}(\alpha_r) = \lambda_P(\alpha_r)$ for all $r\in I'$. Then clearly 
\[
\lambda \in \Lambda_{\nu}^{\gg} \quad \Rightarrow \quad \lambda'\in \Lambda_{\nu'}^{\gg}.
\]
It is also clear that if $\beta \in \hat{\Delta}_{\nu,c}^+$ with $\beta$ in the integer span of the $\alpha_{r,+}$ with $r\in I'$, then $\beta \in \hat{\Delta}_{\nu',c}^+$. Since also $(\rho,\beta) =(\rho',\beta)$, it follows that it suffices to prove the proposition in the case where $\beta$ contains all $\alpha_r$ in its decomposition. In particular, we may restrict to the case where $I$ is $\tau$-connected and $\epsilon$ is regular. Finally, by Lemma \ref{LemInvPosDefDat} we may replace $\epsilon$ by any of its representatives within its equivalence class. We will however not always consider the strongly reduced versions, but will take `preferred choices' allowing for our induction process to go through. 

Fix now $\nu$ and $\lambda =q^{2\gamma} \in \Lambda_{\nu}^{\gg}$. It is enough to show that \eqref{EqPosBeta} holds for $\beta$ a simple root in $\hat{\Delta}_c^+$.

Assume first that $\tau = \id$ and $\mfu^{\nu}$ has non-trivial center (the hermitian symmetric case). Taking $\nu$ in its strongly reduced form, we have for this case that any simple compact root in $\hat{\Delta}_c^+$ is actually simple in $\Delta^+$. Hence \eqref{EqPosBeta} holds by the considerations in the beginning of the proof. 

Assume now that $\tau  =\id$ but $\mfu^{\nu}$ is semisimple. We recall that we use the Vogan diagram formalism for reduced symmetric deformation data, coloring the roots with negative value black. For specific implementations of root systems, we will use the conventions as in \cite[Appendix, \S 2, Table 1]{OV90}.

We will prove this case by induction. We start with the cases which will serve as the induction basis.
\begin{itemize}
\item  If our Vogan diagram is given by $\begin{tikzpicture}[scale=.4,baseline=-0.1cm]
\draw (2.8cm,0) circle (.2cm)node[below]{\small $1$};
\draw[fill = black] (3.8cm,0) circle (.2cm)  node[below]{\small $2$};
\draw
(3,-.06) --++ (0.6,0)
(3,+.06) --++ (0.6,0);
\draw
(3.4,0) --++ (60:-.2)
(3.4,0) --++ (-60:-.2);
\end{tikzpicture}$, with short root $\alpha_2$, then the positive compact roots are $\alpha_1,\beta = \alpha_1+2\alpha_2$. Clearly $(\rho- \gamma,\alpha_1) >0$. As $\alpha_1,\beta$ are orthogonal, it further easily follows that for $\lambda = q^{2\gamma} \in \Lambda_{\nu}^{\gg}$ we have
\[
\sum_{v\in W_{\nu}^+} \hsgn(v) q^{2n(\rho - \gamma,\rho-v^{-1}\rho)} = A_n + q^{2n(\rho- \gamma,\beta)(\rho,\beta^{\vee})}B_n 
\]
with $A_n$ convergent and $B_n \rightarrow 1$. Since the left hand side of \eqref{EqFormcDepOmRho} converges (for $w=e$), it follows that we must have $(\rho- \gamma,\beta) >0$.

\item If our Vogan diagram is given by  $\begin{tikzpicture}[scale=.4,baseline=-0.1cm]
\draw[fill=black] (1.8cm,0) circle (.2cm)node[below]{\small $1$}; 
\draw (2cm,0) --+ (0.6cm,0);
\draw (2.8cm,0) circle (.2cm)node[below]{\small $2$};
\draw[fill = black] (3.8cm,0) circle (.2cm)  node[below]{\small $3$};
\draw
(3,-.06) --++ (0.6,0)
(3,+.06) --++ (0.6,0);
\draw
(3.4,0) --++ (60:-.2)
(3.4,0) --++ (-60:-.2);
\end{tikzpicture}$, we note that the simple compact roots are $\alpha_1,\beta = \alpha_2+2\alpha_3$ and $\delta  = \alpha_1+\alpha_2+\alpha_3$, which are again mutually orthogonal. Using the case proven above, we then see that $(\rho- \gamma,\alpha_1)>0$, $(\rho- \gamma,\beta)>0$ and 
\[
\sum_{v\in W_{\nu}^+} \hsgn(v) q^{2n(\rho - \gamma,\rho-v^{-1}\rho)} = A_n + q^{2n(\rho- \gamma,\delta)(\rho,\delta^{\vee})}B_n 
\]
with $A_n$ convergent and $B_n \rightarrow 1$. We can conclude again that also $(\rho-\gamma,\delta)>0$.
\item Finally, if our Vogan diagram is given by the type $G_2$ diagram $\begin{tikzpicture}[scale=.4,baseline=-0.1cm]
\draw (2.8cm,0) circle (.2cm)node[below]{\small $1$};
\draw[fill = black] (3.8cm,0) circle (.2cm)  node[below]{\small $2$};
\draw
(3,-.1) --++ (0.6,0)
(3,-0) --++(0.6,0)
(3,+.1) --++ (0.6,0);
\draw
(3.25,0) --++ (120:-.3)
(3.25,0) --++ (-120:-.3);
\end{tikzpicture}$, our compact positive roots are easily seen to be $\alpha_1$ and $\beta = 3\alpha_1+2\alpha_2$. As these are again orthogonal, the same method as in the previous cases allows to conclude that \eqref{EqPosBeta} holds also for this case.
\end{itemize}

Let us now prove the remaining cases by induction on the rank. We will follow the list\footnote{This table rather contains the Kac diagrams, but one easily passes to the associated Vogan diagram by chopping of the simple root $\alpha_0$.} in \cite[Appendix, \S 2, Table 7, type I]{OV90}.
\begin{itemize}
\item Assume we are in the case of the Vogan diagram encoding the semisimple Lie algebra $\mfso_{2p,2(l-p)+1}$,
\[ 
\begin{tikzpicture}[scale=.4,baseline]
\draw (0cm,0) circle (.2cm) node[above]{\small $1$} ;
\draw (0.2cm,0) -- +(0.2cm,0);
\draw[dotted] (0.4cm,0) --+ (1cm,0);
\draw (1.4cm,0) --+ (0.2cm,0);
\draw (1.8cm,0) circle (.2cm); 
\draw (2cm,0) --+ (0.6cm,0);
\draw[fill = black] (2.8cm,0) circle (.2cm)  node[above]{\small $p$};
\draw (3cm,0) --+ (0.6cm,0);
\draw(3.8cm,0) circle (.2cm);
\draw (4.2cm,0) -- +(0.2cm,0);
\draw[dotted] (4.4cm,0) --+ (1cm,0);
\draw (5.4cm,0) --+ (0.2cm,0);
\draw (5.8cm,0) circle (.2cm);
\draw (6.8cm,0) circle (.2cm) node[above]{\small $\ell$};
\draw
(6,-.06) --++ (0.6,0)
(6,+.06) --++ (0.6,0);
\draw
(6.4,0) --++ (60:-.2)
(6.4,0) --++ (-60:-.2);
\end{tikzpicture},
\]
where we may assume $p\geq 2$ as we have already dealt with the Hermitian case. The only compact simple root which is not simple in $\Delta^+$ is given as $\beta = \alpha_{p-1} + 2\alpha_p + \ldots + 2\alpha_l$. We may hence restrict to the case $p=2$. In this case, we have 
\[ 
\begin{tikzpicture}[scale=.4,baseline=-0.1cm]
\draw (0.6cm,0) circle (.2cm) node[above]{\small $1$};  
\draw (0.8cm,0) --+ (0.6cm,0);
\draw[fill = black] (1.6cm,0) circle (.2cm)  node[above]{\small $2$};
\draw (1.8cm,0) --+ (0.6cm,0);
\draw(2.6cm,0) circle (.2cm);
\draw (2.8cm,0) -- +(0.6cm,0);
\draw (3.6cm,0) circle (.2cm);
\draw(3.8cm,0) --++ (0.4cm,0);
\draw[dotted] (4.4cm,0) --+ (1cm,0);
\draw (5.4cm,0) --+ (0.2cm,0);
\draw (5.8cm,0) circle (.2cm);
\draw (6.8cm,0) circle (.2cm) node[above]{\small $\ell$};
\draw
(6,-.06) --++ (0.6,0)
(6,+.06) --++ (0.6,0);
\draw
(6.4,0) --++ (60:-.2)
(6.4,0) --++ (-60:-.2);
\end{tikzpicture} \cong 
\begin{tikzpicture}[scale=.4,baseline=-0.1cm]
\draw[fill = black] (0.6cm,0) circle (.2cm) node[above]{\small $1$};  
\draw (0.8cm,0) --+ (0.6cm,0);
\draw (1.6cm,0) circle (.2cm)  node[above]{\small $2$};
\draw (1.8cm,0) --+ (0.6cm,0);
\draw[fill = black](2.6cm,0) circle (.2cm);
\draw (2.8cm,0) -- +(0.6cm,0);
\draw (3.6cm,0) circle (.2cm);
\draw(3.8cm,0) --++ (0.4cm,0);
\draw[dotted] (4.4cm,0) --+ (1cm,0);
\draw (5.4cm,0) --+ (0.2cm,0);
\draw (5.8cm,0) circle (.2cm);
\draw (6.8cm,0) circle (.2cm) node[above]{\small $\ell$};
\draw
(6,-.06) --++ (0.6,0)
(6,+.06) --++ (0.6,0);
\draw
(6.4,0) --++ (60:-.2)
(6.4,0) --++ (-60:-.2);
\end{tikzpicture},
\]
which apart from the trivial compact simple roots has also the compact roots $\alpha_1+\alpha_2+\alpha_3$ and $\alpha_2 + 2\alpha_3+\ldots + 2\alpha_l$. Hence we can chop our diagram into two pieces of lower rank to deal with these two components. The first component gives a hermitian symmetric pair, while by induction the second pair can be reduced to the case 
\[
\begin{tikzpicture}[scale=.4,baseline=-0.1cm]
\draw (1.8cm,0) circle (.2cm)node[below]{\small $1$}; 
\draw (2cm,0) --+ (0.6cm,0);
\draw[fill=black] (2.8cm,0) circle (.2cm)node[below]{\small $2$};
\draw (3.8cm,0) circle (.2cm)  node[below]{\small $3$};
\draw
(3,-.06) --++ (0.6,0)
(3,+.06) --++ (0.6,0);
\draw
(3.4,0) --++ (60:-.2)
(3.4,0) --++ (-60:-.2);
\end{tikzpicture} \cong
\begin{tikzpicture}[scale=.4,baseline=-0.1cm]
\draw[fill=black] (1.8cm,0) circle (.2cm)node[below]{\small $1$}; 
\draw (2cm,0) --+ (0.6cm,0);
\draw (2.8cm,0) circle (.2cm)node[below]{\small $2$};
\draw[fill = black] (3.8cm,0) circle (.2cm)  node[below]{\small $3$};
\draw
(3,-.06) --++ (0.6,0)
(3,+.06) --++ (0.6,0);
\draw
(3.4,0) --++ (60:-.2)
(3.4,0) --++ (-60:-.2);
\end{tikzpicture},
\]
which was already dealt with before. 
\item Assume that we are dealing with the Vogan diagram of $\mfsp_{p,l-p}$, given by 
\[ 
\begin{tikzpicture}[scale=.4,baseline]
\draw (0cm,0) circle (.2cm) node[above]{\small $1$} ;
\draw (0.2cm,0) -- +(0.2cm,0);
\draw[dotted] (0.4cm,0) --+ (1cm,0);
\draw (1.4cm,0) --+ (0.2cm,0);
\draw (1.8cm,0) circle (.2cm); 
\draw (2cm,0) --+ (0.6cm,0);
\draw[fill = black] (2.8cm,0) circle (.2cm)  node[above]{\small $p$};
\draw (3cm,0) --+ (0.6cm,0);
\draw(3.8cm,0) circle (.2cm);
\draw (4.2cm,0) -- +(0.2cm,0);
\draw[dotted] (4.4cm,0) --+ (1cm,0);
\draw (5.4cm,0) --+ (0.2cm,0);
\draw (5.8cm,0) circle (.2cm);
\draw (6.8cm,0) circle (.2cm) node[above]{\small $\ell$};
\draw
(6,-.06) --++ (0.6,0)
(6,+.06) --++ (0.6,0);
\draw
(6.3,0) --++ (120:-.2)
(6.3,0) --++ (-120:-.2);
\end{tikzpicture}
\]
with $1 \leq p \leq \lfloor l/2\rfloor$. Now the compact simple root which is not a simple root in $\Delta^+$ is given as $2\alpha_p+ 2\alpha_{p+1}+\ldots + 2\alpha_{l-1}+\alpha_l$. We may hence assume that $p=1$. Now 
\[
\begin{tikzpicture}[scale=.4,baseline=-0.1cm]
\draw[fill = black]  (0.6cm,0) circle (.2cm) node[above]{\small $1$};  
\draw (0.8cm,0) --+ (0.6cm,0);
\draw(1.6cm,0) circle (.2cm)  node[above]{\small $2$};
\draw (1.8cm,0) --+ (0.6cm,0);
\draw(2.6cm,0) circle (.2cm);
\draw (2.8cm,0) -- +(0.6cm,0);
\draw (3.6cm,0) circle (.2cm);
\draw(3.8cm,0) --++ (0.4cm,0);
\draw[dotted] (4.4cm,0) --+ (1cm,0);
\draw (5.4cm,0) --+ (0.2cm,0);
\draw (5.8cm,0) circle (.2cm);
\draw (6.8cm,0) circle (.2cm) node[above]{\small $\ell$};
\draw
(6,-.06) --++ (0.6,0)
(6,+.06) --++ (0.6,0);
\draw
(6.3,0) --++ (120:-.2)
(6.3,0) --++ (-120:-.2);
\end{tikzpicture} \cong 
\begin{tikzpicture}[scale=.4,baseline=-0.1cm]
\draw[fill = black] (0.6cm,0) circle (.2cm) node[above]{\small $1$};  
\draw (0.8cm,0) --+ (0.6cm,0);
\draw[fill = black] (1.6cm,0) circle (.2cm)  node[above]{\small $2$};
\draw (1.8cm,0) --+ (0.6cm,0);
\draw(2.6cm,0) circle (.2cm);
\draw (2.8cm,0) -- +(0.6cm,0);
\draw (3.6cm,0) circle (.2cm);
\draw(3.8cm,0) --++ (0.4cm,0);
\draw[dotted] (4.4cm,0) --+ (1cm,0);
\draw (5.4cm,0) --+ (0.2cm,0);
\draw (5.8cm,0) circle (.2cm);
\draw (6.8cm,0) circle (.2cm) node[above]{\small $\ell$};
\draw
(6,-.06) --++ (0.6,0)
(6,+.06) --++ (0.6,0);
\draw
(6.3,0) --++ (120:-.2)
(6.3,0) --++ (-120:-.2);
\end{tikzpicture},
\]
which has non-trivial compact simple roots given by $\alpha_1+\alpha_2$ and $2\alpha_2+ \ldots +2\alpha_{l-1}+\alpha_l$. Chopping up the diagram, we are by induction reduced to the case $\begin{tikzpicture}[scale=.4,baseline=-0.1cm]
\draw[fill = black] (2.8cm,0) circle (.2cm)node[below]{\small $1$};
\draw (3.8cm,0) circle (.2cm)  node[below]{\small $2$};
\draw
(3,-.06) --++ (0.6,0)
(3,+.06) --++ (0.6,0);
\draw
(3.3,0) --++ (120:-.2)
(3.3,0) --++ (-120:-.2);
\end{tikzpicture}$, which was already dealt with in the induction basis. 
\item Assume that we are dealing with the Vogan diagram of $\mfso_{2p,2(l-p)}$ for $2\leq p \leq \lfloor l/2\rfloor$, 
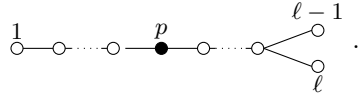
\[
\begin{tikzpicture}[scale=.4,baseline]
\node (v1) at (10,0.8) {};
\node (v2) at (10,-0.8) {};
\draw (0cm,0) circle (.2cm) node[above]{\small $1$} ;
\draw (0.2cm,0) -- +(1cm,0);
\draw (1.4cm,0) circle (.2cm);
\draw (1.6cm,0) -- +(0.2cm,0);
\draw[dotted] (1.8cm,0) --+ (1cm,0);
\draw (2.8cm,0) --+ (0.2cm,0);
\draw (3.2cm,0) circle (.2cm); 
\draw (3.6cm,0) --+ (1cm,0);
\draw[fill = black] (4.8cm,0) circle (.2cm)  node[above]{\small $p$};
\draw (5cm,0) --+ (1cm,0);
\draw (6.2cm,0) circle (.2cm);
\draw (6.4cm,0) -- +(0.2cm,0);
\draw[dotted] (6.6cm,0) --+ (1cm,0);
\draw (7.6cm,0) --+ (0.2cm,0);
\draw (8cm,0) circle (.2cm);
\draw (8.2cm,0) --+ (1.6,0.6);
\draw (8.2cm,0) --+ (1.6,-0.6);
\draw (10cm,0.6) circle (.2cm) node[above]{\small $\ell-1$} ;
\draw (10cm,-0.6) circle (.2cm) node[below]{\small $\ell$};
\end{tikzpicture}.
\]
The compact simple root which is not simple in $\Delta^+$ is given by $\alpha_{p-1}+2\alpha_p+ \ldots + 2\alpha_{l-2}+\alpha_{l-1}+\alpha_l$, so we can restrict to the case $p=2$. In this case 
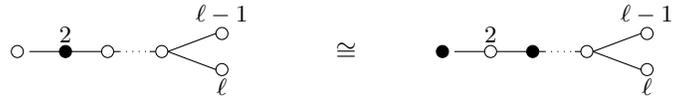
\begin{table}[ht]
\begin{center}
\bgroup
\def\arraystretch{3}
{\setlength{\tabcolsep}{1.5em}
\begin{tabular}{ccc}
\begin{tikzpicture}[scale=.4,baseline=-0.1cm]
\draw (3.2cm,0) circle (.2cm); 
\draw (3.6cm,0) --+ (1cm,0);
\draw[fill = black] (4.8cm,0) circle (.2cm)  node[above]{\small $2$};
\draw (5cm,0) --+ (1cm,0);
\draw (6.2cm,0) circle (.2cm);
\draw (6.4cm,0) -- +(0.2cm,0);
\draw[dotted] (6.6cm,0) --+ (1cm,0);
\draw (7.6cm,0) --+ (0.2cm,0);
\draw (8cm,0) circle (.2cm);
\draw (8.2cm,0) --+ (1.6,0.6);
\draw (8.2cm,0) --+ (1.6,-0.6);
\draw (10cm,0.6) circle (.2cm) node[above]{\small $\ell-1$} ;
\draw (10cm,-0.6) circle (.2cm) node[below]{\small $\ell$};
\end{tikzpicture}
&$\cong$ & 
\begin{tikzpicture}[scale=.4,baseline=-0.1cm]
\draw[fill = black] (3.2cm,0) circle (.2cm); 
\draw (3.6cm,0) --+ (1cm,0);
\draw (4.8cm,0) circle (.2cm)  node[above]{\small $2$};
\draw (5cm,0) --+ (1cm,0);
\draw[fill = black] (6.2cm,0) circle (.2cm);
\draw (6.4cm,0) -- +(0.2cm,0);
\draw[dotted] (6.6cm,0) --+ (1cm,0);
\draw (7.6cm,0) --+ (0.2cm,0);
\draw (8cm,0) circle (.2cm);
\draw (8.2cm,0) --+ (1.6,0.6);
\draw (8.2cm,0) --+ (1.6,-0.6);
\draw (10cm,0.6) circle (.2cm) node[above]{\small $\ell-1$} ;
\draw (10cm,-0.6) circle (.2cm) node[below]{\small $\ell$};
\end{tikzpicture},
\end{tabular}
}
\egroup
\end{center}
\end{table}

where the non-trivial compact simple roots of the latter are given by $\alpha_1+\alpha_2+\alpha_3$ and $\alpha_2+2\alpha_3+\ldots + 2\alpha_{l-2} + \alpha_{l-1}+\alpha_l$. Chopping up the diagram, we can by induction reduce to the case 
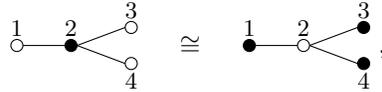
\[
\begin{tikzpicture}[scale=.4,baseline=-0.1cm]
\draw (6.2cm,0) circle (.2cm) node[above]{\small $1$};
\draw (6.4cm,0) -- +(1.4cm,0);
\draw[fill = black] (8cm,0) circle (.2cm) node[above]{\small $2$};
\draw (8.2cm,0) --+ (1.6,0.6);
\draw (8.2cm,0) --+ (1.6,-0.6);
\draw (10cm,0.6) circle (.2cm) node[above]{\small $3$};
\draw (10cm,-0.6) circle (.2cm) node[below]{\small $4$};
\end{tikzpicture}
\quad
\cong 
\quad
\begin{tikzpicture}[scale=.4,baseline=-0.1cm]
\draw[fill = black] (6.2cm,0) circle (.2cm) node[above]{\small $1$};
\draw (6.4cm,0) -- +(1.4cm,0);
\draw (8cm,0) circle (.2cm) node[above]{\small $2$};
\draw (8.2cm,0) --+ (1.6,0.6);
\draw (8.2cm,0) --+ (1.6,-0.6);
\draw[fill = black] (10cm,0.6) circle (.2cm) node[above]{\small $3$};
\draw[fill = black] (10cm,-0.6) circle (.2cm) node[below]{\small $4$};
\end{tikzpicture},
\]
where the simple compact roots of the latter are $\alpha_2,\alpha_1+\alpha_2+\alpha_3, \alpha_1+\alpha_2+ \alpha_4,\alpha_2+\alpha_3+\alpha_4$, all of which live on components of hermitian symmetric type (that have already been dealt with).
\end{itemize} 

This already takes care of all Vogan diagrams of classical type with $\tau = \id$. We now deal with the exceptional cases. 
\begin{itemize}
\item For the diagram 
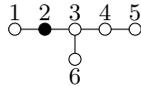
\[
 \begin{tikzpicture}[scale=.4,baseline]
\draw (0cm,0) circle (.2cm) node[above]{\small $1$};
\draw (0.2cm,0) -- +(0.6cm,0);
\draw[fill = black] (1cm,0) circle (.2cm) node[above]{\small $2$};
\draw (1.2cm,0) -- +(0.6cm,0);
\draw (2cm,0) circle (.2cm) node[above]{\small $3$};
\draw (2.2cm,0) -- +(0.6cm,0);
\draw (3cm,0) circle (.2cm) node[above]{\small $4$};
\draw (3.2cm,0) -- +(0.6cm,0);
\draw (4cm,0) circle (.2cm) node[above]{\small $5$};
\draw (2cm,-0.2) -- +(0cm,-0.6);
\draw (2cm,-1) circle (.2cm) node[below]{\small $6$};
\end{tikzpicture}
\]
we note that the root $\alpha_1+2\alpha_2 + 2\alpha_3 + \alpha_4+ \alpha_6$ is compact, hence we can reduce to a diagram of type $D_5$.
\item For the diagram
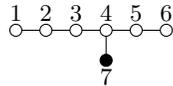
\[
 \begin{tikzpicture}[scale=.4,baseline=-6pt]
\draw (0cm,0) circle (.2cm) node[above]{\small $1$};
\draw (0.2cm,0) -- +(0.6cm,0);
\draw (1cm,0) circle (.2cm) node[above]{\small $2$};
\draw (1.2cm,0) -- +(0.6cm,0);
\draw (2cm,0) circle (.2cm) node[above]{\small $3$};
\draw (2.2cm,0) -- +(0.6cm,0);
\draw (3cm,0) circle (.2cm) node[above]{\small $4$};
\draw (3.2cm,0) -- +(0.6cm,0);
\draw (4cm,0) circle (.2cm) node[above]{\small $5$};
\draw (4.2cm,0) -- +(0.6cm,0);
\draw (5cm,0) circle (.2cm) node[above]{\small $6$};
\draw (3cm,-0.2) -- +(0cm,-0.6);
\draw[fill = black] (3cm,-1) circle (.2cm) node[below]{\small $7$};
\end{tikzpicture}
\]
we have the  compact root $\alpha_2 + 2\alpha_3 + 3\alpha_4 + 2\alpha_5 +\alpha_6+2\alpha_7$, so that we can reduce to the case $E_6$.
\item   For the diagram
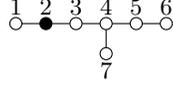
\[
 \begin{tikzpicture}[scale=.4,baseline=-6pt]
\draw (0cm,0) circle (.2cm) node[above]{\small $1$};
\draw (0.2cm,0) -- +(0.6cm,0);
\draw[fill=black] (1cm,0) circle (.2cm) node[above]{\small $2$};
\draw (1.2cm,0) -- +(0.6cm,0);
\draw (2cm,0) circle (.2cm) node[above]{\small $3$};
\draw (2.2cm,0) -- +(0.6cm,0);
\draw (3cm,0) circle (.2cm) node[above]{\small $4$};
\draw (3.2cm,0) -- +(0.6cm,0);
\draw (4cm,0) circle (.2cm) node[above]{\small $5$};
\draw (4.2cm,0) -- +(0.6cm,0);
\draw (5cm,0) circle (.2cm) node[above]{\small $6$};
\draw (3cm,-0.2) -- +(0cm,-0.6);
\draw (3cm,-1) circle (.2cm) node[below]{\small $7$};
\end{tikzpicture}
\]
we have the compact root $2\alpha_2 + 2\alpha_3 + 3\alpha_4+2\alpha_5+\alpha_6+2\alpha_7$,  so that we can reduce to the case $E_6$.
\item For the diagram 
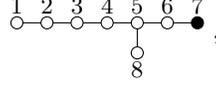
\[
 \begin{tikzpicture}[scale=.4,baseline=-6pt]
\draw (-1cm,0) circle (.2cm) node[above]{\small $1$};
\draw (-0.8cm,0) -- +(0.6cm,0);
\draw (0cm,0) circle (.2cm) node[above]{\small $2$};
\draw (0.2cm,0) -- +(0.6cm,0);
\draw (1cm,0) circle (.2cm) node[above]{\small $3$};
\draw (1.2cm,0) -- +(0.6cm,0);
\draw (2cm,0) circle (.2cm) node[above]{\small $4$};
\draw (2.2cm,0) -- +(0.6cm,0);
\draw (3cm,0) circle (.2cm) node[above]{\small $5$};
\draw (3.2cm,0) -- +(0.6cm,0);
\draw (4cm,0) circle (.2cm) node[above]{\small $6$};
\draw (4.2cm,0) -- +(0.6cm,0);
\draw[fill=black] (5cm,0) circle (.2cm) node[above]{\small $7$};
\draw (3cm,-0.2) -- +(0cm,-0.6);
\draw (3cm,-1) circle (.2cm) node[below]{\small $8$};
\end{tikzpicture},
\]
we have the  compact root $\alpha_2+2\alpha_3 + 3\alpha_4 + 4 \alpha_5 + 3\alpha_6 + 2\alpha_7+2\alpha_8$, so that we can reduce to the case $E_7$.
\item Using the equivalence
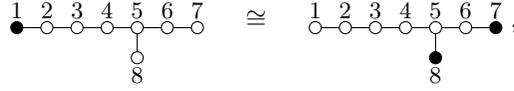
\[
 \begin{tikzpicture}[scale=.4,baseline=0.1pt]
\draw[fill=black] (-1cm,0) circle (.2cm) node[above]{\small $1$};
\draw (-0.8cm,0) -- +(0.6cm,0);
\draw (0cm,0) circle (.2cm) node[above]{\small $2$};
\draw (0.2cm,0) -- +(0.6cm,0);
\draw (1cm,0) circle (.2cm) node[above]{\small $3$};
\draw (1.2cm,0) -- +(0.6cm,0);
\draw (2cm,0) circle (.2cm) node[above]{\small $4$};
\draw (2.2cm,0) -- +(0.6cm,0);
\draw (3cm,0) circle (.2cm) node[above]{\small $5$};
\draw (3.2cm,0) -- +(0.6cm,0);
\draw (4cm,0) circle (.2cm) node[above]{\small $6$};
\draw (4.2cm,0) -- +(0.6cm,0);
\draw (5cm,0) circle (.2cm) node[above]{\small $7$};
\draw (3cm,-0.2) -- +(0cm,-0.6);
\draw (3cm,-1) circle (.2cm) node[below]{\small $8$};
\end{tikzpicture}
\quad
\cong
\quad
 \begin{tikzpicture}[scale=.4,baseline=0.1pt]
\draw (-1cm,0) circle (.2cm) node[above]{\small $1$};
\draw (-0.8cm,0) -- +(0.6cm,0);
\draw (0cm,0) circle (.2cm) node[above]{\small $2$};
\draw (0.2cm,0) -- +(0.6cm,0);
\draw (1cm,0) circle (.2cm) node[above]{\small $3$};
\draw (1.2cm,0) -- +(0.6cm,0);
\draw (2cm,0) circle (.2cm) node[above]{\small $4$};
\draw (2.2cm,0) -- +(0.6cm,0);
\draw (3cm,0) circle (.2cm) node[above]{\small $5$};
\draw (3.2cm,0) -- +(0.6cm,0);
\draw (4cm,0) circle (.2cm) node[above]{\small $6$};
\draw (4.2cm,0) -- +(0.6cm,0);
\draw[fill = black] (5cm,0) circle (.2cm) node[above]{\small $7$};
\draw (3cm,-0.2) -- +(0cm,-0.6);
\draw[fill = black] (3cm,-1) circle (.2cm) node[below]{\small $8$};
\end{tikzpicture},
\]
we have for the latter diagram the compact roots $\alpha_5+\alpha_6+\alpha_7+\alpha_8$ and $\alpha_2+2\alpha_3 + 3\alpha_4+4\alpha_5 + 3\alpha_6+\alpha_7+2\alpha_8$, so that we can chop the diagram up in the pieces $A_4$ and $E_7$, which have already been dealt with.
\item Using the equivalence
\[
\begin{tikzpicture}[scale=.4,baseline]
\draw (0cm,0) circle (.2cm) node[above]{\small $1$};
\draw (0.2cm,0) --+ (0.6cm,0);
\draw (1cm,0) circle (.2cm) node[above]{\small $2$};
\draw (2cm,0) circle (.2cm)  node[above]{\small $3$};
\draw
(1.2,-.06) --++ (0.6,0)
(1.2,+.06) --++ (0.6,0);
\draw
(1.5,0) --++ (60:.2)
(1.5,0) --++ (-60:.2);
\draw (2.2cm,0) --+ (0.6cm,0);
\draw[fill = black] (3cm,0) circle (.2cm) node[above]{\small $4$};
\end{tikzpicture}
\cong
\begin{tikzpicture}[scale=.4,baseline]
\draw[fill = black] (0cm,0) circle (.2cm) node[above]{\small $1$};
\draw (0.2cm,0) --+ (0.6cm,0);
\draw (1cm,0) circle (.2cm) node[above]{\small $2$};
\draw[fill = black] (2cm,0) circle (.2cm)  node[above]{\small $3$};
\draw
(1.2,-.06) --++ (0.6,0)
(1.2,+.06) --++ (0.6,0);
\draw
(1.5,0) --++ (60:.2)
(1.5,0) --++ (-60:.2);
\draw (2.2cm,0) --+ (0.6cm,0);
\draw (3cm,0) circle (.2cm) node[above]{\small $4$};
\end{tikzpicture},
\]
we have for the latter diagram the compact roots $\alpha_1 + \alpha_2 + \alpha_3$ and $2\alpha_2 + 2\alpha_3 + \alpha_4$, so that we can reduce to the cases $B_3$ and $C_3$, which have already been dealt with.
\item Using the equivalence
\[
\begin{tikzpicture}[scale=.4,baseline]
\draw[fill = black] (0cm,0) circle (.2cm) node[above]{\small $1$};
\draw (0.2cm,0) --+ (0.6cm,0);
\draw (1cm,0) circle (.2cm) node[above]{\small $2$};
\draw (2cm,0) circle (.2cm)  node[above]{\small $3$};
\draw
(1.2,-.06) --++ (0.6,0)
(1.2,+.06) --++ (0.6,0);
\draw
(1.5,0) --++ (60:.2)
(1.5,0) --++ (-60:.2);
\draw (2.2cm,0) --+ (0.6cm,0);
\draw (3cm,0) circle (.2cm) node[above]{\small $4$};
\end{tikzpicture}
\cong
\begin{tikzpicture}[scale=.4,baseline]
\draw (0cm,0) circle (.2cm) node[above]{\small $1$};
\draw (0.2cm,0) --+ (0.6cm,0);
\draw[fill = black] (1cm,0) circle (.2cm) node[above]{\small $2$};
\draw (2cm,0) circle (.2cm)  node[above]{\small $3$};
\draw
(1.2,-.06) --++ (0.6,0)
(1.2,+.06) --++ (0.6,0);
\draw
(1.5,0) --++ (60:.2)
(1.5,0) --++ (-60:.2);
\draw (2.2cm,0) --+ (0.6cm,0);
\draw (3cm,0) circle (.2cm) node[above]{\small $4$};
\end{tikzpicture},
\]
we see that the latter has the compact root $2\alpha_2 + \alpha_3$, so that we can reduce to the case $B_2$ already dealt with.
\end{itemize}

This deals with all the cases where $\tau = \id$. 

Assume now that $\tau \neq \id$. If $\epsilon = +$, then we need to prove the result for all $\beta = \alpha_{r,+}$. Since we only need to consider the case where the support of $\beta$ is the whole of $I$, this means that we are in the case of either $A_1\times A_1$ or $A_2$ with non-trivial involution. The theorem follows in this case by the concrete computations in Theorems \ref{TheoAppA1A1} and \ref{TheoClassIrrepA2}. 

Finally, assume that $\tau\neq \id$ and $\epsilon \neq +$. We take the following representative of $\epsilon$ within its equivalence class, listing in the final column the simple $\nu$-compact roots for this preferred representative.

\begin{table}[htpb!]
\begin{center}
\bgroup
\def\arraystretch{2.4}
{\setlength{\tabcolsep}{1.5em}
\begin{tabular}{|c|c|c|c|}
\hline
Diagram & Strong reduction & Preferred representative & Simple $\nu$-compact roots\\ \hline
$A_{2n+1}$& 
\begin{tikzpicture}[scale=.4,baseline]
\node (v1) at (0,0) {};
\node (v2) at (2,0) {};
\node (v3) at (4,0) {};
\node (v4) at (6,0) {};

\draw (0cm,0) circle (.2cm)  node[below]{\small $1$};
\draw (0.2cm,0) -- +(0.2cm,0);
\draw[dotted] (0.4cm,0) --+ (1cm,0);
\draw (1.6cm,0) --+ (0.2cm,0);
\draw (2cm,0) circle (.2cm) node[below]{\small $n$};
\draw (2.2cm,0) -- +(0.6cm,0);
\draw[fill=black] (3cm,0) circle (.2cm) node[below]{\small $n'$};
\draw (3.2cm,0) -- +(0.6cm,0);
\draw (4cm,0) circle (.2cm);
\draw (4.2cm,0) -- +(0.2cm,0);
\draw[dotted] (4.4cm,0) --+ (1cm,0);
\draw (5.6cm,0) --+ (0.2cm,0);
\draw (6cm,0) circle (.2cm);

\draw[<->]
(v1) edge[bend left] (v4);
\draw[<->]
(v2) edge[bend left=50] (v3);
\end{tikzpicture}
& 
\begin{tikzpicture}[scale=.4,baseline]
\node (v1) at (0,0) {};
\node (v2) at (2,0) {};
\node (v3) at (4,0) {};
\node (v4) at (6,0) {};

\draw (0cm,0) circle (.2cm);
\draw (0.2cm,0) -- +(0.2cm,0);
\draw[dotted] (0.4cm,0) --+ (1cm,0);
\draw (1.6cm,0) --+ (0.2cm,0);
\draw (2cm,0) circle (.2cm);
\draw (2.2cm,0) -- +(0.6cm,0);
\draw[fill=black] (3cm,0) circle (.2cm);
\draw (3.2cm,0) -- +(0.6cm,0);
\draw (4cm,0) circle (.2cm);
\draw (4.2cm,0) -- +(0.2cm,0);
\draw[dotted] (4.4cm,0) --+ (1cm,0);
\draw (5.6cm,0) --+ (0.2cm,0);
\draw (6cm,0) circle (.2cm); 

\draw[<->]
(v1) edge[bend left] (v4);
\draw[<->]
(v2) edge[bend left=50] (v3);
\end{tikzpicture}
& 
$\alpha_{\hat{1}},\ldots,\alpha_{\hat{n}},\alpha_{\hat{n}}+\alpha_{n'}$
\\ \hline
$D_{n}, p \textrm{ even} $& 
\begin{tikzpicture}[scale=.4,baseline]
\node (v1) at (8.8,0.8) {};
\node (v2) at (8.8,-0.8) {};
\draw (0cm,0) circle (.2cm) node[above]{\small $1$} ;
\draw (0.2cm,0) -- +(0.6cm,0);
\draw (1cm,0) circle (.2cm);
\draw (1.2cm,0) -- +(0.2cm,0);
\draw[dotted] (1.4cm,0) --+ (1cm,0);
\draw (2.4cm,0) --+ (0.2cm,0);
\draw (2.8cm,0) circle (.2cm); 
\draw (3cm,0) --+ (0.6cm,0);
\draw[fill =black] (3.8cm,0) circle (.2cm)  node[above]{\small $\lfloor \frac{n-p}{2} \rfloor$};
\draw (4cm,0) --+ (0.6cm,0);
\draw (4.8cm,0) circle (.2cm);
\draw (5cm,0) -- +(0.2cm,0);
\draw[dotted] (5.2cm,0) --+ (1cm,0);
\draw (6.2cm,0) --+ (0.2cm,0);
\draw (6.6cm,0) circle (.2cm)  node[above]{\small $n-2$};
\draw (6.8cm,0) --+ (1.6,0.6);
\draw (6.8cm,0) --+ (1.6,-0.6);
\draw (8.6cm,0.6) circle (.2cm)  ;
\draw (8.6cm,-0.6) circle (.2cm) node[below]{\small $n$} ;
\draw[<->]
(v1) edge[bend left] (v2);
\end{tikzpicture}
&
\begin{tikzpicture}[scale=.4,baseline]
\node (v1) at (8.8,0.8) {};
\node (v2) at (8.8,-0.8) {};
\draw (0cm,0) circle (.2cm) node[above]{\small $1$} ;
\draw (0.2cm,0) -- +(0.6cm,0);
\draw (1cm,0) circle (.2cm);
\draw (1.2cm,0) -- +(0.2cm,0);
\draw[dotted] (1.4cm,0) --+ (1cm,0);
\draw (2.4cm,0) --+ (0.2cm,0);
\draw (2.8cm,0) circle (.2cm); 
\draw (3cm,0) --+ (0.6cm,0);
\draw[fill =black] (3.8cm,0) circle (.2cm)  node[above]{\small $p$};
\draw (4cm,0) --+ (0.6cm,0);
\draw[fill =black] (4.8cm,0) circle (.2cm);
\draw (5cm,0) -- +(0.2cm,0);
\draw[dotted] (5.2cm,0) --+ (1cm,0);
\draw (6.2cm,0) --+ (0.2cm,0);
\draw[fill =black] (6.6cm,0) circle (.2cm)  node[above]{\small $n-2$};
\draw (6.8cm,0) --+ (1.6,0.6);
\draw (6.8cm,0) --+ (1.6,-0.6);
\draw (8.6cm,0.6) circle (.2cm)  ;
\draw (8.6cm,-0.6) circle (.2cm) node[below]{\small $n$} ;
\draw[<->]
(v1) edge[bend left] (v2);
\end{tikzpicture}
&
$
\overset{\alpha_1,\ldots,\alpha_{p-1},}{\underset{\alpha_{p}+\alpha_{p+1},\ldots,\alpha_{n-2}+\alpha_{n-1},\alpha_{\hat{n}}}{}}$
\\ \hline
$E_6$ & 
 \begin{tikzpicture}[scale=.4,baseline=-6pt]
\node (v1) at (0,0.2) {};
\node (v2) at (4,0.2) {};
\node (v3) at (1,0.2) {};
\node (v4) at (3,0.2) {};
\draw (0cm,0) circle (.2cm) node[below]{\small $1$};
\draw (0.2cm,0) -- +(0.6cm,0);
\draw (1cm,0) circle (.2cm) node[below]{\small $2$};
\draw (1.2cm,0) -- +(0.6cm,0);
\draw[fill = black]  (2cm,0) circle (.2cm) node[below]{\small $\quad3$};
\draw (2.2cm,0) -- +(0.6cm,0);
\draw (3cm,0) circle (.2cm);
\draw (3.2cm,0) -- +(0.6cm,0);
\draw (4cm,0) circle (.2cm);
\draw (2cm,-0.2) -- +(0cm,-0.6);
\draw(2cm,-1) circle (.2cm) node[below]{\small $6$};
\draw[<->]
(v1) edge[bend left] (v2);
\draw[<->]
(v3) edge[bend left] (v4);
\end{tikzpicture}
&
\begin{tikzpicture}[scale=.4,baseline=-6pt]
\node (v1) at (0,0.2) {};
\node (v2) at (4,0.2) {};
\node (v3) at (1,0.2) {};
\node (v4) at (3,0.2) {};
\draw (0cm,0) circle (.2cm);
\draw (0.2cm,0) -- +(0.6cm,0);
\draw (1cm,0) circle (.2cm);
\draw (1.2cm,0) -- +(0.6cm,0);
\draw[fill = black] (2cm,0) circle (.2cm);
\draw (2.2cm,0) -- +(0.6cm,0);
\draw (3cm,0) circle (.2cm);
\draw (3.2cm,0) -- +(0.6cm,0);
\draw (4cm,0) circle (.2cm);
\draw (2cm,-0.2) -- +(0cm,-0.6);
\draw (2cm,-1) circle (.2cm);
\draw[<->]
(v1) edge[bend left] (v2);
\draw[<->]
(v3) edge[bend left] (v4);
\end{tikzpicture}
&
$\alpha_{\hat{1}},\alpha_{\hat{2}},\alpha_6,\alpha_{\hat{2}}+\alpha_3$\\
\hline
\end{tabular}
}
\egroup
\end{center}
\end{table}

As one sees, the simple $\nu$-compact roots are all supported by twisting data which have been treated already, except for the case $\begin{tikzpicture}[scale=.4,baseline]
\node (v1) at (0,0) {};
\node (v2) at (2.4,0) {};

\draw (0cm,0) circle (.2cm) node[below]{\small $1$};
\draw (0.2cm,0) -- +(1cm,0);
\draw[fill=black] (1.2cm,0) circle (.2cm) node[below]{\small $2$};
\draw (1.4cm,0) -- +(0.8cm,0);
\draw (2.4cm,0) circle (.2cm) node[below]{\small $3$};

\draw[<->]
(v1) edge[bend left=50] (v2);
\end{tikzpicture}
$. By Theorem \ref{TheoAppA1A1}, we have $\gamma = \frac{1-m}{2}(\varpi_1+\varpi_3) + t\varpi_2$ for $m\in \N$ and $t\in \R$. Since $\hat{\Delta}_c^+ = \{\alpha_{\hat{1}},\alpha_{\hat{1}}+\alpha_2\}$ and $W_{\nu}^+ = \{e,\hat{s}_1,s_2\hat{s}_1s_2,\hat{s}_1s_2\hat{s}_1s_2\}$, we easily compute for $\omega = a(\varpi_1+\varpi_3) +b \varpi_2 \in (P^+)^{\tau}$ that, with $z_1 = q^{2a+2}$ and $z_2 = q^{2b+2}$,
\[
\sum_{v\in W_{\nu}^+} \hsgn(v) q^{2(\rho - \gamma,(\omega +\rho)-v^{-1}(\omega+\rho))} = (1-z_1^{m+1})(1-(z_1z_2)^{m+1+2(1-t)}).
\]
By \eqref{EqFormcDepOmRho} for $w = e$, we see that we must have 
\[
2(\rho-\gamma,\alpha_{\hat{1}}+\alpha_2) = m+1+2(1-t) >0,
\]
proving the last remaining case.
\end{proof}

\section{Elements in $G_{\nu,q}\backslash G_{\R,q}/U_q$}\label{SecExist}

In this section, we will give some results, complete or partial, concerning the elements $x\in G_{\nu,q}\backslash G_{\R,q}/U_q$ and the associated $\Lambda_{\nu}(x)$.

\subsection{Correspondence between $G_{\nu,q}\backslash G_{\R,q}/U_q$ and $G_{\nu',q}\backslash G_{\R,q}/U_q$ when $\nu$ and $\nu'$ are equivalent}

Let $\nu,\nu' \in  \mbH_{\tau}^{\ungauge}$ be equivalent ungauged twisting data, say $\epsilon' = g (w\epsilon)$ for $w\in W_{\nu}$ and $g\in \mbH_{\tau}^{>}$. In Lemma \ref{LemWeakEq}, we constructed an isomorphism $\chi_f:  U_q^{\nu'}(\mfu) \rightarrow U_q^{\nu}(\mfu)$ using an $f \in \mbH_{\tau}^{\times}$ such that \eqref{EqSpecCondIso} holds. We can transport $\chi_f$ to a $*$-isomorphism 
\[
\widetilde{\chi}_f: \mcO_q(Z_{\nu'}^{\reg}) \rightarrow \mcO_q(Z_{\nu}^{\reg}),\qquad a_{\omega} \mapsto f_P(\tau(\omega))T_{\omega},\quad x_r \mapsto x_r. 
\]
Since $\chi_f$ moreover respected the $U_q(\mfu)$-action, it follows that also
\[
\mcO_q(Z_{\nu'}) \underset{\widetilde{\chi}_f}{\overset{\cong}{\rightarrow}} \mcO_q(Z_{\nu}),\qquad \msZ(\mcO_q(Z_{\nu'})) \underset{\widetilde{\chi}_f}{\overset{\cong}{\rightarrow}} \msZ(\mcO_q(Z_{\nu})). 
\]
We hence obtain a bijection
\[
\Xi_f: G_{\nu,q}\dbbackslash G_{\R,q}\dbslash U_q \rightarrow G_{\nu',q}\dbbackslash G_{\R,q}\dbslash U_q,\qquad x \mapsto x' = x\circ \widetilde{\chi}_f. 
\]

\begin{Prop}
The map $\Xi_f$ determines a bijection
\[
\Xi_f: G_{\nu,q}\backslash G_{\R,q}/ U_q \rightarrow G_{\nu',q}\backslash G_{\R,q} / U_q.
\]
\end{Prop} 
\begin{proof}
It is sufficient to prove that if $x\in G_{\nu,q}\backslash G_{\R,q}/ U_q$, then $x'\in G_{\nu',q}\backslash G_{\R,q}/ U_q$. 

Pick $\lambda_x \in \Lambda_{\nu}^{\gg}(x)$. Further write $w = uv$ with $u \in W_{\nu}^-$ and $v\in W_{\nu}^+$. By Theorem \ref{TheoBijCorr}, we have that also $u\cdot_{\epsilon} \lambda_x \in \Lambda_{\nu}(x)$, and then $\pi_{u\cdot_{\epsilon} \lambda_x}\circ \widetilde{\chi}_f$ is a highest weight representation of $\mcO_q(Z_{\nu'})$ with central character $x'$ at highest weight $\lambda_{x'}$ with
\[
(\lambda_{x'})_P(\omega) = f_P(\omega)(u\cdot_{\epsilon} \lambda_x)_P(\omega) = f_P(\omega) \epsilon_Q(\omega - u^{-1}\omega) q^{(\omega -u^{-1}\omega,2\rho)}\lambda_P(u^{-1}\omega),\qquad \omega \in P^{\tau}.
\]
However, the condition \eqref{EqSpecCondIso} on $f$ implies that
\[
f_P(\omega) \epsilon_Q(\omega - w^{-1}\omega)  = f_P(\omega)\epsilon_Q(\omega -u^{-1}\omega) \epsilon_Q(u^{-1}\omega - v^{-1}u^{-1}\omega) >0, 
\]
so by Lemma \ref{LemPosWPlus} also 
\[
 f_P(\omega)\epsilon_Q(\omega -u^{-1}\omega) >0,\qquad \omega \in P^{\tau}.
\]
Hence $\lambda_{x'} \in \Lambda_{\nu}^{\gg}(x')$, and $x'\in G_{\nu',q}\backslash G_{\R,q}/ U_q$. 
\end{proof}

\subsection{Characters of $\mcO_q(Z_{\nu})$}

In \cite{BK15b} the authors constructed universal solutions (= universal $K$-matrices) for the $\tau$-modified reflection equation for $U_q(\mfg)$, which by the philosophy in \cite{KoSt09} can be considered as giving characters on $\mcO_q(Z_{\nu})$ - see also \cite{DCM18} where it was shown how these solutions can be slightly modified to produce $*$-characters. One then has the following general result. 

\begin{Prop}
Let $\chi \in \Char_*(\mcO_q(Z_{\nu}))$ be a $*$-character, and let $x = x_{\chi} = \chi_{\mid \msZ(\mcO_q(Z_{\nu}))}$. Let $\pi_{\reg}$ be the regular representation of $\mcO_q(U)$. Then the $*$-representation $\rho_{\nu}^{\chi} := (\chi \otimes \pi_{\reg})\rho_{\nu}$ of $\mcO_q(Z_{\nu})$ on $L^2_q(U)$ is an $x$-representation of regular type.
\end{Prop}
\begin{proof}
Let us first show that $\rho_{\nu}^{\chi}(a_{\rho})$ is not the zero operator. Assume to the contrary that $\rho_{\nu}^{\chi}(a_{\rho})=0$. It is well-known that $\mcO_q(U)$ is a domain \cite[Lemma 9.1.9]{Jos95}, so then all $a_{\varpi}$ with $\varpi \in P^+\setminus \{0\}$ act as zero. Since $\rho_{\nu}^{\chi}$ is  an $\mcO_q(U)$-comodule map, and the $a_{\varpi}$ generate $\mcO_q(Z_{\nu})$ as an $U_q(\mfu)$-module, it follows that $\rho_{\nu}^{\chi}(Z_{\varpi}(\xi,\eta)) = 0$ for all $\varpi\neq 0$, i.e.~ the Haar integral is a character on $\mcO_q(Z_{\nu})$. However, from \cite[Lemma 2.26]{DCM18} it would then follow that the Haar integral is also a character on $\mcO_q(U)$, which is absurd. 

It follows that $\rho_{\nu}^{\chi}$ contains a non-trivial subrepresentation of regular type. If $\lambda \in \mbH_{\tau}^{\times}$ is a highest weight appearing in the decomposition of this subrepresentation, write $\lambda_r = u_r |\lambda_r|$. Then it follows that we have produced an admissible (positive) highest weight for $\mcO_q(Z_{\nu'})$ where $\epsilon_r' = u_P(\alpha_r) \epsilon_r$. Let $x'$ be the corresponding central character. We then have a $*$-representation of $\mcO_q^*(O_{x'}^{\nu'})$ on $L^2_q(U)$. Since however the Haar state of $\mcO_q(U) \subseteq B(L^2_q(U))$ is given by a vector state, and since the Haar state of $\mcO_q(U)$ restricts to the invariant state on $\mcO_q^*(O_{x'}^{\nu'})$, it follows that the above must extend to an embedding $L^{\infty}_q(O_{x'}^{\nu'}) \subseteq B(L^2_q(U))$. From  Theorem \ref{TheoDecompvN}, we however know that the support projection of $a_{\rho}$ equals 1, so $\rho_{\nu}^{\chi}$ must be regular. 
\end{proof}

In general, determining if $x_{\chi} \in G_{\nu,q}\backslash G_{\R,q}/U_q$ involves computing some spectral information of the operator $\rho_{\nu}^{\chi}(a_{\rho})$, namely the existence of a positive eigenvalue. This can in principle be a difficult problem. We will show however that this problem can be circumvened by a deformation argument, exploiting the fact that the signature of $a_{\rho}$ in any $L^{\infty}_q(O_x^{\nu})$ only depends on $\nu$ and not on $q$ or $x$. We will only show how this phenomenon works for the case of symmetric twisting data, as they are the ones that have been treated explicitly\footnote{Also the `flag case' where $\epsilon^2 = \epsilon$ was dealt with in \cite{DCM18}, but the case of a general ungauged twisting datum, interpolating between the flag case and the symmetric case, was not.}  in \cite{DCM18}.

Since we will now need to work with variable $q$, we will index all notations with either $q$ as subscript or superscript. Recall then from 
\cite[Lemma 1.1]{NT11} that, for each $\varpi \in P^+$, one can create $*$-representations $\pi_{\varpi}^q$ of $U_q(\mfu)$ for $0<q\leq 1$ on a fixed finite-dimensional Hilbert space $V_{\varpi}$ such that each $\pi_{\varpi}^q$ is a highest weight representation at weight $\varpi$, such that the weight space decomposition of $V_{\varpi}$ is independent of $q$, and such that the $\pi_{\varpi}^q(E_r)$ depend continuously on $q \in (0,1]$ for each $r\in I$. In the following, we will fix for each $\varpi\in P^+$ such a continuously varying family of representations on the fixed Hilbert space $V_{\varpi}$. One then has the following theorem. 

\begin{Theorem}\cite[Theorem 1.2]{NT11}
There exists a unique structure of continuous field of C$^*$-algebras on $(0,1]$ with fibers given by the $C_q(U)$, such that the $U_{\varpi}^q(\xi,\eta) \in \mcO_q(U)$ form continuous sections.  
\end{Theorem}

Note that for any $f\in \mcO(U)$, one has a unique corresponding $f_q \in \mcO_q(U)$ using that the $U_{\varpi}(e_i,e_j)$ form a basis of $\mcO_q(U)$. The above theorem then states that the norm $q \mapsto \|f_q\|$ varies continuously. 

To proceed, note that also $\mcO_q(Z_{\nu})$ can be meaningfully defined as an algebra for $q=1$. Indeed, by \cite[Definition 2.24]{DCM18} one can view $\mcO_q(Z_{\nu})$ as a copy of $\mcO(U)$ with the new product 
\[
f *_q g = (g_{(2)}\otimes f_{(1)},\Omega_{\epsilon,q}) \mbr_q(f_{(2)},g_{(3)}) f_{(3)}\cdot_q g_{(4)}\mbr_q(f_{(4)},\tau(S_q(g_{(1)}))),
\]
where $\cdot_q$ denotes the product of $\mcO_q(U)$ transported to $\mcO(U)$, where $\mbr_q(f,g) = (f\otimes g)\msR_q$ and where $\Omega_{\epsilon,q}$ acts on each $V_{\varpi}\otimes V_{\varpi'}$ by 
\[
\Omega_{\epsilon,q} \circ \iota  = \iota \circ \msE,\qquad \forall \iota\in \Hom_{U_q(\mfu)}(V_{\varpi''},V_{\varpi}\otimes V_{\varpi'}),
\]
with $\msE$ as in \eqref{DefMsE}. Since $\msR_q (\xi\otimes \eta)$ is continuous in $q$, with limit $\xi\otimes \eta$ at $q=1$, this product then varies continuously in $q$ for the weak topology on $\mcO(U)$ induced by its duality with $\prod_{\varpi} \End(V_{\varpi})$ (to see that the $\Omega_{q,\epsilon}$ also vary continuously, use that $\xi_{\varpi}\otimes \eta_{w_0\varpi'}$ is cyclic, for $\eta_{w_0\varpi'}$ a unit lowest weight vector in $V_{\varpi'}$). Hence $\mcO(Z_{\nu}) = \mcO_1(Z_{\nu})$ may be identfied with $\mcO(U)$ with the product 
\begin{equation}\label{EqProdZCoc}
f * g = (g_{(1)}\otimes f_{(1)},\Omega_{\epsilon}) f_{(2)}g_{(2)}.
\end{equation}
We then denote by $\widetilde{f}_q$ the copy of $f\in \mcO(U)$ inside $\mcO_q(Z_{\nu})$.

Our key lemma will be the following.

\begin{Lem}\label{LemCharGivesPos}
Assume that $\chi_q: \mcO_q(Z_{\nu}) \rightarrow \C$ are a family of $*$-characters such that $q \mapsto \chi_q(\widetilde{f}_q)$ is continuous on $(0,1]$ for each $f\in \mcO(U)$. Assume that that there exists $v\in U$ such that $\chi_1(Z_{\alpha}^1(\tau(v)\xi_{\alpha},v\xi_{\alpha})) >0$ for all $\alpha\in Q^+$. Then $\Lambda_{\nu}^{\gg}(x_q) \neq \emptyset$ for all $0<q<1$, where $x_q = (\chi_q)_{\mid \msZ(\mcO_q(Z_{\nu}))}$.
\end{Lem} 
\begin{proof}
The hypotheses imply that the functions $u \mapsto \chi_1(Z_{\alpha_r}^1(\tau(u)\xi_{\alpha_r},u\xi_{\alpha_r}))$ have a point of $\R_{>0}^{|I|}$ in their joint spectrum. By continuity, it follows that also the operators
\[
(\chi_q \otimes \pi_{\reg,q})\rho_{\nu,q}(a_{\alpha_r,q}) \in \mcO_q(U) \subseteq C_q(U)
\]
have  a point of $\R_{>0}^{|I|}$ in their joint spectrum for $q$ close to $1$. However, by Theorem \ref{TheoDecompvN} the joint spectrum of the $a_{\alpha_r}$ must hit the same connected components of $(\R\setminus \{0\})^{|I^{\tau}|} \times \C^{2|I^*|}$ for all $q$. Hence the $(\chi_q \otimes \pi_{\reg,q})\rho_{\nu,q}(a_{\alpha_r,q})$ must have a point of $\R_{>0}^{|I|}$ in their joint spectrum for \emph{all} $0<q<1$. 

In particular, there must exist a highest weight representation of $\mcO_q(Z_{\nu}^{\reg})$ in which all $a_{\alpha_r,q}$ have positive spectrum. Since the commutation relations of $\mcO_q(Z_{\nu}^{\reg})$ only involve the $a_{\alpha_r,q}$ non-trivially, it follows easily that one can then construct a highest weight representation of $\mcO_q(Z_{\nu}^{\reg})$ at $\lambda \in \Lambda_{\nu}^{\gg}(x_q)$. 
\end{proof}

We are now ready to prove the following theorem.
\begin{Theorem}
Let $\nu = (\tau,\epsilon)$ be a symmetric twisting datum. Then $G_{\nu,q}\backslash G_{\R,q}/U_q\neq \emptyset$.
\end{Theorem} 

\begin{proof}
We follow the terminology and notation of \cite{DCM18}. 

We may assume that $\nu$ is reduced. From the discussion in \cite[Appendix B]{DCM18} one sees that one can find a concrete Satake diagram $(X,\tau\tau_0)$ such that $\epsilon$ is an $(X,\tau\tau_0)$-admissible sign function. We let $\theta$ be the associated Satake involution. 

Let $\chi_q$ be the character on $\mcO_q(Z_{\nu})$ given by 
\[
\chi_q(Z_{\varpi}^q(\xi,\eta)) = \langle \xi, \msK_q\eta\rangle,
\]
with $\msK_q$ determined by \cite[(4.31)]{DCM18}. As in the proof of \cite[Theorem 4.54]{DCM18}, one has that the $\msK_q$ in $B(V_{\varpi})$ depend continuously on $q$, with 
\[
\msK_1 = \msE \widetilde{\epsilon}^{-1}m_0m_X \widetilde{z}_{\tau\tau_0}^{-1} =  \msE \widetilde{\epsilon}^{-1}\widetilde{z}m_0m_X, 
\]
where for the second equality we have used \cite[Lemma 4.26]{DCM18} (note that we have interchanged the notations $\tau$ and $\tau\tau_0$ with respect to the convention in \cite{DCM18}). 

Let now $J_{\varpi}: V_{\varpi} \rightarrow V_{\tau(\varpi)}$ be the unique (unitary) intertwiner between $\pi_{\varpi}\circ \tau$ and $\pi_{\tau(\varpi)}$ sending $\xi_{\varpi}$ to $\xi_{\tau(\varpi)}$. Put 
\[
v_{\varpi} = J_{\varpi} \msE \widetilde{\epsilon}^{-1}\widetilde{z} m_0m_X.
\]
From the definition of $\theta$ and the centrality of $\msE\widetilde{\epsilon}^{-1}$, it follows that
\[
v_{\varpi} \pi_{\varpi}(u) = \pi_{\tau(\varpi)}(\theta(u))v_{\varpi},\qquad u\in U. 
\]
Since by construction $\theta$ and $\nu$ are inner equivalent, we can find $k\in U$ such that 
\[
\nu(u) = k^{-1}\theta(kuk^{-1})k,\qquad u\in U.
\]
In particular, since 
\[
\pi_{\tau(\varpi)}(\nu(u)) = \Ad(J_{\varpi} \msE)(\pi_{\varpi}(u)),\qquad u\in U,
\]
it follows that 
\[
\msE^{-1} J_{\varpi}^{-1} k^{-1}J_{\varpi} \msE \widetilde{\epsilon}^{-1}  \widetilde{z}m_0m_X k = \widetilde{\epsilon}^{-1}\tau(k)^{-1}\widetilde{z} m_0m_Xk.
\]
is central. Hence for all $\alpha \in Q^+$, one has
\begin{eqnarray*}
\chi_1(Z_{\alpha}^1(\tau(k)\xi_{\alpha},k\xi_{\alpha})) 
&=& \langle \tau(k)\xi_{\alpha},\msK_1 k\xi_{\alpha} \rangle \\
&=& \langle \xi_{\alpha},\tau(k)^{-1} \msE \widetilde{\epsilon}^{-1}\widetilde{z}m_0m_X k \xi_{\alpha}\rangle \\
&= & \langle \xi_{\alpha},\tau(k)^{-1} \msE \widetilde{\epsilon}^{-1}\widetilde{z}m_0m_X k \xi_{\alpha}\rangle \\
&=& \langle \xi_{\alpha},\msE \widetilde{\epsilon}^{-1}\tau(k)^{-1} \widetilde{z}m_0m_X k \xi_{\alpha}\rangle\\
&=&  \langle \xi_{\alpha},\msE  \xi_{\alpha}\rangle\\ 
&=& \|\xi_{\alpha}\|^2 >0.
\end{eqnarray*}
From Lemma \ref{LemCharGivesPos} it now follows that $\Lambda_{\nu}^{\gg}(x_{\chi_q}) \neq \emptyset$ for all $0<q<1$. 
\end{proof}

\begin{Rem}
Given the $*$-character $\chi_q$ as above with associated central character $x_q = x_{\chi_q}$, it would be interesting to know explicitly the value of the associated $\lambda_{x_q} \in \Lambda_{\nu,q}^{\gg}(x_q)$. It is not clear to us at the moment how to determine this value without going into the computation of eigenvalues of the corresponding operator $\rho_{\nu}^{\chi_q}(a_{\rho,q})$. 
\end{Rem}

\subsection{Symmetric spaces of the form $U\times U/\diag(U)$}\label{SecSpecExDiag}

Let as before $\mfu$ be a compact semisimple Lie algebra with Dynkin diagram $\Gamma$ built on the index set $I$. Consider the direct sum $\mfu \oplus \mfu$ with its Dynkin diagram $\Gamma\times \Gamma$ built on the set $I\times I$, and consider the associated ungauged twisting datum $\nu = (\tau,+)$ where $\tau(r,s) = (s,r)$. We can then interpret 
\[
\mcO_q(Z_{\nu}) =\mcO_q(G_{\R} \dbbackslash G_{\R}\times G_{\R}).
\]
The associated $*$-algebra $\mcO_q(Z_{\nu}^{\reg})$ is generated by elements $x_{(r,0)}, x_{(0,r)}$ and $a_{(\omega,\chi)}$ where the $x_{(r,0)}$ and $x_{(0,s)}$ satisfy the usual quantum Serre relations,  where any of the $x_{(r,0)}$ and $x_{(0,s)}$ commute and where
\[
a_{(\omega+\omega',\chi+\chi')}  =  a_{(\omega,\chi)}a_{(\omega',\chi')},\qquad a_{(\omega,\chi)}^* = a_{(\chi,\omega)},
\]
\[
a_{(\omega,\chi)} x_{(r,0)} = q^{-(\omega+\chi,\alpha_r)}x_ra_{\omega} ,\qquad a_{(\omega,\chi)}x_{(0,r)} = q^{(\omega+\chi,\alpha_r)}x_{(0,r)}a_{(\omega,\chi)} 
\]
and
\[
x_{(r,0)} x_{(s,0)}^*  - q^{-(\alpha_r,\alpha_s)}  x_{(s,0)}^*x_{(r,0)} = - \frac{\delta_{r,s}}{q_r-q_r^{-1}},\quad x_{(0,r)} x_{(0,s)}^*  - q^{-(\alpha_r,\alpha_s)}  x_{(0,s)}^*x_{(0,r)} = - \frac{\delta_{r,s}}{q_r-q_r^{-1}},
\]
\[
x_{(r,0)} x_{(0,s)}^*  - x_{(0,s)}^*x_{(r,0)} =  \frac{\delta_{r,s}q_r a_{(0,-\alpha_r)}} {q_r-q_r^{-1}}.
\]

On the other hand, consider for $\mfu$ the twisting datum $\mu = (\id,0)$ on $I$ with associated $*$-algebra $\mcO_q(Z_{\mu}^{\reg})$. Then we see that we have an embedding of $*$-algebras
\[
\mcO_q(Z_{\mu}^{\reg})\rightarrow \mcO_q(Z_{\nu}^{\reg}),\qquad a_{\omega}\mapsto a_{(\omega,\omega)},\quad x_r \mapsto x_{(r,0)}. 
\] 

\begin{Prop}
We have $\lambda = q^{\rho} \in \Lambda_{\nu}^{\gg}$. 
\end{Prop}
\begin{proof}
Let $M_{\lambda}$ be the Verma module of $\mcO_q(Z_{\nu}^{\reg})$ at $\lambda$ with highest weight vector $\xi_{\lambda}$, and let $V_{\lambda}$ be its irreducible quotient. An easy computation shows that all the $(x_{(r,0)}^* + x_{(0,r)}^*)\xi_{\lambda}$ are also highest weight vectors in $M_{\lambda}$, and hence 
\[
x_{(0,r)}^* \xi_{\lambda} = -x_{(r,0)}^*\xi_{\lambda} 
\]
in $V_{\lambda}$. It follows that the vector $\xi_{\lambda}\in V_{\lambda}$ is already cyclic for the $*$-algebra generated by the $x_{(r,0)}$ and the $a_{(\omega,\omega)}$. Since the latter is an isomorphic copy of the $*$-algebra $\mcO_q(Z_{\mu}^{\reg})$, it follows by uniqueness that the $\mcO_q(Z_{\nu}^{\reg})$-invariant inner product on $V_{\lambda}$ must coincide with the one coming from its representation of $\mcO_q(Z_{\mu}^{\reg})$. The latter is however positive-definite by \cite[Theorem 2.7]{DeC13}. 
\end{proof}

Let 
\[
x = x_{\lambda}\in G_{\R,q} \backslash G_{\R,q}\times G_{\R,q}/U_q
\]
be the associated central character of $\mcO_q(Z_{\nu})$. We will show that $\mcO_q(O_x^{\nu}) \cong \mcO_q(U)$ are isomorphic as $*$-algebras, and identify the corresponding coaction of $\mcO_q(U)\otimes \mcO_q(U)$. As the computations are somewhat tedious, and as the results are probably known to specialists already in one form or another, we will not specify all details. 

We start with recalling the following results from  \cite{ST09}, to which we refer for the ideas behind the computations. Up to a regauging (and change of conventions), the element $Y$ in the lemma below corresponds to the \emph{half-twist element} of \cite{ST09}. We follow the same conventions as \cite{DCM18}. 

\begin{Lem}
Let $T_{w_0}\in \prod_{\varpi} \End(V_{\varpi})$ be the braid group operator corresponding to the longest element $w_0$ in the Weyl group, and let $C \in  \prod_{\varpi} \End(V_{\varpi})$ be the operator given by 
\[
C \xi = q^{\frac{1}{2}(\wt(\xi),\wt(\xi)) + (\wt(\xi),\rho)}
\]
Let $\mcS_0 = e^{2\pi i \rho^{\vee}}$, where $\rho^{\vee}(\alpha_r) = 1$ for all $r\in I$, and define 
\[
Y = \mcS_0 C T_{w_0}. 
\]
Then the following hold: 
\begin{enumerate}
\item $\msR = (Y\otimes Y)\Delta(Y)^{-1}$. 
\item $Y^2 = \mcS_0 v^{-1}$, where $v \in \prod_{\varpi} \End(V_{\varpi})$ is the standard central ribbon element satisfying 
\[
v_{\varpi} = q^{-(\varpi,\varpi+2\rho)},\qquad \msR_{21}\msR = \Delta(v)(v^{-1}\otimes v^{-1}),
\]
\item $Y^* = R(Y) = Y \mcS_0$, where $R$ is the unitary antipode. 
\item $\Ad(Y)$ is an anti-comultiplicative $*$-algebra automorphism with 
\[
\Ad(Y)K_{\omega} = K_{-\tau_0(\omega)},\qquad \Ad(Y)E_r =-q_rF_{\tau_0(r)},\qquad \Ad(Y)F_r =-q_r^{-1}E_{\tau_0(r)},
\]
where $\tau_0$ is the involution on $\Gamma$ determined by the $-\Ad(w_0)$. 
\end{enumerate}
\end{Lem}

\begin{Prop}
The map 
\[
\chi = \chi_{\msK}: \mcO_q(Z_{\nu})\rightarrow \C,\qquad Z(\xi\otimes \eta,\xi'\otimes \eta') \mapsto \langle \xi\otimes \eta,\msK (\xi'\otimes \eta')\rangle
\]
where 
\[
\msK = (Y^{-1}\otimes 1)\Delta^{\opp}(Y)^{-1}(Y\mcS_0 v^{-1}\otimes v^{-1}),
\]
defines a $*$-character on $\mcO_q(Z_{\nu})$. 
\end{Prop}
\begin{proof}
Using the above properties of $Y$, this follows by an easy, albeit tedious computation following the description of the $*$-compatible $K$-matrix in \cite{DCM18}, after noting that the quasi-$K$-matrix of \cite{BK15b} coincides with $(\id\otimes \Ad(K_{\rho}Y))\widetilde{\msR}$, with $\widetilde{\msR} \in U_q(\mfn) \hat{\otimes} U_q(\mfn^-)$ the quasi-$R$-matrix. 
\end{proof} 

From \cite[Theorem 4.40]{DCM18}, it follows that the associated $*$-homomorphism $\rho_{\nu}^{\chi}: \mcO_q(Z_{\nu}) \rightarrow \mcO_q(U)\otimes \mcO_q(U)$ has as its range the right coideal dual to the left coideal written as $\mathscr{C} = R(\widetilde{\msB})$ in \cite{DCM18}, which in the case at hand is easily seen to be given as
\[
\mathscr{C} = (\Ad(Y)\otimes \id)\Delta^{\opp}(U_q(\mfu)) \subseteq U_q(\mfu)\otimes U_q(\mfu). 
\]
The dual coideal is then given by 
\[
\mcO_q(\diag(U) \backslash (U\times U)) = \{U_{\tau_0(\varpi)\otimes \varpi}(\sum_i YK_{\rho}\overline{e_i}\otimes e_i,\overline{\xi}\otimes \eta)\mid \varpi \in P^+\},
\]
where we identify $V_{\tau_0(\varpi)} \cong \overline{V_{\varpi}}$ with the usual conjugate Hilbert space structure and $U_q(\mfu)$-representation given by $x \overline{\xi} = \overline{R(x)^*\xi}$. Restricting the above $*$-subalgebra of $\mcO_q(U)\otimes \mcO_q(U)$ to functionals on $\C\otimes U_q(\mfu)$, we obtain an isomorphism of $*$-algebras
\[
\mcO_q(U) \cong \mcO_q(\diag(U) \backslash (U\times U)) \qquad U_{\varpi}(\xi,\eta) \mapsto U_{\tau_0(\varpi)\otimes \varpi}(\sum_i YK_{\rho}\overline{e_i}\otimes e_i, (Y^*)^{-1} K_{\rho}^{-1}\overline{\xi}\otimes \eta).
\]
Under this isomorphism, the right coaction of $\mcO_q(U)\otimes \mcO_q(U)$ is transported to the coaction 
\[
\widetilde{\Delta}_{\diag}(f) = f_{(2)}\otimes S(f_{(1)})\circ \Ad(Y)\otimes f_{(3)}. 
\]

Let now $\msZ_1$ be the span of the elements $1\otimes Z_{\varpi}(\xi,\eta)$ in $\mcO_q(Z_{\nu})$ corresponding to the component $1\otimes \mcO_q(U)$, and consider similarly the span $\msZ_2$ of the elements $Z_{\varpi}(\xi,\eta)\otimes 1$. From for example \eqref{EqProdZCoc} one sees that 
\[
\msZ_1 \msZ_2 = \mcO_q(Z_{\nu}) = \msZ_2\msZ_1.
\]
Moreover, it is easily seen that 
\[
\rho_{\nu}^{\chi}(1\otimes Z_{\varpi}(\xi,\eta)) =  U_{\varpi}(Y\xi,\eta),\qquad \rho_{\nu}^{\chi}(Z_{\varpi}(\xi,\eta)\otimes 1) = U_{\varpi}(Y\eta,\xi)^*,  
\]
where the second formula easily follows from the first using that $\rho_{\nu}^{\chi}$ is a $*$-homomorphism. 

We are now ready to prove the following theorem. 

\begin{Theorem}
The $\mcO_q(U)\otimes \mcO_q(U)$ equivariant $*$-homomorphism $\rho_{\nu}^{\chi}$ factors through $\pi_{q^{\rho}}$, inducing a $*$-isomorphism 
\[
\mcO_q(O_{x_{q^{\rho}}}^{\nu}) \cong \mcO_q(U).
\]
\end{Theorem}
\begin{proof}
We see from the above that 
\[
\rho_{\nu}^{\chi}(a_{(0,\varpi)}) =q^{-\frac{1}{2}(\varpi,\varpi)}b_{\varpi},\qquad b_{\varpi} =  U_{\varpi}(Y\xi_{\varpi},\xi_{\varpi}).
\]
Localizing at $b_{\varpi},b_{\varpi}^*$, we obtain a $*$-algebra $\mcO_q^{\loc}(U)$ that is the image of $\mcO_q(Z_{\nu}^{\reg})$. In fact, $\mcO_q^{\loc}(U)$ is isomorphic by $\rho_{\nu}^{\chi}$ to the $*$-algebra $\widetilde{\mcO}_q(Z_{\mu})$ generated by the $a_{(\varpi,\varpi')}$, the $x_{(r,0)}$ and the $x_{(r,0)}^*$, see e.g.~ \cite{DC18}. It now follows immediately from the argument in the construction of $\pi_{q^{\rho}}$ that the latter $*$-representation must factor over $\mcO_q^{\loc}(U)$.  
\end{proof}

\begin{Rem}
Once this identification is made, we note that the specific instance of Theorem \ref{TheoFormInvFunct} in this case was obtained in \cite{RY01}. 
\end{Rem}

\subsection{Quantization of the real Grassmannian spaces $SU(N)/S(U(m)\times U(n))$}

To end, we return to the example we considered in the introduction. Put $\mfu = \mfsu(N)$, and consider the Hermitian symmetric pair $\mfs(\mfu(m) \oplus \mfu(n)) \subseteq \mfsu(N)$. It can be described by the Vogan diagram $(\id,\epsilon)$ where $\epsilon_r = (-1)^{\delta_{r,m}}$.   
\[
\begin{tikzpicture}[scale=.4,baseline=1cm]
\node (v1) at (8.8,0.8) {};
\node (v2) at (8.8,-0.8) {};
\draw (0cm,0) circle (.2cm) node[above]{\small $1$} ;
\draw (0.2cm,0) -- +(0.6cm,0);
\draw (1cm,0) circle (.2cm);
\draw (1.2cm,0) -- +(0.2cm,0);
\draw[dotted] (1.4cm,0) --+ (1cm,0);
\draw (2.4cm,0) --+ (0.2cm,0);
\draw (2.8cm,0) circle (.2cm); 
\draw (3cm,0) --+ (0.6cm,0);
\draw[fill = black] (3.8cm,0) circle (.2cm)  node[above]{\small $m$};
\draw (4cm,0) --+ (0.6cm,0);
\draw (4.8cm,0) circle (.2cm);
\draw (5cm,0) -- +(0.2cm,0);t
\draw[dotted] (5.2cm,0) --+ (1cm,0);
\draw (6.2cm,0) --+ (0.2cm,0);
\draw(6.6cm,0) circle (.2cm);
\draw (6.8cm,0) --+ (0.6cm,0);
\draw (7.6cm,0) circle (.2cm) node[above]{\footnotesize $N-1$};
\end{tikzpicture}
\]

In this case, the set $\Lambda_{\nu}^{\gg}$ was already determined in \cite{DeC13}.

\begin{Prop}\cite[Theorem 2.7]{DeC13} \label{PropPosChar}
Let $\nu = (\id,\epsilon)$ with $\epsilon_r = (-1)^{\delta_{r,m}}$. Then $\lambda \in \Lambda_{\nu}^{\gg}$ if and only if $\lambda = q^{-2\gamma}$ with $\gamma = \sum_{r=1}^{N-1} k_r\varpi_r$ for $k_r\in \N$ when $r\neq m$ and $k_m\in \R$. 
\end{Prop} 

\begin{Rem}
There is a missing assumption in the statement of \cite[Theorem 2.7]{DeC13}, in that it only applies to \emph{hermitian} symmetric spaces. Indeed, the theorem relies crucially on \cite[Proposition 5.13]{JT02}, which needs the property that the distinguished simple root appears only with single multiplicity in all other roots, a property satisfied precisely in the hermitian symmetric case. 
\end{Rem}

Note now that in this case $W_{\nu} = W = S_N$, while $W_{\nu}^+ = S_m\times S_n$. The decomposition $W = W^- \times W^+$ is then the standard one where $w = w_-w_+$ if $w_-$ is the element of minimal length in the $W^+$-coset class of $w$. Thus $W^-$ consists of those $w\in S_N$ with $w(i) \leq w(j)$ for all $1\leq i \leq j \leq m$ and all $m+1 \leq i\leq j \leq N$. We obtain then the following theorem.

\begin{Theorem}
Let $\nu = (\id,\epsilon)$ with $\epsilon_r = (-1)^{\delta_{r,m}}$, and let $\lambda \in \Lambda_{\nu}^{\gg}$ be given as in  Proposition \ref{PropPosChar}. Let $x = x_{\lambda}$ be the associated character. Then 
\[
L^{\infty}_q(O_x^{\nu}) \cong \oplus_{w\in W^-} B(\Hsp_{w\cdot_{\epsilon}\lambda}). 
\]
\end{Theorem} 

This gives a quantum analogue, in the measure-theoretic setting, of the decomposition in \eqref{EqCellDecomp}. Of course, this measure-theoretic setting does not see the lower-dimensional cells, for which finer C$^*$-algebraic techniques need to be used. A full understanding of $C_q(O_x^{\nu})$ is however out of reach at the moment. 

\begin{Rem}
The $*$-algebra $\mcO_q(Z_{\nu}^{\reg})$ is in this case just a copy, up to taking square roots of the Cartan part, of $U_q(\mfsl(N,\C))$ with a peculiar $*$-structure. This means that most of the algebraic machinery to study $U_q(\mfsl(N,\C))$ can also be used to study $\mcO_q(Z_{\nu}^{\reg})$. In particular, one can use the computations of \cite{UTS90} to compute explicit Gelfand-Zeitlin bases for Verma modules of $\mcO_q(Z_{\nu}^{\reg})$ for any ungauged twisting data. It is then not hard to deduce from this the precise form of $\Lambda_{\nu}^{\gg}$. 
\end{Rem}

\appendix

\setcounter{section}{-1}
\section{Computations in rank 2}\label{Ap3}

\renewcommand{\thesection}{A}

In this appendix we classify the highest weight $*$-representations of $\mcO_q(Z_{\nu}^{\reg})$ for the case of a rank $2$ Lie algebra $\mfu$ with $\nu = (\tau,\epsilon)$, $\epsilon$ ungauged and $\tau$ non-trivial. We write the associated simple roots $\alpha_1,\alpha_2$ so that $\tau(1) = 2$. We further write $\epsilon = \epsilon_1= \epsilon_2 \in \R$.  

A highest weight $*$-representation is completely determined by the scalar $\lambda  = \lambda_P(\varpi_1) \in \C^{\times}$. Since we know already that the class of $\lambda$'s appearing this way is stable under multiplication with the circle group $\mbH^{\gauge}$, cf. Lemma \ref{LemCloseGauge}, we can restrict to consdering $\lambda>0$. On the level of $\mcO_q(Z_{\nu}^{\reg})$ this means we may impose the extra relation $a_{\varpi_1} = a_{\varpi_2} = a_{\varpi_1}^*$. We denote the resulting quotient by $\mcO_q(\widetilde{Z}_{\nu}^{\reg})$ and write $a = a_{\varpi_1}$.

\subsection{Twisted conjugacy classe of type $A_1\times A_1$ with nontrivial automorphism}

Let $\mfu = \mfsu(2)\oplus \mfsu(2)$. The associated $*$-algebra $\mcO_q(\widetilde{Z}_{\nu}^{\reg})$ is generated by elements $a^{\pm 1},x_1,x_2$ with $a$ self-adjoint and
\begin{equation}\label{Eqx1ax1starA11}
ax_1 = q^{-1}x_1a,\quad x_1x_1^* = q^{-2}x_1^*x_1+ (q^{-1}-q)^{-1},
\end{equation}
\begin{equation}\label{Eqx2ax2starA11}
 ax_2 =q^{-1}x_2a,\quad x_2x_2^* = q^{-2}x_2^*x_2 +(q^{-1}-q)^{-1},
\end{equation}
\begin{equation}\label{Eqx1ax2starA11}
x_1x_2^* = x_2^*x_1+\frac{\epsilon q}{q-q^{-1}}a^{-2},
\end{equation}
and 
\begin{equation}\label{EqCommA11}
x_1x_2 = x_2x_1.
\end{equation}

\begin{Lem}
The following relations hold for $k,n\in \N$:
\[
x_1(x_1^*)^n = q^{-2n}(x_1^*)^n x_1 + q^{-n+1}\frac{q^{-n}-q^n}{(q^{-1}-q)^2}(x_1^*)^{n-1},
\]
\[
x_1(x_2^*)^k = (x_2^*)^kx_1 - \epsilon q^{-k+2} \frac{q^{-k}-q^k}{(q^{-1}-q)^2}(x_2^*)^{k-1}a^{-2},
\]
\[
x_2(x_1^*)^n = (x_1^*)^nx_2 - \epsilon q^{-n+2} \frac{q^{-n}-q^n}{(q^{-1}-q)^2}(x_1^*)^{n-1}a^{-2},
\]
\[
x_2(x_2^*)^k = q^{-2k}(x_2^*)^k x_2 + q^{-k+1}\frac{q^{-k}-q^k}{(q^{-1}-q)^2}(x_2^*)^{k-1}.
\]
\end{Lem}
\begin{proof}
By straightforward induction.
\end{proof}

As in the proof of Lemma \ref{LemIsoHW}, let $M_{\lambda}$ be the Verma module at weight $\lambda$ with highest weight vector $\xi_{\lambda}$. Put
\[
\xi_{kn} = (x_2^*)^k(x_1^*)^n \xi_{\lambda}.
\]
Also the following lemma follows by straightforward induction. 

\begin{Lem}\label{LemActA11}
The generators act as follows on the $\xi_{kn}$: 
\[
a \xi_{kn} = \lambda q^{k+n} \xi_{kn},\qquad x_1^*\xi_{kn} = \xi_{k,n+1},\qquad x_2^*\xi_{kn} = \xi_{k+1,n},
\]
\[
x_1\xi_{kn} = q^{-n+1}\frac{q^{-n}-q^n}{(q^{-1}-q)^2}\xi_{k,n-1} -\epsilon\lambda^{-2} q^{-k-2n+2}\frac{q^{-k}-q^k}{(q^{-1}-q)^2} \xi_{k-1,n},
\]
\[
x_2 \xi_{kn} = -\epsilon\lambda^{-2} q^{-2k-n+2}\frac{q^{-n}-q^n}{(q^{-1}-q)^2}\xi_{k,n-1} +q^{-k+1} \frac{q^{-k}-q^k}{(q^{-1}-q)^2} \xi_{k-1,n}. 
\]
\end{Lem} 

\begin{Lem}
For each $t\in \N$, the space of highest weight vectors for $x_1$ at weight $\lambda q^{t}$ is one-dimensional, and given by $\C \eta_{0}^{(t)}$ where $\eta_0^{(t)} = \sum_{k+n=t} c_n \xi_{k,n}$ with $c_0 = 1$ and 
\[
c_n =\epsilon \lambda^{-2} q^{-t+2}\frac{q^{-t+n-1}-q^{t-n+1}}{q^{-n}-q^n}c_{n-1},\qquad 1\leq n \leq t.
\] 
\end{Lem}
\begin{proof}
The recursive formula follows at once from Lemma \ref{LemActA11}. 
\end{proof}

Put now 
\[
\eta_n^{(t)} = (x_1^*)^n\eta_0^{(t)}
\]
and consider the unique invariant hermitian form on $M_{\lambda}$ with $\langle \xi_{\lambda},\xi_{\lambda}\rangle = 1$. 
\begin{Lem}
The $\eta_n^{(t)}$ are mutually orthogonal and form a basis of $M_{\lambda}$.  
\end{Lem} 
\begin{proof}
Orthogonality follows straightforwardly from invariance of the scalar product and weight considerations. Also the linear independence of the $\eta_n^{(t)}$ follows directly from weight considerations. They must hence form a basis since the dimension of the weight spaces of their span matches that of $M_{\lambda}$ itself. 
\end{proof}

\begin{Lem}\label{LemActionx2}
The following formulas hold: 
\[
x_2\eta_{0}^{(t)} =q^{-t+1}\frac{q^{-t}-q^t}{(q^{-1}-q)^2}(1-|\epsilon|^2\lambda^{-4} q^{-2t+4})\eta_0^{(t-1)},
\]
\[
x_2^*\eta_{0}^{(t)} = \eta_0^{(t+1)} - \epsilon \lambda^{-2}q^{-2t+1} \eta_1^{(t)}. 
\]
\end{Lem}
\begin{proof}
We note that $x_2\eta_0^{(t)}$ must again be a highest weight vector for $x_1$, hence a scalar multiple of $\eta_{0}^{(t-1)}$. The precise coefficient follows from comparing the coefficient of $\xi_{t-1,0}$ in both expressions. 

By weight considerations $x_2^*\eta_0^{(t)}$ must be a linear combination of $\eta_0^{(t+1)}$ and an element $\zeta_t$ in the span of the elements $\eta_{n}^{(t')}$ for $t'\leq t$. Since 
\[
x_1\zeta_t = x_1 x_2^*\eta_0^{(t)} = -\frac{\epsilon q}{q^{-1}-q}a^{-2} \eta_0^{(t)} =  -\frac{\epsilon \lambda^{-2} q^{-2t+1}}{q^{-1}-q}\eta_0^{(t)}, 
\]
it follows again from weight considerations that we must have $\zeta_t =  - \epsilon \lambda^{-2}q^{-2t+1} \eta_1^{(t)}$. The coefficient for $\eta_0^{(t+1)}$ is found by comparing coefficients of $\xi_{t+1,0}$.  
\end{proof}

\begin{Theorem}\label{TheoAppA1A1}
There exists a unitarizable highest weight module at weight $\lambda$ if and only if $\epsilon =0$ or, in case $\epsilon\neq0$, $\lambda$ is of the form $\lambda = |\epsilon|^{1/2}q^{\frac{1-n}{2}}$ for $n\in \N$.  
\end{Theorem}
\begin{proof}
As in the proof of Lemma \ref{LemIsoHW}, a unitarizable highest weight module at weight $\lambda$ exists if and only if the invariant hermitian form on $M_{\lambda}$ is positive semi-definite. This will be the case if and only if 
\[
c_{t,n} = \langle \eta_n^{(t)},\eta_n^{(t)}\rangle \geq 0 
\]
for all $t,n$. Now since the $x_1,x_1^*$ satisfy the Heisenberg relations, it is easily seen that we only need to check that $c_t = c_{t,0}\geq 0$ for all $t$. However, from Lemma \ref{LemActionx2} we see that 
\[
c_t = \langle \eta_0^{(t)},x_2^*\eta_0^{(t-1)}\rangle = \langle x_2 \eta_0^{(t)},\eta_0^{(t-1)}\rangle = q^{-t+1}\frac{q^{-t}-q^t}{(q^{-1}-q)^2}(1-|\epsilon|^2\lambda^{-4} q^{-2t+4})c_{t-1}. 
\]
Thus $c_t\geq 0$ for all $t$ if and only if $\lambda$ is as indicated in the theorem. 
\end{proof}

\subsection{Twisted conjugacy classe of type $A_2$ with nontrivial automorphism}

Consider $\mfu = \mfsu(3)$ with twisting datum $\nu= (\tau,\epsilon)$ where $\tau\neq \id$. We use the notation as explained at the beginning of the Appendix, and consider the $*$-algebra $\mcO_q(\widetilde{Z}_{\nu}^{\reg})$. It is generated by elements $a^{\pm 1},x_1,x_2$ with $a$ self-adjoint and
\begin{equation}\label{Eqx1ax1star}
ax_1 = q^{-1}x_1a,\quad x_1x_1^* = q^{-2}x_1^*x_1+ (q^{-1}-q)^{-1},
\end{equation}
\begin{equation}\label{Eqx2ax2star}
 ax_2 =q^{-1}x_2a,\quad x_2x_2^* = q^{-2}x_2^*x_2 +(q^{-1}-q)^{-1},
\end{equation}
\begin{equation}\label{Eqx1ax2star}
x_1x_2^* = qx_2^*x_1+\frac{\epsilon q^{3/2}}{q-q^{-1}}a^{-1},
\end{equation}
and such that, with $x_{12} = x_1x_2 - qx_2x_1$, 
\begin{equation}\label{Eqx1ax12}
x_1x_{12} = q^{-1}x_{12}x_1,
\end{equation}
\begin{equation}\label{Eqx2ax12}
x_2x_{12} = qx_{12}x_2. 
\end{equation}

\begin{Lem}
The following further relations hold in $\mcO_q(\widetilde{Z}_{\nu}^{\reg})$: 
\begin{equation}\label{Eqaxx}
ax_{12} = q^{-2}x_{12}a,\qquad ax_{12}^* = q^{2}x_{12}^*a,\qquad x_1x_{12}^* = q^{-1}x_{12}^*x_1,
\end{equation}
\begin{equation}\label{Eqx2bla}
x_2x_{12}^* = q^{-1}x_{12}^*x_2 + qx_1^*-q^{1/2}\epsilon x_2^*a^{-1},
\end{equation}
and
\begin{equation}\label{Eqx12x12star}
x_{12}x_{12}^* = q^{-2}x_{12}^*x_{12} +(q^{-1}-q)x_1^*x_1 +q(q^{-1}-q)^{-1} + \frac{|\epsilon|^2q^{2}}{q^{-1}-q}a^{-2}.
\end{equation}
\end{Lem} 
\begin{proof}
By direct computation.
\end{proof}

\begin{Def}
We define $\msB$ to be the $*$-algebra generated by elements $a^{\pm 1},x_1,x_{12}$ with $a$ selfadjoint, satisfying the relations \eqref{Eqx1ax1star},\eqref{Eqx1ax12}, \eqref{Eqaxx},\eqref{Eqx12x12star}. 
\end{Def}

Clearly, we can define a notion of highest weight module and Verma module for $\msB$. 

\begin{Lem}
For any $\lambda >0$, the Verma module for $\msB$ at $\lambda$ is unitarizable. 
\end{Lem} 
\begin{proof}
Let $\xi_{m,n} = (x_{12}^*)^m(x_1^*)^n\xi_{\lambda}$. Then straightforward computations show that
\[
x_1^*\xi_{m,n}= q^m \xi_{m,n+1},\qquad x_{12}^*\xi_{m,n} = \xi_{m+1,n}, \qquad x_1\xi_{m,n} = q^{-m-n+1} \frac{q^{-n}-q^{n}}{(q^{-1}-q)^2}\xi_{m,n-1},
\]
\[
x_{12}\xi_{m,n} = q^{-m-2n+2}\frac{q^{-m}-q^{m}}{(q^{-1}-q)^2}(1+|\epsilon|^2\lambda^{-2}q^{-2m+3})\xi_{m-1,n}. 
\]
It follows easily from this that the $\xi_{m,n}$ are mutually orthogonal with respect to the invariant sesqui-linear form on the Verma module, and that the $\xi_{m,n}$ have positive norm squared.
\end{proof}

Note that the image of $\msB$ in its Verma module representation must then also be the image in the Verma module at corresponding weight of $\mcO_q(\widetilde{Z}_{\nu}^{\reg})$ through its natural embedding. 

\begin{Lem}
Let $\lambda>0$, and let $M_{\lambda}$ be the Verma module for $\mcO_q(\widetilde{Z}_{\nu}^{\reg})$ at $\lambda$. Put 
\[
\xi_{kmn} = (x_2^*)^k (x_{12}^*)^m(x_1^*)^n\xi_{\lambda}.
\]
Then
\begin{eqnarray*}
x_1\xi_{kmn} &=& q^{k-m-n+1}\frac{q^{-n}-q^n}{(q^{-1}-q)^2}\xi_{k,m,n-1} -\epsilon\lambda^{-1} q^{-2m-n+\frac{3}{2}}\frac{q^{-k}-q^k}{(q^{-1}-q)^2}\xi_{k-1,m,n},\\
x_1^*\xi_{kmn} &=& q^{-k+m}\xi_{k,m,n+1} -q^{-1} \frac{q^{-k}-q^k}{q^{-1}-q}\xi_{k-1,m+1,n},\\
x_2\xi_{kmn} &=& q^{-k+1}\frac{q^{-k}-q^k}{(q^{-1}-q)^2}\xi_{k-1,m,n} +q^{-2k+1} \frac{q^{-m}-q^m}{q^{-1}-q}\xi_{k,m-1,n+1} \\  
&&- \epsilon \lambda^{-1}q^{-2k-m +\frac{3}{2}} \frac{q^{-n}-q^n}{(q^{-1}-q)^2} \xi_{k,m,n-1} - \epsilon\lambda^{-1} q^{-2k-m-n+\frac{3}{2}}\frac{q^{-m}-q^m}{q^{-1}-q} \xi_{k+1,m-1,n},\\
x_2^*\xi_{kmn} &=& \xi_{k+1,m,n},\\
x_{12}\xi_{kmn} &=& q^{-k-m-2n+2}\frac{(q^{-m} -q^m)(1+|\epsilon|^2\lambda^{-2} q^{-2k-2m+3})}{(q^{-1}-q)^2}\xi_{k,m-1,n} \\ &&+ q^{2-m-n} \frac{(q^{-k}-q^{k})(q^{-n}-q^n)(1+|\epsilon|^2\lambda^{-2} q^{-2k-2m+3})}{(q^{-1}-q)^3}\xi_{k-1,m,n-1} \\
&&-\epsilon\lambda^{-1} q^{-2k-2m-n+\frac{9}{2}}\frac{(q^{-k}-q^k)(q^{-m}-q^m)}{(q^{-1}-q)^2} \xi_{k-1,m-1,n+1}\\
&&- \epsilon\lambda^{-1} q^{-k-2m-n+\frac{7}{2}}\frac{(q^{-k}-q^k)(q^{-k+1}-q^{k-1})}{(q^{-1}-q)^3}\xi_{k-2,m,n}\\
x_{12}^*\xi_{kmn} &=& q^k \xi_{k,m+1,n}.
\end{eqnarray*}
\end{Lem}
\begin{proof}
This follows from some tedious but straightforward computations using the defining relations. 
\end{proof}

\begin{Theorem}\label{TheoClassIrrepA2}
If $\epsilon = 0$, then there exists a unitarizable highest weight module for all weights $\lambda>0$. If $\epsilon\neq 0$, this is the case if and only if $\lambda = |\epsilon|q^{\frac{3-n}{2}}$ for some $n\in \N$. 
\end{Theorem}
\begin{proof}
Let $M_{\lambda}$ be the Verma module at weight $\lambda$. We claim that there exists, up to a scalar multiple, a unique non-zero vector $\eta_{00}^{(t)}$ at weight $q^{t}\lambda$ which is highest weight for $\msB$. Indeed, if 
\[
\eta_t = \sum_{k+2m+n=t} c_{k,m}\xi_{kmn}
\]
is such a vector, where again $\xi_{kmn} = (x_2^*)^k (x_{12}^*)^m(x_1^*)^n\xi_{\lambda}$, it follows easily from the commutation relations in the previous lemma that the $c_{k,m}$ must, by the vanishing condition for $x_1$, satisfy the relations
\begin{equation}\label{Eqckm1}
c_{k,m}= \epsilon\lambda^{-1} q^{-k-m+\frac{3}{2}}\frac{q^{-k-1}-q^{k+1}}{q^{-t+k+2m}-q^{t-k-2m}}c_{k+1,m},\qquad 0\leq k+2m <t.
\end{equation}
Writing out the vanishing under $x_{12}$ and using \eqref{Eqckm1}, we also obtain
\[
c_{k,m+1} = -|\epsilon|^2\lambda^{-2} q^{-t+3} \frac{(q^{-k-1}-q^{k+1})(q^{-t+k+2m+1}-q^{t-k-2m-1})}{(q^{-1}-q)(q^{-m-1}-q^{m+1})(1+|\epsilon|^2\lambda^{-2}q^{-2t+2m+5})}c_{k+1,m},\qquad 0\leq k+2m <t-1.
\]
From this, it is easily seen that any value of $c_{t,0}$ determines a unique solution. In the following, we fix a solution $\eta_{00}^{(t)}$ with $c_{t,0}=1$. 

Let now $V^{(t)}$ be the $\msB$-module spanned by $\eta_{00}^{(t)}$, and write $\eta_{mn}^{(t)} = (x_{12}^*)^m(x_2^*)^n\eta_{00}^{(t)}$. Then the invariant hermitian form on $M_{\lambda}$ must restrict to a scalar multiple, possibly negative or zero, of the unique normalized positive $\msB$-invariant inner product on $V^{(t)}$. Since the $V^{(t)}$ have different highest weights under $\msB$, it follows that the $\eta_{mn}^{(t)}$ must be mutually orthogonal for different indices. It is then sufficient to check that 
\[
\langle \eta_{00}^{(t)},\eta_{00}^{(t)}\rangle \geq0,\qquad \forall t\in \N.
\]

Now the vector $x_2\eta_{00}^{(t)}$ must again be a highest weight vector for $\msB$, if it is non-zero. Calculating the coefficient at $\xi_{t-1,0,0}$, we find by another tedious computation that 
\[
x_2 \eta_{00}^{(t)} = q^{-t+1}\frac{(q^{-t}-q^t)(1+|\epsilon|^2\lambda^{-2}q^{-t+4})}{(q^{-1}-q)^2(1+|\epsilon|^2\lambda^{-2}q^{-2t+5})}(1-|\epsilon|^2\lambda^{-2}q^{-t+4}) \eta_{00}^{(t-1)}.
\]
On the other hand, from weight considerations it follows that 
\[
x_2^*\eta_{00}^{(t)} = \eta_{00}^{(t+1)} +\xi_t,\qquad \xi_t \in \oplus_{k=0}^t V^{(k)}.
\]
Hence 
\begin{multline*}
\langle \eta_{00}^{(t)},\eta_{00}^{(t)}\rangle  = \langle \eta_{00}^{(t)},x_2^*\eta_{00}^{(t-1)}\rangle  = \langle x_2\eta_{00}^{(t)},\eta_{00}^{(t-1)}\rangle \\ = q^{-t+1}\frac{(q^{-t}-q^t)(1+|\epsilon|^2\lambda^{-2}q^{-t+4})}{(q^{-1}-q)^2(1+|\epsilon|^2\lambda^{-2}q^{-2t+5})}(1-|\epsilon|^2\lambda^{-2}q^{-t+4})\langle \eta_{00}^{(t-1)}, \eta_{00}^{(t-1)}\rangle,
\end{multline*}
and it follows that $\langle -,-\rangle$ is positive semi-definite if and only if $\lambda$ is of the form prescribed in the statement of the theorem.
\end{proof}

\end{document}